\documentclass[reqno,twoside]{amsart}
\usepackage[utf8]{inputenc}
\usepackage{amscd}
\usepackage{amsfonts}
\usepackage{amsmath}
\usepackage{amssymb}
\usepackage{amsthm}
\usepackage{amsaddr}
\usepackage{fancyhdr}
\usepackage{latexsym}
\usepackage[colorlinks=true, pdfstartview=FitV, linkcolor=blue, citecolor=blue, urlcolor=blue]{hyperref}
\usepackage{enumitem}      
\usepackage{mathtools}            
\usepackage{indentfirst} 
\usepackage{color}
\usepackage{caption}          
\usepackage[normalem]{ulem}
\usepackage{upgreek}
\usepackage{dutchcal}

\usepackage{tabu}
\usepackage{float}
\usepackage{graphicx}
\usepackage{accents}
\usepackage{listings}
\usepackage{epstopdf}
\usepackage{subcaption}
\usepackage{cancel}
\usepackage{tikz}

\usepackage{algorithm,algorithmic}

\usepackage{tikz}
\usetikzlibrary{decorations.markings}
\usetikzlibrary{decorations}

\usepackage{tikz}
\usetikzlibrary{arrows.meta,positioning,calc}

\tikzstyle{block} = [draw, fill=blue!20, rectangle, 
minimum height=3em, minimum width=6em]
\tikzstyle{sum} = [draw, fill=blue!20, circle, node distance=1cm]
\tikzstyle{input} = [coordinate]
\tikzstyle{output} = [coordinate]
\tikzstyle{pinstyle} = [pin edge={to-,thin,black}]

\usetikzlibrary{shapes.misc}

\tikzset{cross/.style={cross out, draw=black, minimum size=2*(#1-\pgflinewidth), inner sep=0pt, outer sep=0pt},
	cross/.default={1pt}}

\DeclareCaptionType[placement={!ht}]{listing}[Listing][Code Listings]

\usetikzlibrary{shapes.misc}

\tikzset{cross/.style={cross out, draw=black, minimum size=2*(#1-\pgflinewidth), inner sep=0pt, outer sep=0pt},
	cross/.default={1pt}}
\newcommand{\nn}{\nonumber}
\newcommand{\p}{\partial}

\newcommand{\no}[1]{\left\| #1 \right\|}
\newcommand{\what}{\widehat}
\usepackage{thmtools}
\declaretheoremstyle[headfont=\normalfont\bfseries, bodyfont=\itshape, spaceabove=7pt, spacebelow=7pt]{theorem} 
\theoremstyle{theorem} 
\newtheorem{theorem}{Theorem}[section] 
\newtheorem{remark}{Remark}[section]
\newtheorem{proposition}{Proposition}[section]
\newtheorem{lemma}{Lemma}[section]

\theoremstyle{definition}
\newtheorem{definition}{Definition}[section]
\allowdisplaybreaks
\numberwithin{equation}{section}
\numberwithin{figure}{section}
\usepackage{geometry}

\usepackage[breakable, theorems, skins]{tcolorbox}
\tcbset{enhanced}

\let\OLDthebibliography\thebibliography
\renewcommand\thebibliography[1]{
  \OLDthebibliography{#1}
  \setlength{\parskip}{0pt}
  \setlength{\itemsep}{0pt plus 0.3ex}
}
\AtBeginDocument{
   \def\MR#1{}
}

\makeatletter
\@namedef{subjclassname@2020}{%
  \textup{2020} Mathematics Subject Classification}
\makeatother

\begin{document}

\title{The complex Ginzburg-Landau equation on a finite interval and chaos suppression via a finite-dimensional boundary feedback stabilizer}

\author{Dionyssios Mantzavinos$^\dagger$,  Türker Özsarı$^*$, Kemal Cem Yılmaz$^+$}

\address{\normalfont $^\dagger$Department of Mathematics, University of Kansas, Lawrence, KS, USA\\
\normalfont $^*$Department of Mathematics, Bilkent University, Çankaya, Ankara, Turkey\\
\normalfont $^+$Department of Mathematics, Izmir Institute of Technology, Urla, Izmir, Turkey} 
\email{\!turker.ozsari@bilkent.edu.tr (Corresponding author)}

\thanks{\textit{Acknowledgements.} DM's research is partially supported by  the U.S. National Science Foundation (NSF-DMS 2206270 and NSF-DMS 2509146) and the Simons Foundation (SFI-MPS-TSM-00013970). TÖ and KCY's research were supported by TÜBİTAK 1001 Grant 122F084.}
%
%
\date{September 30, 2025}

\begin{abstract}
In this paper, we study the well-posedness and boundary stabilization of the initial-boundary value problem  for the complex Ginzburg-Landau (CGL)  equation on a finite interval.  First, we establish a local well-posedness theory for the open-loop model in $L^2$-based fractional Sobolev spaces in the case of Dirichlet-Neumann type inhomogeneous mixed boundary conditions. This local well-posedness result is based on linear estimates  derived by using the weak solution formula obtained via the unified transform (also known as the Fokas method). Next, we study the global well-posedness properties of the open-loop model in presence of inhomogeneous boundary conditions. Then, we turn our attention to the rapid boundary feedback stabilization problem and design a nonlocal controller which uses a finite number of Fourier modes of the state of solution. This design relies on the fact that solutions of the CGL equation can be separated into a slow, finite-dimensional component and a rapidly decaying tail, with the former primarily governing long-term behavior.  We determine the necessary number of modes required to stabilize the system at a specified rate. Additionally, we identify the minimum number of modes that ensure stabilization at an unspecified decay rate. These theoretical results are validated by numerical simulations. The spatiotemporal estimates established in the first part of the paper are also employed to obtain local solutions of the controlled system. The existence of global energy solutions follows from stabilization estimates, while uniqueness follows from the uniqueness of an associated initial-boundary value problem with homogeneous boundary conditions whose solutions are in correspondence with the solutions of the original system through a bounded invertible Volterra-type integral transform on Sobolev spaces.   \vspace{.1in}

\end{abstract}

\maketitle
\markboth
{D. Mantzavinos, T. Özsarı and K. C. Yılmaz}
{The complex Ginzburg-Landau equation on a finite interval and chaos suppression}

{\hypersetup{linkcolor=black}
\tableofcontents}

\section{Introduction}
In this paper, we pursue two interrelated goals. The first objective is to study the well-posedness of the following Dirichlet-Neumann initial-boundary value problem for the complex Ginzburg-Landau (CGL) equation on a finite interval in  fractional-order Sobolev spaces:
\begin{equation}\label{cgl-ibvp}
\begin{aligned}
&u_t - \left(\nu + i \alpha\right) u_{xx} - \gamma u + \left(\kappa + i\beta\right) |u|^p u = 0, \quad (x, t) \in (0, L) \times (0, T),
\\
&u(x,0) = u_0(x),
\\
&u(0, t) = a(t), \quad u_x(L, t) = b(t),
\end{aligned} 
\end{equation}
where $\nu, \kappa, p > 0$, $\alpha, \beta \in \mathbb R$, $\gamma \geq 0$, and the initial and boundary data $\{u_0(x), a(t), b(t)\}$ belong in appropriate fractional-order Sobolev spaces that will be discussed later.  

The second objective is to design a rapid boundary feedback stabilization scheme for the CGL equation in~\eqref{cgl-ibvp} by replacing the Neumann input acting at the right endpoint of the domain with a feedback controller which involves an integral-type nonlocal term and uses only a \textit{finite} number of Fourier modes of the state of solution. More precisely, for that second objective, we take $a(t)\equiv 0$ and construct a feedback law at the right endpoint of the domain in the form of
\begin{equation}\label{boundaryfeedbackform}
u_x(L,t)=b(t):=\int_0^L \xi(y) \, \Gamma [P_N u](y,t) dy + \zeta \, \Gamma [P_N u](x,t) \big|_{x = L}
\end{equation}
where $\zeta\in\mathbb{C}$ is a suitable constant, $\Gamma:H^m(0,L)\rightarrow H^m(0,L)$ is a particular linear bounded operator, and $\xi\in C^\infty([0,L])$ is to be constructed as a special function, with the explicit properties of $\Gamma$ and $\xi$ being intrinsic to the linear part of the CGL equation and to be discussed later.{Here, the Sobolev space $H^m(0, L)$ is the restriction on $(0, L)\subset \mathbb R$ of the Sobolev space $H^m(\mathbb R)$ defined by
\begin{equation}\label{sob-ft}
H^m(\mathbb R) := \left\{f \in L^2(\mathbb R): \left(1+k^2\right)^{\frac m2} \mathcal F\{f\}(k) \in L^2(\mathbb R)\right\},
\quad
m\geq 0,
\end{equation}
with norm $\no{f}_{H^m(\mathbb R)} := \big\| \left(1+k^2\right)^{\frac m2} \mathcal F\{f\}(k) \big\|_{L^2(\mathbb R)}$, 
where $\mathcal F\{f\}$ is the Fourier transform of $f$ over $\mathbb R$ defined by
\begin{equation}
\mathcal F\{f\}(k) := \int_{\mathbb R} e^{-ikx} f(x) dx, \quad k \in \mathbb R.
\end{equation}
%
}
Moreover, $P_N$ is the projection operator defined by
\begin{equation}
		P_N u = \sum_{j=1}^N \left(e_j(\cdot),u(\cdot,t)\right)_2e_j(x) , \quad e_j(x) = \sqrt{\frac{2}{L}} \sin \left(\frac{(2j - 1)\pi x}{2L}\right),
	\end{equation}
	where each of the functions $e_j$ is an eigenfunction of the Laplacian subject to Dirichlet-Neumann boundary conditions, i.e.
	\begin{equation} \label{sturm-liouville}
		\begin{aligned}
			-\varphi_{xx} &= \lambda \varphi_, \quad x \in (0,L), \\
			\varphi(0) &= \varphi^\prime(L) = 0,
		\end{aligned}
	\end{equation}
	with respective eigenvalue 
	$\displaystyle 
		\lambda=\lambda_j = \left(\frac{2j - 1}{2}\right)^2 \frac{\pi^2}{L^2}.
	$

It should be noted that a boundary feedback of the form \eqref{boundaryfeedbackform} is not an artificial choice; rather, it corresponds to a truncated version of the well-known backstepping-type controller commonly used for the CGL equation on a finite interval with mixed boundary conditions, see e.g., \cite{krsticbook}. In particular, when 
$N=\infty$, it reduces to the classical backstepping controller.

The CGL equation provides a mathematical framework for modeling instability waves in systems approaching critical transitions, such as reaction-diffusion processes near a Hopf bifurcation. Note that if $\nu \ll 1$ (and, in particular, if one sets $\nu=0$) then the CGL equation reduces to the nonlinear Schr\"odinger (NLS) equation, while if $\nu \gg 1$ then the CGL equation becomes a reaction-diffusion model.  The CGL equation plays a crucial role in describing various nonlinear phenomena, including wave dynamics, phase transitions, superconductivity, superfluidity, Bose–Einstein condensates, and liquid crystals. More recent applications include doped optical fibers, two-component open Bose--Einstein condensates and systems of self-propelled particles see, e.g., \cite{ZangaFewoTabiKofane2022,OtlaadisaTabiKofane2021,EdoumaBiloaTabiEkobenaFoudaKofane2023} and references therein. For a broader discussion on the diverse applications of this equation, we refer the reader to the article \cite{Ar02} and related literature.

Concerning the well-posedness of initial-boundary value problems with nonzero boundary data for the CGL equation, there are some studies in the literature such as  \cite{GaBu04} and the references therein for results at the $H_x^1$-level with smooth boundary data.  In addition, \cite{LiZh18} is concerned with the CGL equation in negative-indexed fractional Sobolev spaces. There are also some results pertaining to the stochastic CGL equation with white noise as boundary data, see for example \cite{BaKa20}. In the case of homogeneous boundary conditions, more works are available, see for instance \cite{ShYoYo16} and the references therein for a review of such results. 

Our approach for treating the local well-posedness associated with the inhomogeneous initial-boundary value problem \eqref{cgl-ibvp} will be based on the analysis of the weak solution formulation obtained via the unified transform, also known as the Fokas method \cite{Fokas97,Fokasbook}.  Since the introduction of this approach in  \cite{fhm2017} for the NLS equation on the half-line with nonzero Dirichlet data, there has been substantial progress towards the rigorous analysis of solutions of nonlinear initial-boundary value problems, see for instance \cite{hmy2019-nls, OY19, hm2020, KO22, hm2022, AMO24, mmo2024,mo2024} for Schrödinger-type equations and \cite{hmy2019-rd} for a reaction-diffusion equation. The successful implementation of the approach lies in the fact that the unified transform  provides the direct analogue of the Fourier transform when the spatial domain involves a boundary. For further reading on the solution of linear initial-boundary value problems via the unified transform, we refer the reader to the review articles \cite{DTV14,fs2012,cfk2025} and, concerning the case of the finite interval in particular, to the works \cite{p2004, fp2005, s2012, fs2016}.

It is widely recognized that when the Benjamin-Feir-Newell stability condition $\frac{\alpha}{\nu}\frac{\beta}{\kappa}>-1$ is violated, the CGL equation may exhibit instability, leading to the emergence of chaotic behavior.  Such behavior in evolution equations is usually suppressed using interior or  boundary controls, see \cite{K94}, \cite{LT00}, \cite{LT00-2}, and \cite{Zua07} for  general approaches on control/stabilization of evolution pdes.  Some of the studies on the stabilization of unstable solutions of the Ginzburg–Landau equation involve internal feedback controls (global and/or local) \cite{BeMi04, BlSo96, BoBra06, GoNe06, MoSi04, PoSi07, StBe10, StCa07, BaMi96, RuNeGo07}, feedback controls using backstepping \cite{AaKr05, AaKr07, AaKr04} or dynamic boundary conditions \cite{CoOz18}, and non-feedback (open-loop) controls achieving null-controllability \cite{RoZh09}.  

The CGL equation has the important feature that it exhibits a finite-dimensional asymptotic behavior over time \cite{GhHe87}. The second author  of the present work has leveraged this property to extend finite-parameter interior feedback control algorithms developed for the reaction-diffusion equation \cite{AzTi14} to the framework of the CGL equation. These algorithms utilize a finite number of volume elements, Fourier modes, or nodal observables (controllers) to achieve exponential stabilization of solutions (see \cite{KaaOz17} for more details). Most recently, the last two authors of this work developed a finite-dimensional version of the classical backstepping controller design for stabilizing solutions of reaction-diffusion systems from the boundary \cite{KaOzY} with Dirichlet actuation.  The distinctive feature of this control design is that the boundary controller uses only a finite number of Fourier modes of
the state of solution, as opposed to the classical backstepping controller which uses all (infinitely many) modes.  It should be noted that there is also another nice approach that appeared in the framework of the heat equation, which is called the late lumpting technique \cite{Aur19, Woit17} and it relies on the same essential idea that solutions can be split into a slow finite dimensional part and a fast tail, where the former dominates the dynamics in the long run. See \cite{KaOzY} for a through comparison of our approach with the late lumping. A goal of the present paper is to extend this finite-dimensional boundary control design to the setting of the CGL equation with Neumann actuation.  Our primary focus is on the case where $\gamma>0$, as solutions with homogeneous boundary conditions naturally diminish to zero otherwise, eliminating the possibility of instability.
\\[2mm]
\textbf{Main results.}
Our local well-posedness result for the open-loop inhomogeneous nonlinear CGL Dirichlet-Neumann boundary value problem \eqref{cgl-ibvp} can be stated as follows.
\begin{theorem}[Hadamard well-posedness]\label{lwp-t}
Let $s \in \left(\frac 12, \frac 32\right) \cup \left(\frac 32, \frac 52\right)$ and $p>0$ such that either $p \in 2\mathbb N$ or, otherwise,
\begin{equation}\label{p-def}
\begin{aligned}
&\text{if $s\in\mathbb N$, then either $p\geq s$ if $p\in 2\mathbb N-1$, or $\left\lfloor p\right\rfloor\geq s-1$ if $p\notin \mathbb N$},
\\
&\text{if $s\not\in \mathbb N$, then either $p>s$ if $p \in 2\mathbb N-1$, or $\left\lfloor p\right\rfloor\geq \left\lfloor s \right\rfloor$ if $p\notin \mathbb N$}.
\end{aligned}
\end{equation}
Furthermore, let $T>0$ satisfy the conditions \eqref{T-into}, \eqref{T-into-2}, \eqref{T-contr} and \eqref{T-contr-2}. Then, for initial data $u_0 \in H^s(0, L)$ and boundary data $a \in H^{\frac{2s+1}{4}}(0, T)$, $b \in H^{\frac{2s-1}{4}}(0, T)$ satisfying the compatibility conditions
\begin{equation}\label{comp-cond}
u_0(0) = a(0), \ s \in \left(\tfrac 12, \tfrac 32\right) \cup \left(\tfrac 32, \tfrac 52\right),
\quad
u_0'(L) = b(0), \ s \in \left(\tfrac 32, \tfrac 52\right),
\end{equation}
the Dirichlet-Neumann initial-boundary value problem~\eqref{cgl-ibvp} for the complex Ginzburg-Landau equation on the finite interval $(0, L)$ is well-posed in the sense of Hadamard, namely it possesses a solution 
$$
u \in C_t([0, T]; H_x^s(0, L)) \cap C_x([0, L]; H_t^{\frac{2s+1}{4}}(0, T))
$$ 
which is unique within the larger space $C_t([0, T]; H_x^s(0, L))$ and depends continuously on the initial and boundary data, with the relevant size estimate
\begin{equation}
\begin{aligned}
&\quad
\sup_{t\in [0, T]} \no{u(t)}_{H_x^s(0, L)} 
+
\sup_{x\in [0, L]} \no{u(x)}_{H_t^{\frac{2s+1}{4}}(0, T)} 
\leq 
2\mathcal B \, 
\Big( 
\no{u_0}_{H^s(0, L)} 
+
\no{a}_{H^{\frac{2s+1}{4}}(0, T)} 
+
\no{b}_{H^{\frac{2s-1}{4}}(0, T)}
\Big)
\end{aligned}
\end{equation}
where the constant $\mathcal B=\mathcal B(s, \lambda, L, \alpha, \nu, \gamma, T)$ is given by \eqref{cr-def}.
\end{theorem}

Beyond the local well-posedness result of Theorem \ref{lwp-t}, which is the one that is essential for the design of our boundary feedback stabilization scheme for the CGL equation, we also prove global well-posedness of energy solutions for certain values of $p$ depending on the sign of $\beta$. We emphasize that this latter result is not required for the closed-loop problem relevant to the boundary feedback stabilization, yet it enhances our analysis of the open-loop problem \eqref{cgl-ibvp}.

\begin{theorem}[Global well-posedness]\label{gwp-t}
    Let $u_0\in H^1(0,L)$ and $a,b\in H^1_{\textnormal{loc}}(\mathbb{R}_+)$ with the compatibility condition $u_0(0)=a(0)$. Suppose that $p\in (0,2]$ if $\beta\ge 0$ and $p\in (0,\frac{4}{5})$ if $\beta<0$.  If $u \in C_t([0, T]; H_x^1(0, L))$ is a local solution of the Dirichlet-Neumann initial-boundary value problem \eqref{cgl-ibvp} for the CGL equation, then $u$ is also a global-in-time solution.   
\end{theorem}

Theorems \ref{lwp-t} and \ref{gwp-t} summarize our results in the direction of the first objective of this paper, namely that of well-posedness of the open-loop problem \eqref{cgl-ibvp}. Regarding our second objective, namely finite-dimensional boundary feedback stabilization, the spectral structure of the differential operator $A : D(A) \to X$
    \begin{equation*}
        A := (\nu + i \alpha) \partial_x^2 + \gamma I
    \end{equation*}
    with domain $D(A) = \{\varphi \in H^2(0,L) \, | \, \varphi(0) = \varphi^\prime(L) = 0 \}$ and range $X = L^2(0,L)$ plays a central role in our stabilization strategy. Since the resolvent of $A$ is compact, clearly its spectrum is discrete. In fact, the discrete spectral structure of the operator $A$ provides a natural framework for analyzing the dynamics in terms of modal components. The solution of the uncontrolled homogeneous system admits a Fourier mode expansion, which allows us to analyze the long-time behavior of the system through its modal components. Moreover, for the given problem parameters, the operator $A$ has at most finitely many eigenvalues
    \begin{equation*}
        \tilde \lambda_j = -(\nu + i \alpha) \lambda_j + \gamma, \quad j \in \mathbb{Z}^+
    \end{equation*}
    with nonnegative real part. Therefore, the possible instability of the system is essentially generated by only those finitely many modal components, namely the ones associated with eigenvalues having positive real part. Hence, it is natural to seek a stabilization mechanism based on a control input acting only on finitely many Fourier modes.

    In this framework, the stabilization problem is twofold. First, we address the rapid stabilization problem, which concerns the design of a feedback control law ensuring stability with a prescribed exponential decay rate. In this sense, prescribing a larger decay rate can be interpreted as forcing the solution to enter a prescribed neighborhood of the zero equilibrium within a shorter time interval. Second, we study the minimal number of Fourier modes required in the control law to stabilize the system. Regarding this problem, let us first define the notion of unstable Fourier modes as follows: Suppose that the eigenvalues of the operator $A$ are ordered according to their real parts, with the first $M$ ones having nonnegative real part:
    \begin{equation*}
        \text{Re} \tilde \lambda_1 > \dotsm > \text{Re} \tilde \lambda_M \geq 0, \quad \text{Re} \tilde \lambda_{M+1} < 0
    \end{equation*}
    which reduces to
    \begin{equation*}
        \lambda_M \leq \frac{\nu}{\gamma} < \lambda_{M+1}.
    \end{equation*}
    We refer to the Fourier modes corresponding to these $M$ eigenvalues as the unstable Fourier modes, in the sense that they do not decay asymptotically to the zero equilibrium. Then, it is natural to ask whether stabilization can be achieved by a control input involving precisely $M$ Fourier modes, namely the same number as the unstable Fourier modes of the uncontrolled system. We show that this is indeed possible, so that our result provides a sharp characterization of the required controller dimension in terms of the unstable Fourier modes of the system.

Motivated by the discussion above, we first establish the well-posedness and stabilization theory for the associated linear dynamics, i.e. the closed-loop problem \eqref{pde_lin}, via the following theorem.

\begin{theorem}[Linear stabilization] \label{thm-stab-linear}
	Let $\nu > 0$, $\alpha \in \mathbb{R}$, $\gamma \geq 0$, $u_0 \in H^1(0,L)$ and $k$ be a solution of the boundary value problem \eqref{kernel_pde}.
	\begin{enumerate}[label=\textnormal{(\roman*)}, itemsep=2mm, topsep=2mm, leftmargin=7mm]
        \item  \textnormal{(Rapid stabilization)} Suppose $\mu > \gamma - \nu \lambda_1$. If $N$ satisfies
		\begin{equation}
			N > \max \left\{\frac{\mu}{4 \nu \lambda_1} - \frac{1}{2},\frac{\mu}{2(\mu + \nu \lambda_1 - \gamma)} - \frac{1}{2} \right\}
		\end{equation}
		and $(\mu,N)$ is chosen in the sense of Definition \ref{ratemodepair}, then the zero equilibrium of the closed-loop problem~\eqref{pde_lin} subject to the feedback control law  $u_x(L,t) = b(P_Nu)$ given by \eqref{controller} is globally exponentially stable in $H^1(0,L)$. Moreover, the weak solution of \eqref{pde_lin} satisfies the exponential decay estimate
		\begin{equation}
			\no{u(t)}_{H_x^1(0,L)} \leq c_k e^{-\eta_1 t} \no{u_0}_{H^1(0,L)}, \quad \eta_1 = \nu\lambda_1 - \gamma + \mu \left(1 - \frac{1}{2N+1}\right), \quad \text{a.e. }t > 0,
		\end{equation}
		where the constant $c_k>0$ depends on $k$ but is independent of the initial datum.
		
		\item  \textnormal{(Minimal Fourier mode count)} Let $\nu, \alpha, L$ be such that $\lambda_M \leq \frac{\nu}{\gamma} < \lambda_{M+1}$. If $\mu$ satisfies
		\begin{equation}
			2(\gamma - \nu \lambda_1) \left(1 - \frac{1}{2M+1}\right)^{-1} < \mu < 2\nu \lambda_{M+1}
		\end{equation} 
		and $(\mu,M)$ is chosen in the sense of Definition \ref{ratemodepair}, then the zero equilibrium of the closed-loop problem~\eqref{pde_lin} subject to the feedback control law $u_x(L,t) = b(P_M u(\cdot,t))$ given by \eqref{controller} is globally exponentially stable in $H^1(0,L)$. Moreover, the weak solution of \eqref{pde_lin} satisfies the exponential decay estimate
		\begin{equation}
			\no{u(t)}_{H_x^1(0,L)} \leq e^{-\eta_2 t} \no{u_0}_{H^1(0,L)}, \quad \eta_2 = \nu\lambda_1 - \gamma + \frac{\mu}{2} \left(1 - \frac{1}{(2M+1)^2}\right), \quad \text{a.e. }t > 0,
		\end{equation}
		where the constant $c_k>0$  depends on $k$ but is independent of the initial datum.
	\end{enumerate}
\end{theorem}

Finally, we prove the following well-posedness and stabilization theorem for the fully nonlinear problem with sufficiently small initial data.

\begin{theorem}[Nonlinear stabilization] \label{thm-stab-nonlinear}
	Let $0<p<4$, $\nu, \kappa > 0$, $\alpha, \beta \in \mathbb{R}$, $u_0 \in H^1(0,L)$,  $\eta^* \in \{\eta_1, \eta_2\}$. Suppose that
	\begin{equation}
		\no{u_0}_{L^2(0,L)} \leq \min \left\{\left(\frac{2\eta^* - \varepsilon \lambda_1}{c_{1,\varepsilon} + c_2}\right)^{\frac{4-p}{3p + 4}},\left(\frac{2\eta^* - \varepsilon \lambda_1}{c_{1,\varepsilon} + c_2}\right)^{\frac{1}{p+1}}\right\}
	\end{equation}    
	where $\varepsilon > 0$ is such that $2\eta^* - \varepsilon \lambda_1 > 0$, $c_2$ and $c_{1,\varepsilon}$ are respectively given by \eqref{norm-constants} and \eqref{c1star}, and
	\begin{equation}
		\|u_0^\prime\|_{L^2(0,L)} \leq \min \left\{\left(\frac{2\eta^* - \varepsilon \lambda_1}{ c_{3,\varepsilon,\lambda_1} +  c_{4,\varepsilon,\lambda_1}}\right)^{\frac{4 - p}{7p+4}},\left(\frac{2\eta^* - \varepsilon \lambda_1}{ c_{3,\varepsilon,\lambda_1} +  c_{4,\varepsilon,\lambda_1}}\right)^\frac{1}{2p+1}\right\}
	\end{equation}
	where $c_{3,\varepsilon,\lambda_1}$ and $c_{4,\varepsilon,\lambda_1}$ are given by \eqref{c3c4}. Let $k$ be a solution of the boundary value problem \eqref{kernel_pde}.
	\begin{enumerate}[label=\textnormal{(\roman*)}, itemsep=2mm, topsep=2mm, leftmargin=7mm]
    \item \textnormal{(Rapid stabilization)} Let $\eta^* = \eta_1$. Suppose $\mu > \gamma - \nu \lambda_1$. If $N$ satisfies
		\begin{equation} \label{N_condd}
			N > \max \left\{\frac{\mu}{4 \nu \lambda_1} - \frac{1}{2},\frac{\mu}{2(\mu + \nu \lambda_1 - \gamma)} - \frac{1}{2} \right\}
		\end{equation}
		and $(\mu,N)$ is chosen in the sense of Definition \ref{ratemodepair}, then the zero equilibrium of the closed-loop system \eqref{pde_lin} subject to the feedback control law $u_x(L,t) = b(P_N u)$ given by \eqref{controller} is locally exponentially stable in $H^1(0,L)$. Moreover, the weak solution, $u$ of \eqref{pde_lin} satisfies the decay estimate
		\begin{equation}
			\no{u(t)}_{H_x^1(0,L)} \lesssim c_k e^{-\left(\eta_1 - \frac{\varepsilon \lambda_1}{2}\right) t} \no{u_0}_{H^1(0,L)}, \quad \eta_1 = \nu\lambda_1 - \gamma + \mu \left(1 - \frac{1}{2N+1}\right), \quad \text{a.e. }t > 0,
		\end{equation}
		where $c_k$ is a nonnegative constant that depends on $k$ and independent of the initial datum.
		
		\item \textnormal{(Minimal Fourier mode count)} Let $\eta^* = \eta_2$ and $\nu, \alpha, L$ be such that $\lambda_M \leq \frac{\nu}{\gamma} < \lambda_{M+1}$. If $\mu$ satisfies
		\begin{equation}
			2(\gamma - \nu \lambda_1) \left(1 - \frac{1}{(2M+1)}\right)^{-1}< \mu < 2\nu \lambda_{M+1}
		\end{equation} 
		and $(\mu,M)$ is chosen in the sense of Definition \ref{ratemodepair}, then zero equilibrium of the closed-loop system \eqref{pde_lin} subject to the feedback control law $u_x(L,t) = b(P_M u(\cdot,t))$ given by \eqref{controller} is locally exponentially stable in $H^1(0,L)$. Moreover, the weak solution, $u$ of \eqref{pde_lin} satisfies the decay estimate
		\begin{equation}
			\no{u(t)}_{H_x^1(0,L)} \lesssim e^{-\left(\eta_2 - \frac{\varepsilon \lambda_1}{2}\right) t} \no{u_0}_{H^1(0,L)}, \quad \eta_2 = \nu\lambda_1 - \gamma + \frac{\mu}{2} \left(1 - \frac{1}{(2M+1)^2}\right), \quad \text{a.e. }t > 0,
		\end{equation}
		where $c_k$ is a nonnegative constant that depends on $k$ and independent of the initial datum.
	\end{enumerate}
\end{theorem}

\noindent
\textbf{Structure.} In Section \ref{wp-s}, we prove Theorem \ref{lwp-t} on the local well-posedness of the open-loop problem \eqref{cgl-ibvp}. Furthermore, at the end of that section, we establish the global well-posedness result of Theorem \ref{gwp-t}. In Section~\ref{stab-s}, we develop a finite-dimensional controller for the CGL equation and prove the closed-loop well-posedness and stabilization results of Theorems \ref{thm-stab-linear} and \ref{thm-stab-nonlinear}. Numerical simulations for the control problem are presented  in Section \ref{numerical}. The solution formula for the forced linear CGL equation on a finite interval via the unified transform is derived in Appendix \ref{utm-sol-s}. In Appendix \ref{app-kernel}, we derive a boundary value problem whose solution is used as the kernel of the boundary feedback stabilizer to be constructed.

\section{Well-posedness of the open-loop problem}
\label{wp-s}

In this section, we establish Theorem \ref{lwp-t} on the local well-posedness of the open-loop problem \eqref{cgl-ibvp} for the CGL equation on a finite interval with inhomogeneous Dirichlet-Neumann data. We then proceed to the proof of the global well-posedness result of Theorem \ref{gwp-t}. 

Theorem \ref{lwp-t} will proved via a contraction mapping argument which will crucially rely on the estimation of the forced linear counterpart of the initial-boundary value problem \eqref{cgl-ibvp}, namely of the problem
\begin{equation}\label{lcgl-ibvp}
\begin{aligned}
&u_t - \left(\nu + i \alpha \right) u_{xx} - \gamma u = \mathfrak f(x, t), \quad (x, t) \in (0, L) \times (0, T),
\\
&u(x,0) = u_0(x),
\\
&u(0, t) = a(t), \quad u_x(L, t) = b(t),
\end{aligned} 
\end{equation}
where $\mathfrak f$ is some given forcing in an appropriate function space to be discussed below. 
With this plan in mind, our first task is to derive an explicit solution formula for the linear problem \eqref{lcgl-ibvp}. 
As noted in the introduction, for this purpose we will employ the unified transform, also known as the Fokas method, which provides the direct analogue of the Fourier transform in the case of initial-boundary value problems. Then, using our linear solution formula, we will establish various linear estimates which we will then combine with a contraction mapping argument in order to prove Theorem \ref{lwp-t}.

In fact, we observe that the term $\gamma u$ can be removed from the linear problem \eqref{lcgl-ibvp} via the transformation 
\begin{equation}\label{gamma-trans}
u(x, t) = e^{\gamma t} w(x, t),
\end{equation}
which turns problem \eqref{lcgl-ibvp} into  
\begin{equation}\label{lcgl-ibvp-w}
\begin{aligned}
&w_t - \left(\nu + i \alpha \right) w_{xx} = e^{-\gamma t} \mathfrak f(x, t) =: f(x, t), \quad (x, t) \in (0, L) \times (0, T),
\\
&w(x,0) = u_0(x),
\\
&w(0,t) = e^{-\gamma t} a(t) =: g_0(t), \quad w_x(L, t) = e^{-\gamma t} b(t) =: h_1(t).
\end{aligned} 
\end{equation}
%
%
By employing the unified transform, we derive in Appendix \ref{utm-sol-s} the following explicit solution formula for the modified problem \eqref{lcgl-ibvp-w}:
\begin{align}\label{lcgl-sol-T}
w(x, t) 
&=
\frac{1}{2\pi} \int_{\mathbb R} e^{ikx-\omega t} \, \what u_0(k) \, dk 
+ \frac{1}{2\pi} \int_{\mathbb R} e^{ikx-\omega t} 
\int_0^t e^{\omega  t'} \what f(k, t') dt' dk
\nn\\
&\quad
- \frac{1}{2\pi} 
\int_{\p D^+}  
\frac{e^{ikx-\omega t}}{e^{ikL} + e^{-ikL}} 
\left\{
 \left[ e^{ikL} \what u_0(k) + e^{-ikL} \what u_0(-k) \right]
+
 \int_0^t e^{\omega  t'} \left[e^{ikL} \what f(k, t') + e^{-ikL} \what f(-k, t') \right] dt'  
\right\} dk 
\nn\\
&\quad
- \frac{\nu+i\alpha}{\pi} 
\int_{\p D^+}  
\frac{e^{ikx-\omega t}}{e^{ikL} + e^{-ikL}} 
\left[
\widetilde h_1(\omega, T)
+
 ik e^{-ikL} \, \widetilde g_0(\omega, T)
\right] dk 
\\
&\quad
+ \frac{1}{2\pi} 
\int_{\p D^-}    
\frac{e^{ik(x-L)-\omega t}}{e^{ikL} + e^{-ikL}} 
\left\{
\big[ \what u_0(k) - \what u_0(-k) \big]
+
\int_0^t e^{\omega t'} \left[\what f(k, t') - \what f(-k, t') \right] dt' \right\} dk
\nn\\
&\quad
- \frac{\nu+i\alpha}{\pi} 
\int_{\p D^-}   
\frac{e^{ik(x-L)-\omega t}}{e^{ikL} + e^{-ikL}} 
\left[
e^{ikL} \, \widetilde h_1(\omega, T)
+
ik \widetilde g_0(\omega, T) 
\right] dk,
 \nn
\end{align}
where $\what u_0(k)$ and $\what f(k, t')$ denote the finite interval Fourier transforms of $u_0(x)$ and $f(x, t')$ with respect to $x$ as defined by \eqref{ft}, the temporal transforms $\widetilde g_0(\omega, T)$ and $\widetilde h_1(\omega, T)$ are defined by \eqref{tilde-def}, and the contours $\p D^\pm$ are the positively oriented boundaries of the regions $D^\pm$ depicted in Fig. \ref{dpm-f} and are described by the pair of straight lines
\begin{equation}\label{hyp-eq}
\text{Im}(k) = \lambda \, \text{Re}(k), \quad \text{Im}(k) = -\frac{1}{\lambda} \, \text{Re}(k),
\quad \lambda := \frac{-\alpha  + \sqrt{\alpha^2 + \nu^2}}{\nu}.
\end{equation}

\begin{figure}[ht!]
\centering
\begin{tikzpicture}[scale=1]
\draw[line width=0.5pt, black](2.15,2.5)--(2.15,2.15);
\draw[line width=0.5pt, black](2.15,2.15)--(2.5,2.15);
\node[] at (2.27, 2.36) {\fontsize{8}{8} $k$};
\draw[line width=0.5pt, black, ->](-2.5,0)--(2.5,0);
\draw[line width=0.5pt, black, ->, dotted](0,-2.5)--(0,2.5);
\node[] at (-0.18, -0.24) {\fontsize{9}{9} $0$};
\draw[domain=-2:2, variable=\x, smooth, line width=0.8pt, blue] plot ({\x}, {0.5*\x});
\draw[domain=-0.9:0.9, variable=\x, smooth, line width=0.8pt, red] plot ({\x}, {-2*\x});
\node[] at (0.5, 1.1) {\fontsize{10}{10} $D^+$};
\node[] at (-0.4, -1.1) {\fontsize{10}{10} $D^-$};
\end{tikzpicture}
\caption{
The straight lines \eqref{hyp-eq} that define the region   $D^+ \cup D^-$.
}
\label{dpm-f}
\end{figure}

By means of the transformation \eqref{gamma-trans},  formula \eqref{lcgl-sol-T} yields a corresponding solution formula for the original forced linear CGL problem \eqref{lcgl-ibvp}. However, the linear estimates that will be used in the proof of the local well-posedness result of Theorem \ref{lwp-t} will actually be derived at the level of the modified problem \eqref{lcgl-ibvp-w} via the solution formula \eqref{lcgl-sol-T}.
\\[2mm]
\textbf{Outline of linear analysis.} 
First, we will estimate the solution to a specific version of the linear problem~\eqref{lcgl-ibvp-w} known as the \textit{reduced} initial-boundary value problem, which involves zero initial data, zero forcing, and compactly supported boundary data. This process leads to the space and time estimates of Theorems \ref{pure-ibvp-t} and \ref{pure-te-t}, and reveals the natural choice of function spaces for the boundary data. This choice of boundary data spaces will then be confirmed via the derivation of suitable estimates for the homogeneous as well as the forced linear CGL initial value (Cauchy) problems (Theorems \ref{U-ivp-t} and \ref{W-ivp-t}). These individual results will be combined with carefully constructed extensions involving the boundary data and the traces of the aforementioned Cauchy problems in order to infer, via the superposition principle, 
the linear estimates \eqref{u-se} and \eqref{u-se-2} for the full linear problem \eqref{lcgl-ibvp-w}. These linear estimates will provide the basis for a contraction mapping argument for the proof of the local well-posedness result of Theorem \ref{lwp-t} for the nonlinear problem \eqref{cgl-ibvp}.

\subsection{Estimates for the reduced  initial-boundary value problem}  
\label{pure-s}
We now estimate the solution to the following fundamental problem, which has zero forcing, zero initial data, and boundary data with compact support, and which is hereafter referred to as the \textit{reduced initial-boundary value problem}:
\begin{equation}\label{pure-ibvp}
\begin{aligned}
&v_t - \left(\nu + i \alpha \right) v_{xx} = 0, \quad (x, t) \in (0, L) \times (0, T'), \ T' > T, 
\\
&v(x,0) = 0,
\\
&v(0,t) = g(t), \quad v_x(L,t) = h(t), \quad \text{supp}(g), \, \text{supp}(h) \subset (0, T'). 
\end{aligned} 
\end{equation}
It is natural to ask which choice of function spaces for the boundary data $g(t), h(t)$ ensures that the solution $v(x, t)$ belongs to the Sobolev space $H^s(0, L)$ as a function of $x$. This question is addressed via the following~theorem.
\begin{theorem}[Space estimate]
\label{pure-ibvp-t}
Let $s \geq \frac 12$. Then, the solution $v(x, t)$ to the reduced initial-boundary value problem~\eqref{pure-ibvp} with boundary data $g\in H^{\frac{2s+1}{4}}(\mathbb R), h\in H^{\frac{2s-1}{4}}(\mathbb R)$ that have support inside $(0, T')$ belongs to the Sobolev space $H^s(0, L)$ as a function of $x$ by virtue of the estimate
\begin{equation}\label{pure-se}
\sup_{t\in [0, T']} \no{v(t)}_{H_x^s(0, L)}
\leq 
c \, \big(\no{g}_{H^{\frac{2s+1}{4}}(\mathbb R)} + \no{h}_{H^{\frac{2s-1}{4}}(\mathbb R)}\big),
\end{equation}
where $c = c(s, \lambda, L, \alpha, \nu, T')>0$. In addition, the map $[0, T'] \ni t \mapsto v(t) \in H_x^s(0, L)$ is continuous.
\end{theorem}

\begin{remark}
Estimate \eqref{pure-se} will actually be established by means of the more general estimate
\begin{equation}\label{pure-se-l2}
\sup_{t\in [0, T']} \no{v(t)}_{H_x^s(0, L)}
\leq 
c \, 
\big(
\no{g}_{H^{\frac{2s+1}{4}}(\mathbb R)} + \no{h}_{L^2(\mathbb R)} + \no{h}_{H^{\frac{2s-1}{4}}(\mathbb R)}\big),
\quad s \geq 0, 
\end{equation}
which readily yields \eqref{pure-se} when $s\geq \frac 12$. 
\end{remark}

\begin{remark}[CGL versus NLS]
\label{cgl-nls-r}
In the NLS limit $\nu \to 0^+$, either $\lambda$ or $\frac 1\lambda$ tends to $0^+$ causing the constant $c(s, \lambda, L, \alpha, \nu, T')$ in \eqref{pure-se} to blow up (see the relevant constants in \eqref{vj-dir} below). Thus, the CGL estimate \eqref{pure-se} does not directly imply a corresponding estimate for NLS.
\end{remark}

\begin{proof}
In view of the compact support of the boundary data, 
%
$
\widetilde g(\omega, T') 
=
\mathcal F\{g\}(i\omega)
$
%
and $\widetilde h(\omega, T') = \mathcal F\{h\}(i\omega)$. Thus, taking also into account that the initial datum and the forcing are both zero, in the case of the reduced problem \eqref{pure-ibvp} the unified transform solution \eqref{lcgl-sol-T} takes the form $v = v^+ + v^-$ with
\begin{align}
v^+(x, t) 
&=
- \frac{\nu+i\alpha}{\pi} 
\int_{\p D^+}  
\frac{e^{ikx-\omega t}}{e^{ikL} + e^{-ikL}} 
\left[
\mathcal F\{h\}(i\omega)
+
 ik e^{-ikL} \mathcal F\{g\}(i\omega)
\right] dk,
\label{v+}
\\
v^-(x, t) 
&=
- \frac{\nu+i\alpha}{\pi} 
\int_{\p D^-}   
\frac{e^{ik(x-L)-\omega t}}{e^{ikL} + e^{-ikL}} 
\left[
e^{ikL} \mathcal F\{h\}(i\omega)
+
ik \mathcal F\{g\}(i\omega)
\right] dk.
\label{v-}
\end{align}

The estimation of $v^+$ and $v^-$ is entirely analogous,  thus we only consider $v^-$. 
Parametrizing along the contour $\p D^-$ by letting, as appropriate (recall \eqref{hyp-eq}), $k = -\left(1 + i \lambda\right)r$, $r \in [0, \infty)$, or $k = \left(1 - \frac{i}{\lambda}\right)r$, $r \in [0, \infty)$, we write $v^- = v_{\lambda, 1}^- + v_{\lambda, 2}^-$ where 
\begin{align}
v_{\lambda, 1}^-(x, t) 
&=
\frac{\nu+i\alpha}{\pi} 
\int_0^\infty   
\frac{e^{-i (1 + i \lambda)r (x-L)- \omega((1 + i \lambda)r) t}}{e^{i(1 + i \lambda)rL} + e^{-i(1 + i \lambda)rL}} 
\nn\\
&\quad
\cdot 
\left[
e^{-i(1 + i \lambda)rL} \mathcal F\{h\}(i\omega((1 + i \lambda)r))
-
i(1 + i \lambda)r \mathcal F\{g\}(i\omega((1 + i \lambda)r))
\right] (1 + i \lambda) dr,
\label{v-1}
\\
v_{\lambda, 2}^-(x, t) 
&=
\frac{\nu+i\alpha}{\pi} 
\int_0^\infty   
\frac{e^{i(1 - \frac{i}{\lambda})r(x-L)- \omega((1 - \frac{i}{\lambda}) r) t}}{e^{i(1 - \frac{i}{\lambda})rL} + e^{-i(1 - \frac{i}{\lambda})rL}} 
\nn\\
&\quad
\cdot 
\left[
e^{i(1 - \frac{i}{\lambda})rL} \mathcal F\{h\}(i\omega((1 - \tfrac{i}{\lambda}) r))
+
i\left(1 - \tfrac{i}{\lambda}\right)r \mathcal F\{g\}(i\omega((1 - \tfrac{i}{\lambda}) r))
\right] (1 - \tfrac{i}{\lambda})dr.
\label{v-2}
\end{align}
We will exploit the $L^2$ characterization of the Sobolev norm, according to which, for any $s\in\mathbb N\cup \{0\}$,
\begin{equation}\label{l2char}
\no{v_{\lambda, 1}^-(t)}_{H_x^s(0, L)}
=
c(s, L) \sum_{j=0}^s \no{\p_x^j v_{\lambda, 1}^-(t)}_{L_x^2(0, L)}
\end{equation}
and similarly for $v_{\lambda, 2}^-$. Thus, for each $j\in\mathbb N \cup \{0\}$ and $t\in [0, T']$, we will estimate $\p_x^j v_{\lambda, 1}^-$ and $\p_x^j v_{\lambda, 2}^-$ as a functions of $x$ in $L_x^2(0, L)$. In this connection, we note that the exponentials involved in \eqref{v-1} and \eqref{v-2} are oscillatory in $t$ due to the fact that $\text{Re}(\omega((1 + i \lambda)r)) = 0 = \text{Re}(\omega((1 - \tfrac{i}{\lambda}) r))$ since $\omega$ is purely imaginary along $\p D^\pm$. Hence, by the triangle inequality,  
\begin{equation}
\begin{aligned}
\big\| \p_x^j v_{\lambda, 1}^-(t) \big\|_{L_x^2(0, L)}
&\leq
\frac{\left(\alpha^2+\nu^2\right)^{\frac 12}}{\pi}
\, \bigg\|
\int_0^\infty 
\frac{e^{\lambda r(x-L)}}{\delta_{\lambda, 1}(r)}
\left(1+\lambda^2\right)^{\frac{j+1}2} r^j
\Big[
e^{\lambda rL} \, \big|\mathcal F\{h\}(i\omega((1 + i \lambda)r)) \big|
\\
&\hskip 5cm
+
\left(1+\lambda^2\right)^{\frac 12} r \big| \mathcal F\{g\}(i\omega((1 + i \lambda)r)) \big|
\Big] dr
\bigg\|_{L_x^2(0, L)}
\end{aligned}
\end{equation}
and
\begin{equation}
\begin{aligned}
\big\| \p_x^j v_{\lambda, 2}^-(t) \big\|_{L_x^2(0, L)}
&\leq
\frac{\left(\alpha^2+\nu^2\right)^{\frac 12}}{\pi}
\, \bigg\|
\int_0^\infty 
\frac{e^{\frac{1}{\lambda} r (x-L)}}{\delta_{\lambda, 2}(r)}
\, \left(1 + \tfrac{1}{\lambda^2}\right)^{\frac{j+1}2} r^j 
\Big[
e^{\frac{1}{\lambda} rL} \,
\big| \mathcal F\{h\}(i\omega((1 - \tfrac{i}{\lambda}) r)) \big| 
\\
&\hskip 5cm
+
\left(1+\tfrac{1}{\lambda^2}\right)^{\frac 12} r 
\big|\mathcal F\{g\}(i\omega((1 - \tfrac{i}{\lambda}) r)) \big|
\Big] dr
\bigg\|_{L_x^2(0, L)}
\end{aligned}
\end{equation}
where
\begin{equation}
\delta_{\lambda, 1}(r) := \left|e^{i (1 + i \lambda) rL} + e^{-i(1 + i \lambda)rL}\right|,
\quad
\delta_{\lambda, 2}(r) := \left|e^{i (1 - \frac{i}{\lambda}) rL} + e^{-i(1 - \frac{i}{\lambda})rL}\right|.
\end{equation}
It is now helpful to make
the change of variables $y = \lambda(L-x)$ in the $L_x^2(0, L)$ norm of $v_{\lambda, 1}^-$ and $y = \frac{1}{\lambda} (L-x)$ in the $L_x^2(0, L)$ norm of $v_{\lambda, 2}^-$, as the relevant exponentials are then simplified into Laplace-type exponentials:
\begin{align}
\big\| \p_x^j v_{\lambda, 1}^-(t) \big\|_{L_x^2(0, L)}
&\leq
\frac{\left(\alpha^2+\nu^2\right)^{\frac 12} \left(1+\lambda^2\right)^{\frac{j+1}2}}{\pi \sqrt \lambda}
\, 
\Bigg\{
\bigg\|
\int_0^\infty 
e^{-ry} \, 
\frac{r^j e^{\lambda rL} \, \big|\mathcal F\{h\}(i\omega((1 + i \lambda)r)) \big|}{\delta_{\lambda, 1}(r)}
\, dr
\bigg\|_{L_y^2(0, \lambda L)}
\label{vjln}
\\
&\quad
+
\left(1+\lambda^2\right)^{\frac 12}
\bigg\|
\int_0^\infty 
e^{-ry} \, 
\frac{r^{j+1} \big| \mathcal F\{g\}(i\omega((1 + i \lambda)r)) \big|}{\delta_{\lambda, 1}(r)}
\, dr
\bigg\|_{L_y^2(0, \lambda L)}
\Bigg\}
\label{vjld}
\end{align}
and
\begin{align}
\big\| \p_x^j v_{\lambda, 2}^-(t) \big\|_{L_x^2(0, L)}
&\leq
\frac{\left(\alpha^2+\nu^2\right)^{\frac 12} \left(1 + \tfrac{1}{\lambda^2}\right)^{\frac{j+1}2}}{\pi \sqrt{\frac 1\lambda}}
\Bigg\{ \bigg\|
\int_0^\infty 
e^{-r y} \, 
\frac{r^j e^{\frac{1}{\lambda} rL} \,
\big| \mathcal F\{h\}(i\omega((1 - \tfrac{i}{\lambda}) r)) \big|}{\delta_{\lambda, 2}(r)} \, dr 
\bigg\|_{L_y^2(0, \frac{L}{\lambda})}
\label{vjmn}
\\
&\quad
+ \left(1+\tfrac{1}{\lambda^2}\right)^{\frac 12}
\, \bigg\|
\int_0^\infty 
e^{-r y} \, 
\frac{r^{j+1} \big|\mathcal F\{g\}(i\omega((1 - \tfrac{i}{\lambda}) r)) \big|}{\delta_{\lambda, 2}(r)} \, dr 
\bigg\|_{L_y^2(0, \frac{L}{\lambda})}
\Bigg\}.
\label{vjmd}
\end{align}

For the Dirichlet data terms \eqref{vjld} and \eqref{vjmd}, we use the fact that the Laplace transform is a bounded linear operator on $L^2(0, \infty)$. This classical result was first established by Hardy in Theorem 3 of \cite{h1929} (see also Lemma 3.2 in \cite{fhm2017}) and, when combined with the fact that
\begin{equation}\label{d-form2}
\begin{aligned}
&\delta_{\lambda, 1}(r) 
= \left|e^{i (1 + i \lambda)rL} + e^{-i (1 + i \lambda)rL}\right|
\geq
e^{\lambda r L} - e^{-\lambda r L},
\\
&\delta_{\lambda, 2}(r) 
=
\left|e^{i(1 - \frac{i}{\lambda})rL}-e^{-i(1 - \frac{i}{\lambda})rL}\right|
\geq
e^{\frac 1\lambda r L} - e^{-\frac 1\lambda r L},
\end{aligned}
\quad
r\geq 0,
\end{equation}
it yields the bounds
\begin{equation}\label{dbound0}
\begin{aligned}
\eqref{vjld}
&\leq
\frac{\left(\alpha^2+\nu^2\right)^{\frac 12} \left(1 + \lambda^2\right)^{1+\frac j2}}{\sqrt{\pi \lambda}}
\, \bigg\|
\frac{r^{j+1} \big| \mathcal F\{g\}(i\omega((1 + i \lambda)r)) \big|}{e^{\lambda r L} - e^{-\lambda r L}}
\bigg\|_{L_r^2(0, \infty)},
\\
\eqref{vjmd}
&\leq
\frac{\left(\alpha^2+\nu^2\right)^{\frac 12} \left(1 + \tfrac{1}{\lambda^2}\right)^{1+\frac j2}}{\sqrt{\frac{\pi}{\lambda}}}
\, \bigg\|
\frac{r^{j+1} \big|\mathcal F\{g\}(i\omega((1 - \tfrac{i}{\lambda}) r)) \big|}{e^{\frac 1\lambda r L} - e^{-\frac 1\lambda r L}}
\bigg\|_{L_r^2(0, \infty)}.
\end{aligned}
\end{equation}
Note that the singularity at $r=0$ is removable thanks to the factor of $r$ present in the numerator of both integrands. In particular, the Mean Value Theorem for the function $e^{\lambda r x}$ on $[-L, L]$ implies 
\begin{equation}\label{d-bound-n}
\frac{r}{e^{\lambda r L} - e^{-\lambda r L}}
\leq
\frac{e^{\lambda r L}}{2L \lambda}.
\end{equation}
The upper bound \eqref{d-bound-n} can be used for $r\in [0, 1]$, while for $r\geq 1$ we observe that
\begin{equation}\label{d-bound-f}
e^{\lambda r L} - e^{-\lambda r L} \geq e^{\lambda L} - e^{-\lambda L}.
\end{equation}
Hence, we use the triangle inequality to split the $r$-integrals in \eqref{dbound0} as follows:
\begin{align}
\eqref{vjld}
&\leq
\frac{\left(\alpha^2+\nu^2\right)^{\frac 12} \left(1 + \lambda^2\right)^{1+\frac j2}}{\sqrt{\pi \lambda}}
\Bigg\{ \bigg\|
\frac{r^{j+1}\big| \mathcal F\{g\}(i\omega((1 + i \lambda)r)) \big|}{e^{\lambda r L} - e^{-\lambda r L}}
\bigg\|_{L_r^2(0, 1)}
\label{vjld3a}
\\
&\hskip 4.5cm
+ \bigg\|
\frac{r^{j+1}\big| \mathcal F\{g\}(i\omega((1 + i \lambda)r)) \big|}{e^{\lambda r L} - e^{-\lambda r L}}
\bigg\|_{L_r^2(1, \infty)}
\Bigg\},
\label{vjld3b}
\\
\eqref{vjmd}
&\leq
\frac{\left(\alpha^2+\nu^2\right)^{\frac 12} \left(1 + \tfrac{1}{\lambda^2}\right)^{1+\frac j2}}{\sqrt{\frac{\pi}{\lambda}}}
\Bigg\{ \bigg\|
\frac{r^{j+1} \big|\mathcal F\{g\}(i\omega((1 - \tfrac{i}{\lambda}) r)) \big|}{e^{\frac 1\lambda r L} - e^{-\frac 1\lambda r L}}
\bigg\|_{L_r^2(0, 1)}
\label{vjmd3a}
\\
&\hskip 4.5cm
+
\bigg\|
\frac{r^{j+1} \big|\mathcal F\{g\}(i\omega((1 - \tfrac{i}{\lambda}) r)) \big|}{e^{\frac 1\lambda r L} - e^{-\frac 1\lambda r L}}
\bigg\|_{L_r^2(1, \infty)}
\Bigg\}.
\label{vjmd3b}
\end{align}

Concerning \eqref{vjld3a}, recalling that $\text{Re}(\omega((1 + i \lambda)r)) = 0$ and $\text{supp}(g) \subset (0, T')$, we have
\begin{equation}\label{ft-l2}
\big|\mathcal F\{g\}(i\omega((1 + i \lambda)r)) \big|
=
\left|
\int_0^{T'} e^{\omega((1 + i \lambda)r) t} g(t) dt
\right|
\leq
\int_0^{T'} \left| g(t) \right| dt
\leq
\sqrt{T'} \no{g}_{L^2(\mathbb R)}.
\end{equation}
Thus, using also \eqref{d-bound-n}, we find
\begin{equation}
\eqref{vjld3a} 
\leq
\frac{\sqrt{T'} \left(\alpha^2+\nu^2\right)^{\frac 12} \left(1 + \lambda^2\right)^{1+\frac j2} e^{\lambda L} }{2\sqrt{\pi} L  \lambda^{\frac 32}}
 \no{g}_{L^2(\mathbb R)}.
\label{vjld4a}
\end{equation}
Similarly,  for $\eqref{vjmd3a}$ we have
\begin{equation}
\eqref{vjmd3a} 
\leq 
\frac{\sqrt{T'} \left(\alpha^2+\nu^2\right)^{\frac 12} \left(1 + \tfrac{1}{\lambda^2}\right)^{1+\frac j2} \lambda^{\frac 32} e^{\frac L \lambda}}{2\sqrt{\pi} L}
 \no{g}_{L^2(\mathbb R)}.
\label{vjmd4a}
\end{equation}
Concerning \eqref{vjld3b}, we first use the bound \eqref{d-bound-f} to obtain
$$
\eqref{vjld3b}
\leq
\frac{\left(\alpha^2+\nu^2\right)^{\frac 12} \left(1 + \lambda^2\right)^{1+\frac j2}}{\sqrt{\pi \lambda} \left(e^{\lambda L} - e^{-\lambda L}\right)}
\, 
\Big\|r^{j+1}\big| \mathcal F\{g\}(i\omega((1 + i \lambda)r)) \big| \Big\|_{L_r^2(1, \infty)}.
$$
Next, we aim to simplify the argument of the Fourier transform on the right side in order to be able to employ the Fourier characterization of the Sobolev norm \eqref{sob-ft}. To this end, we note that
$i\omega((1 + i \lambda)r) = (i\nu - \alpha) (1+i\lambda)^2 r^2$ and
$
(i\nu-\alpha)(1 + i \lambda)^2
=
- \frac{2 \left(\alpha^2+\nu^2\right) \left(-\alpha+\sqrt{\alpha^2+\nu^2}\right)}{\nu^2} < 0,
$
thus we make  the change of variable
\begin{equation}\label{cov}
r = \frac{\sqrt{-\tau}}{\sqrt{(\alpha - i\nu)(1 + i \lambda)^2}}
\end{equation}
and employ the Fourier characterization \eqref{sob-ft} to arrive at the estimate
\begin{align}\label{vjld4b}
\eqref{vjld3b}
&\leq
\frac{\left(\alpha^2+\nu^2\right)^{\frac 12} \left(1 + \lambda^2\right)^{1+\frac j2}}{\sqrt{\pi \lambda} \left(e^{\lambda L} - e^{-\lambda L}\right)}
\no{
\frac{|\tau|^{\frac{j+1}{2}}}{\left[(\alpha - i\nu)(1 + i \lambda)^2\right]^{\frac{j+1}{2}}} 
\cdot
\frac{\big| \mathcal F\{g\}(\tau) \big|}{\sqrt 2\left[(\alpha - i\nu)(1 + i \lambda)^2 |\tau|\right]^{\frac 14}}
}_{L_\tau^2(-\infty, -(\alpha-i\nu)(1 + i \lambda)^2)}
\nn\\
&\leq
\frac{\left(\alpha^2+\nu^2\right)^{\frac{1-2j}{8}} \left(1 + \lambda^2\right)^{\frac 14}}{\sqrt{2\pi \lambda} \left(e^{\lambda L} - e^{-\lambda L}\right)}
\no{g}_{H^{\frac{2j+1}{4}}(\mathbb R)}.
\end{align}
Similarly, for \eqref{vjmd3b} we find
\begin{align}\label{vjmd4b}
\eqref{vjmd3b}
&\leq
\frac{\left(\alpha^2+\nu^2\right)^{\frac 12} \left(1 + \tfrac{1}{\lambda^2}\right)^{1+\frac j2}}{\sqrt{\frac{2\pi}{\lambda}} \big(e^{\frac L \lambda} - e^{-\frac L \lambda}\big)}
\Big\|
r^{j+1} \big|\mathcal F\{g\}(i\omega((1 - \tfrac{i}{\lambda}) r)) \big|
\Big\|_{L_r^2(1, \infty)}
\nn\\
&\leq
\frac{\left(\alpha^2+\nu^2\right)^{\frac 12} \left(1 + \tfrac{1}{\lambda^2}\right)^{1+\frac j2}}{\sqrt{\frac{2\pi}{\lambda}} \big(e^{\frac L \lambda} - e^{-\frac L \lambda}\big)}
\frac{1}{\left(|\alpha - i\nu| |1 - \frac{i}{\lambda}|^2\right)^{\frac{2j+3}{4}}}
\no{
|\tau|^{\frac{2j+1}{4}}
\big| \mathcal F\{g\}(\tau) \big| 
}_{L_\tau^2(\mathbb R)}
\nn\\
&\leq
\frac{\left(\alpha^2+\nu^2\right)^{\frac{1-2j}{8}} \left(1 + \tfrac{1}{\lambda^2}\right)^{\frac 14}}{\sqrt{\frac{2\pi}{\lambda}} \big(e^{\frac L \lambda} - e^{-\frac L \lambda}\big)}
\no{g}_{H^{\frac{2j+1}{4}}(\mathbb R)}.
\end{align}
Combining the individual estimates \eqref{vjld4a}-\eqref{vjmd4b}, we deduce the following estimates for the Dirichlet data terms~\eqref{vjld} and \eqref{vjmd}:
\begin{equation}\label{vj-dir}
\begin{aligned}
\eqref{vjld}
&\leq
\frac{\sqrt{T'} \left(\alpha^2+\nu^2\right)^{\frac 12} \left(1 + \lambda^2\right)^{1+\frac j2} e^{\lambda L} }{2\sqrt{\pi} L  \lambda^{\frac 32}}
 \no{g}_{L^2(\mathbb R)}
+
\frac{\left(\alpha^2+\nu^2\right)^{\frac{1-2j}{8}} \left(1 + \lambda^2\right)^{\frac 14}}{\sqrt{2\pi \lambda} \left(e^{\lambda L} - e^{-\lambda L}\right)}
\no{g}_{H^{\frac{2j+1}{4}}(\mathbb R)},
\\
\eqref{vjmd}
&\leq
\frac{\sqrt{T'} \left(\alpha^2+\nu^2\right)^{\frac 12} \left(1 + \tfrac{1}{\lambda^2}\right)^{1+\frac j2} \lambda^{\frac 32} e^{\frac L \lambda}}{2\sqrt{\pi} L}
 \no{g}_{L^2(\mathbb R)}
 +
 \frac{\left(\alpha^2+\nu^2\right)^{\frac{1-2j}{8}} \left(1 + \tfrac{1}{\lambda^2}\right)^{\frac 14}}{\sqrt{\frac{2\pi}{\lambda}} \big(e^{\frac L \lambda} - e^{-\frac L \lambda}\big)}
\no{g}_{H^{\frac{2j+1}{4}}(\mathbb R)}.
\end{aligned}
\end{equation}

Proceeding to the Neumann data terms \eqref{vjln} and \eqref{vjmn}, we begin by noticing that the factor of $r$ that was present in the numerators of the integrands in the Dirichlet terms \eqref{vjld} and \eqref{vjmd} is now missing. This prevents us from using the boundedness of the Laplace transform in $L^2(0, \infty)$ directly, since combining it with the lower bounds \eqref{d-form2} and then the inequality \eqref{d-bound-n} would introduce a singularity at $r=0$ in our estimates.
This singularity, however, is only a result of the (crude) lower bounds \eqref{d-form2}, since the denominators $\delta_{\lambda, 1}(r)$ and $\delta_{\lambda, 2}(r)$ are in fact strictly positive for $r\geq 0$ and $\lambda > 0$. Indeed, 
%
$
\delta_{\lambda, 1}(r) 
=
\sqrt{e^{2\lambda r L} + e^{-2\lambda r L} + 2 \cos(2rL)}
$
%
and $e^{2\lambda r L} + e^{-2\lambda r L}$ has a global minimum equal to $2$ at $r=0$ and is strictly increasing for $r>0$. Hence, $\delta_{\lambda, 1}(r)> \sqrt{2\left[1+\cos(2rL)\right]} = 2\left|\cos(rL)\right|$ for all $r>0$, which implies that $\delta_{\lambda, 1}(r)>0$ for all $r>0$. 
Moreover, $\delta_{\lambda, 1}(0) = 2$ and $\lim_{r\to\infty} \delta_{\lambda, 1}(r) = \infty$. Thus, $\delta_{\lambda, 1}(r) > 0$ for $r \in [0, \infty)$. The same reasoning applies for $\delta_{\lambda, 2}(r)$. 

Therefore, no singularity is present in the $r$-integrals of \eqref{vjln}-\eqref{vjmd} and, in particular, in the Neumann data terms \eqref{vjln} and \eqref{vjmn}. At the same time, due to the fact that for small $\lambda>0$ the denominators $\delta_{\lambda, 1}(r)$ and $\delta_{\lambda, 2}(r)$ can get very close to zero for some positive values of $r$, in order to estimate the Neumann data terms \eqref{vjln} and \eqref{vjmn}, prior to invoking the boundedness of the Laplace transform in $L^2(0, \infty)$ we  break the $r$-integrals via the triangle inequality as follows:
\begin{align}
\eqref{vjln}
&\leq
\frac{\left(\alpha^2+\nu^2\right)^{\frac 12} \left(1+\lambda^2\right)^{\frac{j+1}2}}{\pi \sqrt \lambda}
\Bigg\{
\bigg\|
\int_0^1
e^{-ry} \, 
\frac{r^j e^{\lambda rL} \, \big|\mathcal F\{h\}(i\omega((1 + i \lambda)r)) \big|}{\delta_{\lambda, 1}(r)}
\, dr
\bigg\|_{L_y^2(0, \lambda L)}
\label{vjln-n}
\\
&\hskip 4.5cm
+
\bigg\|
\int_1^\infty 
e^{-ry} \, 
\frac{r^j e^{\lambda rL} \, \big|\mathcal F\{h\}(i\omega((1 + i \lambda)r)) \big|}{\delta_{\lambda, 1}(r)}
\, dr
\bigg\|_{L_y^2(0, \lambda L)}
\Bigg\},
\label{vjln-f}
\\
\eqref{vjmn}
&\leq
\frac{\left(\alpha^2+\nu^2\right)^{\frac 12} \left(1 + \tfrac{1}{\lambda^2}\right)^{\frac{j+1}2}}{\pi \sqrt{\frac 1\lambda}}
\Bigg\{
\bigg\|
\int_0^1
e^{-r y} \, 
\frac{r^j e^{\frac{1}{\lambda} rL} \,
\big| \mathcal F\{h\}(i\omega((1 - \tfrac{i}{\lambda}) r)) \big|}{\delta_{\lambda, 2}(r)} \, dr 
\bigg\|_{L_y^2(0, \frac{L}{\lambda})}
\label{vjmn-n}
\\
&\hskip 4.5cm
+
\int_1^\infty 
e^{-r y} \, 
\frac{r^j e^{\frac{1}{\lambda} rL} \,
\big| \mathcal F\{h\}(i\omega((1 - \tfrac{i}{\lambda}) r)) \big|}{\delta_{\lambda, 2}(r)} \, dr 
\bigg\|_{L_y^2(0, \frac{L}{\lambda})}
\Bigg\}.
\label{vjmn-f}
\end{align}

For \eqref{vjln-f} and \eqref{vjmn-f}, the boundedness of the Laplace transform in $L^2(0, \infty)$ \cite[Theorem 3]{h1929} yields
\begin{align*}
\eqref{vjln-f}
&\leq
\frac{\left(\alpha^2+\nu^2\right)^{\frac 12} \left(1+\lambda^2\right)^{\frac{j+1}2}}{\sqrt{\pi \lambda}}
\,
\bigg\|
\frac{r^j e^{\lambda rL} \, \big|\mathcal F\{h\}(i\omega((1 + i \lambda)r)) \big|}{\delta_{\lambda, 1}(r)}
\bigg\|_{L_r^2(1, \infty)},
\\
\eqref{vjmn-f}
&\leq
\frac{\left(\alpha^2+\nu^2\right)^{\frac 12} \left(1 + \tfrac{1}{\lambda^2}\right)^{\frac{j+1}2}}{\sqrt{\frac{\pi}{\lambda}}}
\,
\bigg\|
\frac{r^j e^{\frac{1}{\lambda} rL} \,
\big| \mathcal F\{h\}(i\omega((1 - \tfrac{i}{\lambda}) r)) \big|}{\delta_{\lambda, 2}(r)} 
\bigg\|_{L_r^2(1, \infty)}.
\end{align*}
Then, combining the lower bounds \eqref{d-form2} with the fact that for $r\geq 1$ we have, via the bound  \eqref{d-bound-f},
\begin{equation}\label{d-bound-f2}
\frac{e^{\lambda r L}}{e^{\lambda r L} - e^{-\lambda r L}}
=
1 + \frac{e^{-\lambda r L}}{e^{\lambda r L} - e^{-\lambda r L}}
\leq
1 + \frac{1}{e^{\lambda L} - e^{-\lambda L}},
\end{equation}
we can proceed as for \eqref{vjld3b} and \eqref{vjmd3b} to eventually obtain the analogues of \eqref{vjld4b} and \eqref{vjmd4b}, namely 
\begin{equation}
\begin{aligned}
\eqref{vjln-f}
&\leq
\frac{\left(\alpha^2+\nu^2\right)^{\frac{3-2j}{8}} \left(1 + \lambda^2\right)^{\frac 34}}{\sqrt{2\pi \lambda}}
\left(1 + \frac{1}{e^{\lambda L} - e^{-\lambda L}}\right)
\no{h}_{H^{\frac{2j-1}{4}}(\mathbb R)},
\label{vjln-f2}
\\
\eqref{vjmn-f}
&\leq
\frac{\left(\alpha^2+\nu^2\right)^{\frac{3-2j}{8}} \left(1 + \tfrac{1}{\lambda^2}\right)^{\frac 34}}{\sqrt{\frac{2\pi}{\lambda}}}
\left(1 + \frac{1}{e^{\frac{L}{\lambda}} - e^{-{\frac{L}{\lambda}}}}\right)
\no{h}_{H^{\frac{2j-1}{4}}(\mathbb R)}.
\end{aligned}
\end{equation}

On the other hand, for the terms \eqref{vjln-n} and \eqref{vjmn-n}, while the bound \eqref{d-bound-f2} is not valid and, moreover, the bound \eqref{d-bound-n} (which is valid) is not convenient as it results in a singularity at $r=0$, we use the fact that $\delta_{\lambda, 1}(r) \delta_{\lambda, 2}(r) > 0$ for  $r \in [0, \infty)$ together with the analogue of inequality \eqref{ft-l2} for $h$ in order to infer
\begin{equation}
\begin{aligned}
\eqref{vjln-n} 
&\leq 
\frac{\left(\alpha^2+\nu^2\right)^{\frac 12} \left(1+\lambda^2\right)^{\frac{j+1}2}}{\pi \sqrt \lambda}
\, \sqrt{T'} 
\no{h}_{L^2(\mathbb R)} 
\bigg\|
\int_0^1
e^{-ry} \, 
\frac{r^j e^{\lambda rL}}{\delta_{\lambda, 1}(r)} \, dr
\bigg\|_{L_y^2(0, \lambda L)},
\\
\eqref{vjmn-n} 
&\leq 
\frac{\left(\alpha^2+\nu^2\right)^{\frac 12} \left(1+\frac{1}{\lambda^2}\right)^{\frac{j+1}2}}{\pi \sqrt{\frac 1\lambda}}
\, \sqrt{T'} 
\no{h}_{L^2(\mathbb R)} 
\bigg\|
\int_0^1
e^{-ry} \, 
\frac{r^j e^{\frac 1\lambda rL}}{\delta_{\lambda, 2}(r)} \, dr
\bigg\|_{L_y^2(0, \frac{L}{\lambda})},
\label{vjmn-n2}
\end{aligned}
\end{equation}
Since $\delta_{\lambda, 1}(r) \delta_{\lambda, 2}(r) > 0$ for $r \in [0, \infty)$, we have
\begin{equation}\label{small-r}
\bigg\|
\int_0^1
e^{-ry} \, 
\frac{r^j e^{\lambda rL}}{\delta_{\lambda, 1}(r)} \, dr
\bigg\|_{L_y^2(0, \lambda L)}
+
\bigg\|
\int_0^1
e^{-ry} \, 
\frac{r^j e^{\frac 1\lambda rL}}{\delta_{\lambda, 2}(r)} \, dr
\bigg\|_{L_y^2(0, \frac{L}{\lambda})}
= c_{\lambda, L, j} < \infty.
\end{equation}
Thus, up to a finite multiplicative constant depending on the various parameters of the problem, \eqref{vjln-n} and \eqref{vjmn-n} are bounded above by the $L^2$ norm of the Neumann datum $h$.

Overall, combining the Dirichlet data estimates \eqref{vj-dir} with the Neumann data estimates \eqref{vjln-f2}-\eqref{vjmn-n2} and the $L^2$ characterization of the Sobolev norm \eqref{l2char}, we deduce the following Sobolev space estimate for the component $v^-(x, t)$ of the solution of the reduced initial-boundary value problem \eqref{pure-ibvp}:
\begin{equation}\label{vj-est00}
\no{v^-(t)}_{H_x^j(0, L)}
\leq 
c \, 
\big(
\no{g}_{H^{\frac{2j+1}{4}}(\mathbb R)} + \no{h}_{L^2(\mathbb R)} + \no{h}_{H^{\frac{2j-1}{4}}(\mathbb R)}\big), \quad j \in \mathbb N \cup \{0\}, \ t \in [0, T'],
\end{equation}
where $c = c(j, \lambda, L, \alpha, \nu, T')$. Furthermore, through interpolation (e.g. see Theorem 5.1 in \cite{lm1972} or Theorem~4.1.2 in \cite{bl1976}), the validity of this estimate can be extended to any $s\geq 0$, namely  
\begin{equation}\label{vj-est0}
\no{v^-(t)}_{H_x^s(0, L)}
\leq 
c \, 
\big(
\no{g}_{H^{\frac{2s+1}{4}}(\mathbb R)} + \no{h}_{L^2(\mathbb R)} + \no{h}_{H^{\frac{2s-1}{4}}(\mathbb R)}\big), \quad s\geq 0, \ t \in [0, T'].
\end{equation}
Estimate \eqref{vj-est0} is precisely estimate \eqref{pure-se-l2} for $v^-$. 
Restricting $s\geq \frac 12$ allows us to control the $L^2$ norms of $g$ and $h$ by their associated higher-order Sobolev norms and hence further deduce the desired estimate \eqref{pure-se} for $v^-$, i.e. 
\begin{equation}\label{vj-est}
\no{v^-(t)}_{H_x^s(0, L)}
\leq 
c \, 
\big(\no{g}_{H^{\frac{2s+1}{4}}(\mathbb R)} + \no{h}_{H^{\frac{2s-1}{4}}(\mathbb R)}\big), \quad s\geq \frac 12, \ t \in [0, T'].
\end{equation}

We emphasize that, as already noted in Remark \ref{cgl-nls-r} above, attempting to take the NLS limit $\nu \to 0^+$ of estimate~\eqref{vj-est} will result in the blow-up of the relevant constant $c = c(s, \lambda, L, \alpha, \nu, T')$. 
Also, as noted earlier, the estimation of the component $v^+$ of the solution $v$ of the reduced initial-boundary value problem \eqref{pure-ibvp} is entirely analogous.
Moreover, the continuity of the map $[0, T'] \ni t \mapsto v(t) \in H_x^s(0, L)$ follows via the dominated convergence theorem and estimations entirely analogous to the ones above. Thus, the proof of Theorem \ref{pure-ibvp-t} is complete. 
\end{proof}

In addition to the space regularity established by Theorem \ref{pure-ibvp-t}, we have the following time regularity estimates for the reduced initial-boundary value problem \eqref{pure-ibvp}.
\begin{theorem}[Time estimates]
\label{pure-te-t}
Suppose $s \geq \frac 12$. Then, the solution of the reduced initial-boundary value problem~\eqref{pure-ibvp} admits the Sobolev time estimates
\begin{align}
\sup_{x\in [0, L]} \no{v(x)}_{H_t^{\frac{2s+1}{4}}(0, T')}
&\leq 
c \,  
\big(\no{g}_{H^{\frac{2s+1}{4}}(\mathbb R)} + \no{h}_{H^{\frac{2s-1}{4}}(\mathbb R)}\big),
\label{pure-te}
\\
\sup_{x\in [0, L]} \no{v_x(x)}_{H_t^{\frac{2s-1}{4}}(0, T')}
&\leq 
c \,  
\big(\no{g}_{H^{\frac{2s+1}{4}}(\mathbb R)} + \no{h}_{H^{\frac{2s-1}{4}}(\mathbb R)}\big),
\label{pure-te2}
\end{align}
where  $c = c(s, \lambda, L, \alpha, \nu, T')>0$. In addition, the maps $[0, L] \ni x \mapsto v(x) \in H_t^{\frac{2s+1}{4}}(0, T')$ and $[0, L] \ni x \mapsto v_x(x) \in H_t^{\frac{2s-1}{4}}(0, T')$ are continuous.
\end{theorem}
\begin{proof}
The main ideas are those used in the proof of Theorem \ref{pure-ibvp-t}. In fact, the proof of Theorem \ref{pure-te-t} is simpler in the sense that it does not require the boundedness of the Laplace transform in $L^2(0, \infty)$. Thus, we only provide the essential details.
Returning to the expression \eqref{v-1} for the component $v_{\lambda, 1}^-$, we split the range of integration near and away from $r=0$ to write
\begin{align}
v_{\lambda, 1}^-(x, t) 
&=
\frac{\nu+i\alpha}{\pi} 
\int_1^\infty   
\frac{e^{-i (1 + i \lambda)r (x-L)- \omega((1 + i \lambda)r) t}}{e^{i(1 + i \lambda)rL} + e^{-i(1 + i \lambda)rL}} 
\nn\\
&\hskip 2cm
\cdot 
\left[
e^{-i(1 + i \lambda)rL} \mathcal F\{h\}(i\omega((1 + i \lambda)r))
-
i(1 + i \lambda)r \mathcal F\{g\}(i\omega((1 + i \lambda)r))
\right] (1 + i \lambda) dr
\label{v1-f}
\\
&\quad
+ \frac{\nu+i\alpha}{\pi} 
\int_0^1   
\frac{e^{-i (1 + i \lambda)r (x-L)- \omega((1 + i \lambda)r) t}}{e^{i(1 + i \lambda)rL} + e^{-i(1 + i \lambda)rL}} 
\nn\\
&\hskip 2cm
\cdot 
\left[
e^{-i(1 + i \lambda)rL} \mathcal F\{h\}(i\omega((1 + i \lambda)r))
-
i(1 + i \lambda)r \mathcal F\{g\}(i\omega((1 + i \lambda)r))
\right] (1 + i \lambda) dr
\label{v1-n}
\end{align}
and then proceed as follows. 

For the term \eqref{v1-f}, since $r>1$ we apply the change of variable \eqref{cov} in order to exploit the fact that $\text{Re}(\omega) = 0$ along $\p D^-$ and hence use the Fourier characterization of the $H^{\frac{2s+1}{4}}(\mathbb R)$ norm \eqref{sob-ft}. More specifically, 
\begin{align}
\eqref{v1-f}
&=
\frac{\nu+i\alpha}{\pi} 
\int_{-\infty}^{-(\alpha - i\nu)(1 + i \lambda)^2} 
e^{i\tau t} 
\cdot
\frac{e^{-i (1 + i \lambda)r(\tau) (x-L)}}{e^{i(1 + i \lambda)r(\tau)L} + e^{-i(1 + i \lambda)r(\tau)L}} 
\nn\\
&\quad
\cdot 
\left[
e^{-i(1 + i \lambda)r(\tau)L} \mathcal F\{h\}(\tau)
-
i\left(1 + i \lambda\right) \frac{\sqrt{-\tau}}{\sqrt{(\alpha - i\nu)(1 + i \lambda)^2}} \mathcal F\{g\}(\tau)
\right] \frac{1 + i \lambda}{2\sqrt{(\alpha - i\nu)(1 + i \lambda)^2} \sqrt{-\tau}} \, d\tau
\nn
\end{align}
and so, via the Fourier inversion theorem, the temporal Fourier transform of \eqref{v1-f} is readily extracted as

$$
\begin{aligned}
\mathcal F\left\{\eqref{v1-f}\right\}(\tau)
&=
\left(\nu+i\alpha\right) \chi_{(-\infty, -(\alpha - i\nu)(1 + i \lambda)^2]}(\tau)
\frac{e^{-i (1 + i \lambda)r(\tau) (x-L)}}{e^{i(1 + i \lambda)r(\tau)L} + e^{-i(1 + i \lambda)r(\tau)L}} 
\\
&\quad
\cdot 
\left[
e^{-i(1 + i \lambda)r(\tau)L} \mathcal F\{h\}(\tau)
-
i\left(1 + i \lambda\right) \frac{\sqrt{-\tau}}{\sqrt{(\alpha - i\nu)(1 + i \lambda)^2}}  \mathcal F\{g\}(\tau)
\right] \frac{1 + i \lambda}{\sqrt{(\alpha - i\nu)(1 + i \lambda)^2} \sqrt{-\tau}}.
\end{aligned}
$$
Thus, in view of \eqref{sob-ft}, \eqref{d-form2}, \eqref{d-bound-f} and \eqref{d-bound-f2}, we find
\begin{align}
&
\no{\eqref{v1-f}}_{H_t^{\frac{2s+1}{4}}(\mathbb R)}^2
=
2\left(\alpha^2+\nu^2\right)^{\frac 12}
\left(1 + \frac{1}{e^{\lambda L} - e^{-\lambda L}}\right)^2
\int_{-\infty}^{-(\alpha - i\nu)(1 + i \lambda)^2} 
\left(1+\tau^2\right)^{\frac{2s+1}{4}} |\tau|^{-1}
\left|\mathcal F\{h\}(\tau)\right|^2 d\tau
\nn\\
&\hspace*{3.4cm}
+
\frac{2}{\left(e^{\lambda L} - e^{-\lambda L}\right)^2} 
\int_{-\infty}^{-(\alpha - i\nu)(1 + i \lambda)^2} 
\left(1+\tau^2\right)^{\frac{2s+1}{4}} \left|\mathcal F\{g\}(\tau)\right|^2 d\tau
\nn\\
&
\leq
2 \left(1 + \frac{1}{e^{\lambda L} - e^{-\lambda L}}\right)^2 
\left[\left(\alpha^2+\nu^2\right)^{\frac 12} + \frac{1}{1 +  \lambda^2}\right] \no{h}_{H^{\frac{2s-1}{4}}(\mathbb R)}^2
+
\frac{2}{\left(e^{\lambda L} - e^{-\lambda L}\right)^2} \no{g}_{H^{\frac{2s+1}{4}}(\mathbb R)}^2,
\quad
s\in\mathbb R,
\label{v1-f-est}
\end{align}
with the last inequality thanks to the fact that, since $|\tau| \geq (\alpha - i\nu)(1 + i \lambda)^2 > 0$, for $c^2 = 1 + \frac{1}{(\alpha - i\nu)^2(1 + i \lambda)^4} > 1$ we have $|\tau|^{-1} \leq c \left(1+\tau^2\right)^{-\frac 12}$. 

On the other hand, for the term \eqref{v1-n} we use the $L^2$ characterization of the $H^{\frac{2s+1}{4}}(0, T')$ norm and the analogue of \eqref{small-r}. In particular, for each $\mu \in \mathbb N\cup \{0\}$, recalling that $\text{Re}(\omega((1 + i \lambda)r)) = 0$ we have
\begin{align}
\no{\eqref{v1-n}}_{H_t^\mu(0, T')}
\leq
c(\mu, T') \, 
\frac{\left(\alpha^2+\nu^2\right)^{\frac 12}\left(1+\lambda^2\right)^{\frac 12}}{\pi}
\sum_{j=0}^\mu
&\Bigg[
e^{\lambda L}
\no{ 
\int_0^1   
\frac{\left|\omega((1 + i \lambda)r)\right|^\mu}{\delta_{\lambda, 1}(r)} 
\left|\mathcal F\{h\}(i\omega((1 + i \lambda)r))\right|
 dr
}_{L_t^2(0, T')}
\nn\\
+
\left(1+\lambda^2\right)^{\frac 12}
&\no{
\int_0^1   
\frac{\left|\omega((1 + i \lambda)r)\right|^\mu r}{\delta_{\lambda, 1}(r)} 
\left|\mathcal F\{g\}(i\omega((1 + i \lambda)r))\right|
 dr
}_{L_t^2(0, T')}
\Bigg].
\nn
\end{align}
Hence, by \eqref{ft-l2} and the fact that $\delta_{\lambda, 1}(r) > 0$ for all $r \in [0, \infty)$, we obtain
\begin{align}
\no{\eqref{v1-n}}_{H_t^\mu(0, T')}
&\leq
c(\mu, T') \, 
\frac{\left(\alpha^2+\nu^2\right)^{\frac 12}\left(1+\lambda^2\right)^{\frac 12}}{\pi}
\sum_{j=0}^\mu
\Bigg[
e^{\lambda L}
\sqrt{T'} \no{h}_{L_t^2(0, T')}
\no{ 
\int_0^1   
\frac{\left|\omega((1 + i \lambda)r)\right|^\mu}{\delta_{\lambda, 1}(r)} 
 dr
}_{L_t^2(0, T')}
\nn\\
&\hskip 3.15cm
+
\left(1+\lambda^2\right)^{\frac 12}
\sqrt{T'} \no{g}_{L_t^2(0, T')}
\no{
\int_0^1   
\frac{\left|\omega((1 + i \lambda)r)\right|^\mu r}{\delta_{\lambda, 1}(r)} 
 dr
}_{L_t^2(0, T')}
\Bigg]
\nn\\
&\leq
c \left(\no{h}_{L_t^2(0, T')} + \no{g}_{L_t^2(0, T')} \right)
\label{v1-n-est}
\end{align}
where $c = c(\mu, \lambda, L, \alpha, \nu, T')$. The validity of this estimate can actually be extended to any $\mu \geq 0$ via interpolation \cite[Theorem 5.1]{lm1972}. Hence, upon restricting $s\geq \frac 12$ so that the $L^2$ norm of $h$ can be controlled by its $H^{\frac{2s-1}{4}}$ norm,  estimates~\eqref{v1-f-est} and \eqref{v1-n-est} yield the time estimate \eqref{pure-te} for the component $v_{\lambda, 1}^-$ of $v^-$. The remaining terms $v_{\lambda, 2}^-$ and $v^+$ can be handled analogously. 

The proof of the second time estimate \eqref{pure-te2} is simpler because the derivative with respect to $x$ creates an extra factor of $r$ in the integrands of \eqref{v-1} and \eqref{v-2} and hence eliminates the need for splitting the range of integration near and away from zero. Finally, the continuity of the corresponding maps follows via the same arguments combined with the dominated convergence theorem.
\end{proof}

\subsection{Linear decomposition}

Our goal is to use Theorem \ref{pure-ibvp-t} for the reduced initial-boundary value problem~\eqref{pure-ibvp} in order to infer an analogous result for the general forced linear initial-boundary value problem \eqref{lcgl-ibvp-w}. This will be accomplished by exploiting  linearity and the superposition principle. Before doing so, however, in addition to Theorem \ref{pure-ibvp-t} we must derive appropriate estimates for the Cauchy (initial value) problem of the linear CGL equation.

More precisely, we consider the initial-boundary value problem \eqref{lcgl-ibvp-w} with data $u_0 \in H^s(0, L)$, $g_0 \in H^{\frac{2s+1}{4}}(0, T)$, $h_1 \in H^{\frac{2s-1}{4}}(0, T)$ and forcing $f \in C_t([0, T]; H_x^s(0, L))$, $s>\frac 12$, and let $U_0 \in H^s(\mathbb R)$ and $F\in C_t([0, T]; H_x^s(\mathbb R))$ be extensions of $u_0$ and $f$ such that
\begin{equation}\label{U0F-def}
\no{U_0}_{H^s(\mathbb R)} \leq 2\no{u_0}_{H^s(0, L)}
\quad
\text{and}
\quad
\no{F(t)}_{H_x^s(\mathbb R)} \leq 2\no{f(t)}_{H_x^s(0, L)}, 
\ t\in [0, T].
\end{equation}
Note that the existence of such extension follows from the definition of  $H^s(0, L)$ as the restriction of $H^s(\mathbb R)$ on the interval $(0, L)$ equipped with the infimum norm and the infimum approximation property. 
Then, the solution of  problem \eqref{lcgl-ibvp-w}, which is given by
\begin{equation}\label{s-not}
w(x, t) = S\big[u_0, g_0, h_1; f\big](x, t)
\end{equation}
with the solution operator on the right side provided by the unified transform formula \eqref{lcgl-sol-T}, can be expressed as the linear combination
\begin{equation}\label{lin-sup}
S\big[u_0, g_0, h_1; f\big]
= 
S\big[U_0; 0\big]\big|_{x\in (0, L)}
+
S\big[0; F\big]\big|_{x\in (0, L)}
+
S\big[0, \psi_0, \psi_1; 0\big] 
\end{equation}
where $S\big[U_0; 0\big]:= U$ is the solution of the homogeneous Cauchy problem
\begin{equation}\label{U-ivp}
\begin{aligned}
&U_t - \left(\nu + i \alpha \right) U_{xx} = 0,
\quad
x\in\mathbb R, \ t\in(0,T),
\\
&U(x,0) = U_0(x),
\end{aligned} 
\end{equation}
$S\big[0; F] := W$ satisfies the inhomogeneous Cauchy problem
\begin{equation}\label{W-ivp}
\begin{aligned}
&W_t - \left(\nu + i \alpha \right) W_{xx} = F(x, t),
\quad
x\in\mathbb R, \ t\in(0,T),
\\
&W(x,0) = 0,
\end{aligned} 
\end{equation}
and the third term on the right side of \eqref{lin-sup} denotes the solution to a reduced version of the initial-boundary value problem \eqref{lcgl-ibvp-w} in which the initial datum and the forcing are zero and the boundary data are given by
\begin{equation}\label{psi-def}
\begin{aligned}
&\psi_0(t) := g_0(t) - S\big[U_0; 0\big](0, t) - S\big[0; F\big](0, t),
\\
&\psi_1(t) := h_1(t) - \p_x S\big[U_0; 0\big](L, t) - \p_x S\big[0; F\big](L, t),
\end{aligned}
\end{equation}
i.e. they have been obtained from the original boundary data $g_0, h_1$ by subtracting from them the corresponding traces of the solutions to the Cauchy problems \eqref{U-ivp} and \eqref{W-ivp}. 
In this connection, in order to use the decomposition~\eqref{lin-sup}  for the purpose of estimates, we need to show that $\psi_0, \psi_1$ belong to the same spaces as $g_0, h_1$ by establishing that the aforementioned Cauchy problem traces belong to these spaces as well. This is accomplished through the results of Theorems \ref{U-ivp-t} and \ref{W-ivp-t} below. After proving these results, we will return to the superposition \eqref{lin-sup} and combine them with the space estimate of Theorem \ref{pure-ibvp-t} in order to deduce a corresponding  estimate for the initial-boundary value problem \eqref{lcgl-ibvp-w}.

\subsection{Estimates for the homogeneous linear Cauchy problem}
We now estimate the homogeneous linear problem \eqref{U-ivp}, before proceeding to the more elaborate estimation of the forced linear problem \eqref{W-ivp}.
\begin{theorem}[Homogeneous linear Cauchy problem]\label{U-ivp-t}
The solution $U = S\big[U_0; 0\big]$ to the Cauchy problem \eqref{U-ivp}  satisfies the space estimate
\begin{equation}\label{U-ivp-se}
\sup_{t\in [0, T]} \no{U(t)}_{H_x^s(\mathbb R)} \leq \no{U_0}_{H^s(\mathbb R)},  \quad s\in \mathbb R,
\end{equation}
and  the time estimates 
\begin{align}\label{U-ivp-te}
\sup_{x\in\mathbb R} \no{U(x)}_{H_t^{\frac{2s+1}{4}}(0, T)} &\leq c_2(s) \no{U_0}_{H^s(\mathbb R)}, \quad s\geq -\frac 12, 
\\
\label{Ux-ivp-te}
\sup_{x\in\mathbb R} \no{U_x(x)}_{H_t^{\frac{2s-1}{4}}(0, T)} &\leq c_2(s-1) \no{U_0}_{H^s(\mathbb R)}, \quad s\geq \frac 12,
\end{align}
with the constant $c_2(s) = c_2(s, \alpha, \nu, T)$ given by \eqref{c2-def}.
Furthermore, the map $[0, T] \ni t \mapsto U(t) \in H_x^s(\mathbb R)$ as well as the maps $\mathbb R \ni x \mapsto U(x) \in H_t^{\frac{2s+1}{4}}(0, T)$ and $\mathbb R \ni x \mapsto U_x(x) \in H_t^{\frac{2s-1}{4}}(0, T)$ are continuous.
\end{theorem}

\begin{proof}
The solution to problem \eqref{U-ivp} is given by
\begin{equation}\label{U-sol}
U(x, t) = \frac{1}{2\pi} \int_{\mathbb R} e^{ikx - \omega t} \mathcal F\{U_0\}(k) dk
\end{equation}
where $\mathcal F\{U_0\}(k) := \int_{\mathbb R} e^{-ikx} U_0(x) dx$ is the Fourier transform on the real line. In particular, $\mathcal F\{U\}(k, t) = e^{-\omega t} \mathcal F\{U_0\}(k)$ and so, by the definition of the Sobolev $H^s(\mathbb R)$ norm and the fact that $\big|e^{-\omega t}\big| = e^{-\nu k^2 t} \leq 1$ for $k\in\mathbb R$, $\alpha \in \mathbb R$, $\nu > 0$ and $t\geq 0$,  we readily obtain the space estimate \eqref{U-ivp-se}.

For the time estimate \eqref{U-ivp-te}, we note that, contrary to the reduced initial-boundary value problem \eqref{pure-ibvp} and the proof of Theorem \ref{pure-te-t}, where the change of variables $\tau = i\omega$ was employed due to the fact that $\text{Re}(\omega) = 0$ for $k \in \p D^\pm$, now that $k\in\mathbb R$ we have $\text{Re}(\omega) =  \nu k^2$ and so it is not possible to estimate the $H^{\frac{2s+1}{4}}(0, T)$ norm   of $U(x, t)$  via its Fourier characterization. Instead, we employ the $L^2(0, T)$ characterization (cf. \eqref{l2char}) 
\begin{equation}\label{hm-l2}
\no{U(x)}_{H_t^m(0, T)}
=
c(m, T)
\, \bigg(\sum_{j=0}^{[m]} \big\| \p_t^j U(x) \big\|_{L_t^2(0, T)} + \big\|\p_t^{[m]} U(x)\big\|_\mu\bigg)
=
c(m, T) \no{U(x)}_{W_t^m(0, T)}, \quad m\geq 0,
\end{equation}
where $W^m(0, T) := W^{m, p}(0, T)$ with $p=2$ is the Sobolev-Slobodeckij space,  the constant 
$c(m, T) \simeq T^{-m}$ (cf. \cite[Theorem 1.43]{l2023}), and $[m]$, $\mu$ respectively denote the integer and fractional parts of $m$ so that $\mu = m - [m]\in [0, 1)$, and the Sobolev-Slobodeckij-Gagliardo seminorm $\no{\cdot}_{\mu}$ is zero for $\mu = 0$ and  is otherwise given by
\begin{equation}\label{t-frac-def}
\no{U(x)}_{\mu}^2
=
2
\int_0^T\int_0^{T-t} \frac{\left|U(x, z+t) - U(x, t)\right|^2}{z^{1+2\mu}} dz dt,
\quad
\mu \in (0, 1).
\end{equation}
We note that the restriction $s\geq -\frac 12$ in estimate \eqref{U-ivp-te} is imposed so that $m = \frac{2s+1}{4} \geq 0$, thus allowing us to employ \eqref{hm-l2} for this choice of $m$.

We first estimate the whole derivative norms of \eqref{hm-l2} and then proceed to the fractional part \eqref{t-frac-def}.
Differentiating \eqref{U-sol} in $t$ and applying the triangle inequality, we have
\begin{align}
\big\| \p_t^j U(x) \big\|_{L_t^2(0, T)}
&\leq
\frac{1}{2\pi} 
\no{
\int_{|k|>1} e^{-\nu k^2 t} |\omega|^j \left|\mathcal F\{U_0\}(k)\right| dk}_{L^2(0, T)}
\label{U-te1}
\\
&\quad
+
\frac{1}{2\pi} 
\no{
\int_{|k|<1} e^{-\nu k^2 t} |\omega|^j \left|\mathcal F\{U_0\}(k)\right| dk}_{L^2(0, T)}.
\label{U-te2}
\end{align}
For the term \eqref{U-te1}, making the changes of variables $k=-\sqrt{\frac \tau \nu}$ for $k<0$ and $k= \sqrt{\frac \tau \nu}$ for $k>0$, both of which imply that $|\omega| = \nu k^2 = \tau$, we find
$$
\eqref{U-te1}
\leq
\frac{1}{4\pi \sqrt \nu} 
\no{
\int_\nu^\infty e^{-\tau t} \cdot \tau^{j-\frac 12} \left(
\big|\mathcal F\{U_0\}(-\sqrt{\tfrac \tau \nu})\big|
+
\big|\mathcal F\{U_0\}(\sqrt{\tfrac \tau \nu})\big|
\right)
 d\tau
}_{L^2(0, T)}.
$$
Thus, by the boundedness of the Laplace transform in $L^2(0, \infty)$, 
$$
\eqref{U-te1}
\leq
\frac{1}{4\sqrt{\pi \nu}} 
\no{\tau^{j-\frac 12} \left(
\big|\mathcal F\{U_0\}(-\sqrt{\tfrac \tau \nu})\big|
+
\big|\mathcal F\{U_0\}(\sqrt{\tfrac \tau \nu})\big|
\right)
}_{L^2(\nu, \infty)}.
$$
Changing variable from $\tau$ back to $k$ and using the fact that  $k^2 \simeq 1+k^2$ for $|k|>1$, we obtain
\begin{equation}\label{U-te3}
\eqref{U-te1}
\leq
\frac{\nu^{j-\frac 12}}{2\sqrt{\pi}} 
\no{(k^2)^{\frac{2j-\frac 12}{2}} 
\big|\mathcal F\{U_0\}(k)\big|
}_{L^2(\mathbb R \setminus [-1, 1])}
\leq 
\frac{\nu^{j-\frac 12}}{2\sqrt{\pi}} \no{U_0}_{H^{2j-\frac 12}(\mathbb R)}, \quad j\in\mathbb N\cup\{0\}.
\end{equation}
Furthermore, by Minkowski's integral inequality and the fact that $e^{-\nu k^2 t} \leq 1$,
$$
\eqref{U-te2}
\leq
\frac{1}{2\pi} 
\int_{|k|<1} \no{
e^{-\nu k^2 t} |\omega|^j \left|\mathcal F\{U_0\}(k)\right|}_{L^2(0, T)} dk
\leq
\frac{\sqrt T}{2\pi} 
\int_{|k|<1}   |\omega|^j \left|\mathcal F\{U_0\}(k)\right| dk
$$
and hence, by Cauchy-Schwarz inequality, which is convenient to use since the range of integration in $k$ is bounded, 
\begin{equation}\label{U-te4}
\begin{aligned}
\eqref{U-te2}
&\leq
\frac{\sqrt T}{2\pi} 
\left(\int_{|k|<1}   |\omega|^{2j} (1+k^2)^{-(2j-\frac 12)} dk\right)^{\frac 12} 
\left(\int_{|k|<1}  (1+k^2)^{2j-\frac 12}\left|\mathcal F\{U_0\}(k)\right|^2 dk\right)^{\frac 12}
\\
&\leq
\frac{\nu^j \sqrt T}{2^{\frac 14} \pi} 
\no{U_0}_{H^{2j-\frac 12}(\mathbb R)}, \quad j\in\mathbb N\cup\{0\}.
\end{aligned}
\end{equation}
Thus, returning to \eqref{U-te1}-\eqref{U-te2}, we deduce
\begin{equation}\label{Uj-te}
\big\|\p_t^j U(x)\big\|_{L_t^2(0, T)}
\leq
\frac{\nu^j(1+\frac{1}{\sqrt \nu}) (1+\sqrt T)}{2^{\frac 14} \sqrt \pi} 
\no{U_0}_{H^{2j-\frac 12}(\mathbb R)}, 
\quad j \in \mathbb N \cup \{0\}.
\end{equation}

Concerning the fractional seminorm \eqref{t-frac-def}, we once again split the $k$-integral near and away from zero to write
\begin{align}
\big\|\p_t^{[m]} U(x)\big\|_{\mu}^2
&\leq
4
\int_0^T\int_0^{T-t} z^{-(1+2\mu)}
\left|
\frac{1}{2\pi} \int_{|k|<1} \left|e^{-\omega z} - 1\right| |\omega|^{[m]} \left| \mathcal F\{U_0\}(k)\right| dk
\right|^2 dz dt
\label{Uf-n}
\\
&\quad
+
4
\int_0^T\int_0^{T-t} z^{-(1+2\mu)}
\left|
\frac{1}{2\pi} \int_{|k|>1} e^{-\nu k^2 t} \left|e^{-\omega z} - 1\right| |\omega|^{[m]} \left| \mathcal F\{U_0\}(k)\right| dk
\right|^2 dz dt.
\label{Uf-f}
\end{align}
The $k$-integral in \eqref{Uf-n} is over a bounded domain, so we apply the Cauchy-Schwarz inequality to get
\begin{equation}\label{zeta-bound00}
\eqref{Uf-n}
\leq
\frac{T}{\pi^2} \no{U_0}_{H^s(\mathbb R)}^2
\int_{|k|<1} 
\left(
\int_0^\infty \frac{\left|e^{-\omega z} - 1\right|^2}{z^{1+2\mu}} dz \right)  |\omega|^{2[m]} \left(1+k^2\right)^{-s}  dk.
\end{equation}
The dependence of the $z$-integral on $k$ can be extracted by means of the change of variable $\zeta = \nu k^2 z$:
\begin{equation}\label{zeta-bound0}
\int_0^\infty \frac{\left|e^{-\omega z} - 1\right|^2}{z^{1+2\mu}} dz
=
\left(\nu k^2\right)^{2\mu} 
\int_0^\infty \frac{\big|e^{-\left(1+i \frac\alpha \nu \right) \zeta} - 1\big|^2}{\zeta^{1+2\mu}} d\zeta.
\end{equation}
The $\zeta$-integral on the right side converges for all $\mu \in (0, 1)$. Indeed, by Lemma 2.1 of \cite{fhm2016}, 
$$
\big|e^{-\left(1+i \frac\alpha \nu \right) \zeta} - 1\big|
\leq
\sqrt 2 \left(1+ \frac \alpha \nu\right) \left(1-e^{-\zeta}\right),
\ \zeta \geq 0,
$$ 
therefore, making crucial use of the fact that $\mu \in (0, 1)$, we find
\begin{equation}\label{zeta-bound}
\int_0^\infty \frac{\big|e^{-\left(1+i \frac\alpha \nu \right) \zeta} - 1\big|^2}{\zeta^{1+2\mu}} d\zeta
\leq
2\left(1+ \frac \alpha \nu\right)^2 
\int_0^\infty \frac{\big(1-e^{-\zeta}\big)^2}{\zeta^{1+2\mu}} d\zeta
=
\frac{2\pi \left(2-2^{2\mu}\right) \csc(2\pi \mu) \left(1+ \frac \alpha \nu\right)^2}{\Gamma(1+2\mu)},
\end{equation}
where $\Gamma(\cdot)$ is the Gamma function.
Combining \eqref{zeta-bound0} and \eqref{zeta-bound} with \eqref{zeta-bound00}, and since $2m - s = \frac 12$, we obtain
\begin{align}
\eqref{Uf-n}
&\leq
\frac{2T \left(2-2^{2\mu}\right) \csc(2\pi \mu) \left(1 + \frac \alpha \nu\right)^2 \nu^{2\mu}}{\pi \Gamma(1+2\mu)} 
\no{U_0}_{H^s(\mathbb R)}^2
\int_{|k|<1}   \left(k^2\right)^{2\mu} \left(\alpha^2 + \nu^2\right)^{[m]} (k^2)^{2[m]}  \left(1+k^2\right)^{-s}  dk
\nn\\
&\leq
\frac{2T \left(2-2^{2\mu}\right) \csc(2\pi \mu) \left(1 + \frac \alpha \nu\right)^2 \left(1 + \frac{\alpha^2}{\nu^2}\right)^{[m]} \nu^{2m}}{\pi \Gamma(1+2\mu)} 
\no{U_0}_{H^s(\mathbb R)}^2
\int_{|k|<1}   \left(1+k^2\right)^{\frac 12}  dk
\nn\\
&\leq
\frac{4\sqrt 2 \, T \left(2-2^{2\mu}\right) \csc(2\pi \mu) \left(1 + \frac \alpha \nu\right)^2 \left(1 + \frac{\alpha^2}{\nu^2}\right)^{[m]} \nu^{2m}}{\pi \Gamma(1+2\mu)}
\no{U_0}_{H^s(\mathbb R)}^2.
\label{zeta-bound1}
\end{align}

Next, we estimate the term \eqref{Uf-f}, which requires a different approach since integration in $k$ takes place over an unbounded domain. Rearranging and making the change of variable $\tau = \nu k^2$ in the $k$-integral, we have
\begin{align*}
&\eqref{Uf-f}
=
4
\int_0^T\int_0^{T-t} z^{-(1+2\mu)}
\left|
\frac{1}{2\pi} \int_1^\infty e^{-\nu k^2 t} \left|e^{-\omega z} - 1\right| 
|\omega|^{[m]} \big( \left| \mathcal F\{U_0\}(k)\right| +  \left| \mathcal F\{U_0\}(-k)\right|\big) dk
\right|^2 dz dt
\\
&=
\frac{1}{\pi^2}
\int_0^T z^{-(1+2\mu)} 
\int_0^{T-z} 
\left|
\int_\nu^\infty e^{-\tau t} \left|e^{-\omega(k(\tau)) z} - 1\right| |\omega|^{[m]} \big( \left| \mathcal F\{U_0\}(k(\tau))\right| +  \left| \mathcal F\{U_0\}(-k(\tau))\right|\big) \frac{1}{2\nu k(\tau)} \, d\tau
\right|^2 dt dz
\\
&\leq
\frac{1}{\pi^2}
\int_0^T z^{-(1+2\mu)} 
\no{
\mathcal L\left\{\chi_{[\nu, \infty)}(\tau) \left|e^{-\omega(k(\tau)) z} - 1\right| |\omega|^{[m]} \big( \left| \mathcal F\{U_0\}(k(\tau))\right| +  \left| \mathcal F\{U_0\}(-k(\tau))\right|\big) \frac{1}{2\nu k(\tau)}\right\}}_{L_t^2(0, \infty)}^2 dz.
\end{align*}
Hence, by the boundedness of the Laplace transform in $L^2(0, \infty)$, 
$$
\eqref{Uf-f}
\leq
\frac{1}{\pi}
\int_0^T z^{-(1+2\mu)} 
\no{
\left|e^{-\omega(k(\tau)) z} - 1\right| |\omega|^{[m]} \big( \left| \mathcal F\{U_0\}(k(\tau))\right| +  \left| \mathcal F\{U_0\}(-k(\tau))\right|\big) \frac{1}{2\nu k(\tau)}}_{L_\tau^2(\nu, \infty)}^2 dz
$$
so, restoring $k$ in place of $\tau$ and rearranging, we have
\begin{align*}
&\eqref{Uf-f}
\leq
\frac{1}{2\nu \pi}
\int_1^\infty 
\left( \int_0^T \frac{\left|e^{-\omega z} - 1\right|^2}{z^{1+2\mu}} dz \right) \frac{\nu^{2[m]} \left(1 + \frac{\alpha^2}{\nu^2}\right)^{[m]} \left(k^2\right)^{2[m]}}{k} \big( \left| \mathcal F\{U_0\}(k)\right| +  \left| \mathcal F\{U_0\}(-k)\right|\big)^2 dk
\\
&\leq
\frac{\sqrt 2 \left(2-2^{2\mu}\right) \csc(2\pi \mu) \left(1 + \frac \alpha \nu\right)^2 \left(1 + \frac{\alpha^2}{\nu^2}\right)^{[m]} \nu^{2m-1}}{\Gamma(1+2\mu)}
\int_1^\infty \left(k^2\right)^{2(\mu+[m])} \left(1+k^2\right)^{-\frac 12} \big( \left| \mathcal F\{U_0\}(k)\right| +  \left| \mathcal F\{U_0\}(-k)\right|\big)^2 dk
\\
&\leq
\frac{2\sqrt 2 \left(2-2^{2\mu}\right) \csc(2\pi \mu) \left(1 + \frac \alpha \nu\right)^2 \left(1 + \frac{\alpha^2}{\nu^2}\right)^{[m]} \nu^{2m-1}}{\Gamma(1+2\mu)}
\int_{\mathbb R \setminus [-1, 1]} \left(k^2\right)^{2m} \left(1+k^2\right)^{-\frac 12} \left| \mathcal F\{U_0\}(k)\right|^2 dk
\end{align*}
with the second inequality due to combining \eqref{zeta-bound0} and \eqref{zeta-bound} with the fact that $k\geq 1$ implies $k^2 \geq \frac 12 \left(1+k^2\right)$. Therefore, since $2m - \frac 12 = s$,  
\begin{equation}\label{zeta-bound2}
\eqref{Uf-f}
\leq
\frac{2\sqrt 2 \left(2-2^{2\mu}\right) \csc(2\pi \mu) \left(1 + \frac \alpha \nu\right)^2 \left(1 + \frac{\alpha^2}{\nu^2}\right)^{[m]} \nu^{2m-1}}{\Gamma(1+2\mu)}
\no{U_0}_{H^s(\mathbb R)}^2.
\end{equation}
This bound combined with \eqref{zeta-bound1} implies the fractional seminorm estimate
\begin{equation}\label{frac-U}
\big\|\p_t^{[m]} U(x)\big\|_{\mu}^2
\leq
\frac{2\sqrt 2 \left(2 \nu T + \pi\right) \left(2-2^{2\mu}\right) \csc(2\pi \mu)  \left(1 + \frac \alpha \nu\right)^2 \left(1 + \frac{\alpha^2}{\nu^2}\right)^{[m]} \nu^{2m}}{\pi \nu \Gamma(1+2\mu)}
\no{U_0}_{H^s(\mathbb R)}^2
\end{equation}
for $s\geq - \frac 12$ and $\mu \in (0, 1)$, 
which is in turn combined with the $L^2$-estimate \eqref{Uj-te} to yield, via \eqref{hm-l2}, the desired estimate \eqref{U-ivp-te} with constant
%
\begin{multline}\label{c2-def}
c_2(s, \alpha, \nu, T) := 
c(m, T) \bigg(\frac{\left(1-\nu^{[m]+1}\right) (1+\sqrt T)}{2^{\frac 14} \sqrt{\nu \pi} \left(1-\sqrt \nu\right)} \\
+
\frac{2^{\frac 34} \sqrt{\pi + 2 \nu T} \sqrt{\left(2-2^{2\mu}\right) \csc(2\pi \mu)}  \left(1 + \frac \alpha \nu\right) \left(1 + \frac{\alpha^2}{\nu^2}\right)^{\frac{[m]}{2}}  \nu^{m} \chi_{(0, 1)}(\mu)}{\sqrt{\pi \nu \Gamma(1+2\mu)}}\bigg),
\end{multline}
where $c(m, T)$ is the constant in \eqref{hm-l2} and we recall that $m = \frac{2s+1}{4}$, $\mu = m-[m]$. We note that, although $c(m, T)$ blows up in the limit $T\to 0^+$, which is relevant for the contraction mapping argument carried out later in the nonlinear analysis, estimate \eqref{U-ivp-te} will be employed with $T$ replaced by $T+1$ and so the relevant constant will then remain bounded as $T\to 0^+$.

Furthermore, observing that $U_x$ satisfies the Cauchy problem 
\begin{equation}\label{Ux-ivp}
\begin{aligned}
&(U_x)_t - \left(\nu + i \alpha \right) (U_x)_{xx} = 0,
\quad
x\in\mathbb R, \ t\in(0,T),
\\
&U_x(x,0) = U_0'(x),
\end{aligned} 
\end{equation}
we employ estimate \eqref{U-ivp-te} with $U_x$, $U_0'$ and $s-1$ in place of $U$, $U_0$ and $s$, respectively, to deduce
$$
\no{U_x(x)}_{H_t^{\frac{2(s-1)+1}{4}}(0, T)} \leq c_2(s-1, \alpha, \nu, T) \no{U_0'}_{H^{s-1}(\mathbb R)}, \ x\in \mathbb R, \ s-1\geq -\frac 12,
$$
which readily yields the second time estimate \eqref{Ux-ivp-te} in view of the fact that $\no{U_0'}_{H^{s-1}(\mathbb R)} \leq \no{U_0}_{H^s(\mathbb R)}$.

Finally, the continuity of the relevant maps follows via the dominated convergence theorem and estimations entirely analogous to the ones provided above.
\end{proof}

\begin{remark}[Fractional seminorm vs. interpolation]
Concerning the proof of estimates \eqref{U-ivp-te} and \eqref{Ux-ivp-te} for fractional values of $\frac{2s+1}{4}$, an alternative to estimating the fractional Sobolev-Slobodeckij-Gagliardo seminorm is to use interpolation (Theorem 5.1 in \cite{lm1972} or Theorem 4.1.2 in \cite{bl1976}) between the whole derivative estimates~\eqref{Uj-te}, similarly to what we did in the proof of Theorem \ref{pure-ibvp-t}. However, in the interest of presenting both approaches, in the proof of Theorem \ref{U-ivp-t} above, as well as of its inhomogeneous counterpart Theorem \ref{W-ivp-t} below, we have chosen to estimate the fractional seminorm explicitly.
\end{remark}

\subsection{Estimates for the forced linear Cauchy problem}
The next result concerns the forced linear Cauchy problem \eqref{W-ivp} and provides the final piece needed for estimating the initial-boundary value problem \eqref{lcgl-ibvp-w} via the superposition \eqref{lin-sup}.
\begin{theorem}[Forced linear Cauchy problem]\label{W-ivp-t}
The solution $W = S\big[0; F\big]$ to the Cauchy problem \eqref{W-ivp} satisfies the space estimate
\begin{equation}\label{W-ivp-se}
\sup_{t\in [0, T]} \no{W(t)}_{H_x^s(\mathbb R)} \leq  \no{F}_{L_t^1((0, T); H_x^s(\mathbb R))}, \quad s\in \mathbb R,
\end{equation}
and  the time estimates
\begin{align}\label{W-ivp-te}
\sup_{x\in\mathbb R} \no{W(x)}_{H_t^{\frac{2s+1}{4}}(0, T)} &\leq c_3(s) 
\no{F}_{L_t^2((0, T); H_x^s(\mathbb R))}, \quad -\frac 12 \leq s \leq \frac 32, \ s \neq \frac 12,  
\\ 
 \label{W-ivp-te2}
\sup_{x\in\mathbb R} \no{W(x)}_{H_t^{\frac{2s+1}{4}}(0, T)} &\leq 
c_4(s)
\left(\no{F}_{L_t^2((0, T); H_x^s(\mathbb R))}
 +
\no{F(x)}_{W_t^{\frac{2s+1}{4}-1}(0, T)}
\right), 
\quad
\frac 32 < s \leq \frac 72, \ s \neq \frac 52, 
\\
\label{Wx-ivp-te}
\sup_{x\in\mathbb R} \no{W_x(x)}_{H_t^{\frac{2s-1}{4}}(0, T)} &\leq c_3(s-1) \no{F}_{L_t^2((0, T); H_x^s(\mathbb R))}, \quad \frac 12 \leq s \leq \frac 52, \ s \neq \frac 32, 
\end{align}
with the constants $c_3(s) = c_3(s, \alpha, \nu, T)$ and $c_4(s) = c_4(s, \alpha, \nu, T)$ given by \eqref{c3-def} and \eqref{c4-def}.
Furthermore, the maps $[0, T] \ni t \mapsto W(t) \in H_x^s(\mathbb R)$ as well as $\mathbb R \ni x \mapsto W(x) \in H_t^{\frac{2s+1}{4}}(0, T)$ and $\mathbb R \ni x \mapsto W_x(x) \in H_t^{\frac{2s-1}{4}}(0, T)$ are continuous.
\end{theorem}
\begin{proof}
The solution to problem \eqref{W-ivp}  is given by
\begin{equation}\label{W-sol}
W(x, t) = \frac{1}{2\pi} \int_0^t \int_{\mathbb R} e^{ikx-\omega(t-t')} \mathcal F\{F\}(k, t') dk dt'
\end{equation}
with $\mathcal F\{F\}$ denoting the spatial Fourier transform of the forcing $F$. Thus, by Minkowski's integral inequality,
$$
\no{W(t)}_{H_x^s(\mathbb R)} \leq \int_0^t \no{\frac{1}{2\pi}\int_{\mathbb R} e^{ikx-\omega(t-t')} \mathcal F\{F\}(k, t') dk}_{H_x^s(\mathbb R)} dt'
$$
and, using the Fourier characterization of the $H_x^s(\mathbb R)$ norm along with the fact that $\text{Re}(\omega) = \nu k^2 \geq 0$ for $k\in\mathbb R$,
$$
\no{W(t)}_{H_x^s(\mathbb R)} \leq  \int_0^t \left(\int_{\mathbb R} \left(1+k^2\right)^s \left| e^{-\omega(t-t')} \mathcal F\{F\}(k, t')\right|^2 dk\right)^{\frac 12} dt'
\leq
\no{F}_{L_t^1((0, T); H_x^s(\mathbb R))},
$$
which is the space estimate \eqref{W-ivp-se}.

Concerning the time estimates \eqref{W-ivp-te}-\eqref{Wx-ivp-te}, we begin by noting that \eqref{W-ivp-te} with $W$, $F$, $s$ replaced respectively by $W_x$, $F_x$, $s-1$  amounts to
$$
\no{W_x(x)}_{H_t^{\frac{2(s-1)+1}{4}}(0, T)} \leq c_3(s-1) \no{F_x}_{L_t^2((0, T); H_x^{s-1}(\mathbb R))}, \ -\frac 12 \leq s-1 \leq \frac 32, \ s-1 \neq \frac 12, \ x\in\mathbb R,
$$
which readily implies \eqref{Wx-ivp-te} since $\no{F_x(t)}_{H_x^{s-1}(\mathbb R)} \leq \no{F(t)}_{H_x^s(\mathbb R)}$. 
Hence, it suffices to prove \eqref{W-ivp-te} and \eqref{W-ivp-te2}. 
To this end, we hereafter restrict $s\geq -\frac 12$ in order to use the $L^2$ characterization of the Sobolev norm given by \eqref{hm-l2}. 
First, according to Duhamel's principle, we express the solution $W = S\big[0; F\big]$ to the inhomogeneous problem~\eqref{W-ivp} in terms of the solution $U = S\big[U_0; 0\big]$ to the homogeneous problem \eqref{U-ivp} as 
\begin{equation}\label{duh}
W(x, t) = \int_0^t S\big[F(\cdot, t'); 0\big](x, t-t') dt'.
\end{equation}
\textbf{The case of $s=-\frac 12 \Leftrightarrow \frac{2s+1}{4} = 0$.} For each $x\in\mathbb R$, 
\begin{equation}\label{erf-3}
\no{W(x)}_{L_t^2(0, T)}^2
\leq
\int_0^T \left( \int_0^t \left|S\big[F(\cdot, t'); 0\big](x, t-t')\right| dt'\right)^2 dt
\end{equation}
and so, by the Cauchy-Schwarz inequality in $t'$, 
\begin{align}
\no{W(x)}_{L_t^2(0, T)}^2
&\leq
\int_0^T t \left( \int_0^t \left|S\big[F(\cdot, t'); 0\big](x, t-t')\right|^2 dt'\right) dt
\nn\\
&\leq
T \int_0^T   \int_{t'}^T \left|S\big[F(\cdot, t'); 0\big](x, t-t')\right|^2 dt  dt'
\nn
\end{align}
after bounding $t$ by $T$ and changing the order of integration. Letting $t = t'+\rho$, we then have
\begin{equation}\label{erf-2}
\no{W(x)}_{L_t^2(0, T)}^2
\leq
T \int_0^T   \int_{0}^{T-t'} \left|S\big[F(\cdot, t'); 0\big](x, \rho)\right|^2 d\rho  dt'
=
T  \int_0^T \no{S\big[F(\cdot, t'); 0\big](x, \rho)}_{L_\rho^2(0, T-t')}^2 dt'.
\end{equation}
We estimate the $L_\rho^2(0, T-t')$ norm on the right side by proceeding as in the proof of the homogeneous time estimate \eqref{U-ivp-te}. In particular, we use formula \eqref{U-sol} and the triangle inequality to write
\begin{align}
\no{S\big[F(\cdot, t'); 0\big] (x, \rho)}_{L_\rho^2(0, T-t')}
&\leq
\frac{1}{2\pi} 
\no{
\int_{|k|>1} e^{-\nu k^2 \rho} \left|\mathcal F\{F\}(k, t')\right| dk}_{L_\rho^2(0, T-t')}
\label{WU-te1}
\\
&\quad
+
\frac{1}{2\pi} 
\no{
\int_{|k|<1} e^{-\nu k^2 \rho} \left|\mathcal F\{F\}(k, t')\right| dk}_{L_\rho^2(0, T-t')}.
\label{WU-te2}
\end{align}

For the term \eqref{WU-te1}, making the changes of variables $k=-\sqrt{\frac \tau \nu}$ for $k<0$ and $k= \sqrt{\frac \tau \nu}$ for $k>0$, both of which imply that $|\omega| = \nu k^2 = \tau$, we find
\begin{equation}\label{erf-z}
\begin{aligned}
\eqref{WU-te1}
&\leq
\frac{1}{4\pi \sqrt \nu} 
\no{
\int_\nu^\infty e^{-\tau \rho} \cdot \tau^{-\frac 12} \left(
\big|\mathcal F\{F\}(-\sqrt{\tfrac \tau \nu}, t')\big|
+
\big|\mathcal F\{F\}(\sqrt{\tfrac \tau \nu}, t')\big|
\right)
 d\tau
}_{L_\rho^2(0, T-t')}
\\
&\leq
\frac{1}{4\pi \sqrt \nu} 
\no{
\int_0^\infty e^{-\tau \rho} \cdot \chi_{[\nu, \infty)}(\tau) \tau^{-\frac 12} \left(
\big|\mathcal F\{F\}(-\sqrt{\tfrac \tau \nu}, t')\big|
+
\big|\mathcal F\{F\}(\sqrt{\tfrac \tau \nu}, t')\big|
\right)
 d\tau
}_{L_\rho^2(0, \infty)}.
\end{aligned}
\end{equation}
Thus, by the boundedness of the Laplace transform in $L^2(0, \infty)$, 
$$
\eqref{WU-te1}
\leq
\frac{1}{4\sqrt{\pi \nu}} 
\no{\tau^{-\frac 12} \left(
\big|\mathcal F\{F\}(-\sqrt{\tfrac \tau \nu}, t')\big|
+
\big|\mathcal F\{U_0\}(\sqrt{\tfrac \tau \nu}, t')\big|
\right)
}_{L_\tau^2(\nu, \infty)}.
$$
Changing variable from $\tau$ back to $k$ and using the fact that  $k^2 \simeq 1+k^2$ for $|k|>1$, we obtain
\begin{equation}\label{WU-te3}
\eqref{WU-te1}
\leq
\frac{1}{2\sqrt{\pi \nu}} 
\no{(k^2)^{-\frac 14} 
\big|\mathcal F\{F\}(k, t')\big|
}_{L_k^2(\mathbb R \setminus [-1, 1])}
\leq 
\frac{1}{2\sqrt{\pi \nu}} \no{F(t')}_{H_x^{-\frac 12}(\mathbb R)}.
\end{equation}

Also, by Minkowski's integral inequality and the fact that $e^{-\nu k^2 \rho} \leq 1$,
$$
\eqref{WU-te2}
\leq
\frac{1}{2\pi} 
\int_{|k|<1} \no{
e^{-\nu k^2 \rho} \left|\mathcal F\{F\}(k, t')\right|}_{L_\rho^2(0, T-t')} dk
\leq
\frac{\sqrt{T-t'}}{2\pi} 
\int_{|k|<1}  \left|\mathcal F\{F\}(k, t')\right| dk
$$
hence by the Cauchy-Schwarz inequality we obtain
$$
\eqref{WU-te2}
\leq
\frac{\sqrt{T-t'}}{2\pi} 
\left(\int_{|k|<1}  \left(1+k^2\right)^{\frac 12} dk\right)^{\frac 12} 
\left(\int_{|k|<1}  \left(1+k^2\right)^{-\frac 12}\left|\mathcal F\{F\}(k, t')\right|^2 dk\right)^{\frac 12}
\leq
\frac{\sqrt{T-t'}}{2^{\frac 14} \pi} 
\no{F(t')}_{H_x^{-\frac 12}(\mathbb R)}.
$$
Thus, returning to \eqref{WU-te1}-\eqref{WU-te2}, we deduce
\begin{equation}\label{WUj-te}
\no{S\big[F(\cdot, t'); 0\big] (x, \rho)}_{L_\rho^2(0, T-t')}
\leq
\frac{\left(1+ \sqrt \nu\, \right) \left(1+\sqrt{T-t'} \, \right)}{2^{\frac 14} \sqrt{\pi \nu}} 
\no{F(t')}_{H_x^{-\frac 12}(\mathbb R)}.
\end{equation}
In turn, \eqref{erf-2} yields
\begin{equation}\label{W-l2}
\no{W(x)}_{L_t^2(0, T)}^2
\leq
\frac{T\left(1+ \sqrt \nu\, \right)^2 \big(1+\sqrt T \, \big)^2}{\sqrt 2 \pi \nu} 
\no{F}_{L_t^2((0, T); H_x^{-\frac 12}(\mathbb R))}^2.
\end{equation}
\textbf{The case of $s=\frac 32 \Leftrightarrow \frac{2s+1}{4} = 1$.} Using Leibniz's rule to differentiate \eqref{W-sol} with respect to $t$, we have
\begin{align}\label{wt}
W_t(x, t) 
&= 
\frac{1}{2\pi} \int_{\mathbb R} e^{ikx} \mathcal F\{F\}(k, t) dk
+
\int_0^t \frac{1}{2\pi} \int_{\mathbb R} e^{ikx-\omega(t-t')} (-\omega) \mathcal F\{F\}(k, t') dk dt'
\nn\\
&= 
F(x, t)
+
\left(\nu+i\alpha\right)
\int_0^t S\big[F_{xx}(\cdot, t'); 0\big](x, t-t') dt'
\end{align}
with the second equality after recalling that $\omega = \left(\nu+i\alpha\right) k^2$. Thus, by the triangle inequality, 
$$
\no{W_t(x)}_{L_t^2(0, T)}
\leq
\no{F(x)}_{L_t^2(0, T)} 
+ 
\sqrt{\alpha^2 + \nu^2} \no{\int_0^t S\big[F_{xx}(\cdot, t'); 0\big](x, t-t') dt'}_{L_t^2(0, T)}.
$$
Handling the second norm on the right side like \eqref{erf-3}, i.e. via estimate \eqref{W-l2} but now with $F_{xx}$ in place of $F$, we obtain
\begin{equation}\label{W-te1}
\no{W_t(x)}_{L_t^2(0, T)}
\leq
\no{F(x)}_{L_t^2(0, T)}
+
\frac{\sqrt T\left(1+ \sqrt \nu\, \right) \big(1+\sqrt T \, \big) \sqrt{\alpha^2 + \nu^2}}{2^{\frac 14} \sqrt{\pi \nu}} 
\no{F_{xx}}_{L_t^2((0, T); H_x^{-\frac 12}(\mathbb R))}.
\end{equation}
In addition,  by the Sobolev embedding theorem, 
\begin{equation}\label{W-te2}
\no{F(x)}_{L_t^2(0, T)}^2
\leq
\int_0^T \no{F(t)}_{L_x^\infty(\mathbb R)}^2 dt 
\leq
\int_0^T \no{F(t)}_{H_x^{\frac 12+}(\mathbb R)}^2 dt 
=
\no{F}_{L_t^2((0, T); H_x^{\frac 12+}(\mathbb R))}^2.
\end{equation}
Combining \eqref{W-te1}, \eqref{W-te2}
and \eqref{W-l2} with the definition  \eqref{hm-l2} we deduce
%
\begin{equation}\label{W-h1}
\no{W(x)}_{H_t^1(0, T)}
\leq
c(T)
\left(1 + \frac{\sqrt T\left(1+ \sqrt \nu\, \right) \big(1+\sqrt T \, \big) \left(1+ \sqrt{\alpha^2 + \nu^2}\right)}{2^{\frac 14} \sqrt{\pi \nu}}  \right)
\no{F}_{L_t^2((0, T); H_x^{\frac 32}(\mathbb R))},
 \quad
 x\in\mathbb R.
\end{equation}
\textbf{The case of $s=\frac 72 \Leftrightarrow \frac{2s+1}{4} = 2$.} 
Differentiating \eqref{wt} once more with respect to $t$, we have
\begin{align}\label{wtt}
W_{tt}(x, t) 
&= 
F_t(x, t) + \left(\nu+i\alpha\right) S\big[F_{xx}(\cdot, t); 0\big](x, 0)
+
\left(\nu+i\alpha\right)^2 \int_0^t S\big[\p_x^4 F(\cdot, t'); 0\big](x, t-t') dt'.
\end{align}
Therefore, by the triangle inequality and the fact that $S\big[F_{xx}(\cdot, t); 0\big](x, 0) = F_{xx}(x, t)$, we find
$$
\no{W_{tt}(x)}_{L_t^2(0, T)}
\leq 
\no{F_t(x)}_{L_t^2(0, T)} 
+
\sqrt{\alpha^2+\nu^2} \no{F_{xx}(x)}_{L_t^2(0, T)}
+
\left(\alpha^2+\nu^2\right) \int_0^t \no{S\big[\p_x^4 F(\cdot, t'); 0\big](x, t-t')}_{L_t^2(0, T)} dt'.
$$
The second term can be handled analogously to \eqref{W-te2} and the third term like \eqref{erf-3}, i.e. via estimate \eqref{W-l2} but with $\p_x^4 F$ in place of $F$. Hence, we further obtain
\begin{align}
\no{W_{tt}(x)}_{L_t^2(0, T)}
&\leq
\no{F_t(x)}_{L_t^2(0, T)} 
+
\sqrt{\alpha^2+\nu^2}  
\no{F_{xx}}_{L_t^2((0, T); H_x^{\frac 12+}(\mathbb R))}
\nn\\
&\quad
+
\frac{\sqrt T\left(1+ \sqrt \nu\, \right) \big(1+\sqrt T \, \big)\left(\alpha^2+\nu^2\right)}{2^{\frac 14} \sqrt{\pi \nu}} 
\no{\p_x^4 F}_{L_t^2((0, T); H_x^{-\frac 12}(\mathbb R))}
\nn\\
&\leq
\no{F_t(x)}_{L_t^2(0, T)} 
+
\sqrt{\alpha^2+\nu^2}\, \no{F}_{L_t^2((0, T); H_x^{\frac 52+}(\mathbb R))}
\nn\\
&\quad
+
\frac{\sqrt T\left(1+ \sqrt \nu\, \right) \big(1+\sqrt T \, \big)\left(\alpha^2+\nu^2\right)}{2^{\frac 14} \sqrt{\pi \nu}} 
\no{F}_{L_t^2((0, T); H_x^{\frac 72}(\mathbb R))},
\label{wtt-est}
\end{align}
which combines with the definition \eqref{hm-l2} and the $H^1$ estimate \eqref{W-h1} to yield
\begin{align}\label{W-h2}
\no{W(x)}_{H_t^2(0, T)}
&\leq
c(T) \left(1 + \frac{\sqrt T\left(1+ \sqrt \nu\, \right) \big(1+\sqrt T \, \big) \left(1+ 2\sqrt{\alpha^2 + \nu^2}\right)}{2^{\frac 14} \sqrt{\pi \nu}} + \sqrt{\alpha^2+\nu^2} \right)
\no{F}_{L_t^2((0, T); H_x^s(\mathbb R))}
 \nn\\
 &\quad
 +
c(T) \no{F(x)}_{W_t^1(0, T)}
\end{align}
where $c(T)=c(2, T)$ with $c(m, T)$  as in   \eqref{hm-l2}.

Estimates \eqref{W-l2}, \eqref{W-h1} and \eqref{W-h2} provide the desired result \eqref{W-ivp-te} in the cases of $s=-\frac 12$, $s=\frac 32$ and $s=\frac 72$ respectively, which for $m = \frac{2s+1}{4}$ correspond to the integer values $m=0$, $m=1$ and $m=2$ respectively. In order to extend the validity of \eqref{W-ivp-te} to $s\in \left(-\frac 12, \frac 72\right)\setminus \{\frac 12, \frac 52\}$, which corresponds to $m \in (0, 2)\setminus \{\frac 12, \frac 32\}$, we will additionally estimate the Sobolev-Slobodeckij-Gagliardo seminorm  \eqref{t-frac-def}. 
%
\\[2mm]
\textbf{The case of $s \in \left(-\frac 12, \frac 32\right) \setminus \left\{\frac 12\right\}$.} Then, $m = \frac{2s+1}{4} \in \left(0,  1\right) \setminus \left\{\frac 12\right\}$ and so $[m]=0$. Thus, according to \eqref{hm-l2}, we need to estimate $\no{W(x)}_\mu$ with $\mu \in (0, 1) \setminus \left\{\frac 12\right\}$. By \eqref{t-frac-def} and the Duhamel representation~\eqref{duh}, we have  
\begin{align}
\no{W(x)}_\mu^2
&\leq
4 \int_0^T \int_0^{T-t}
z^{-(1+2\mu)} \left| \int_0^t 
\Big\{
S\big[F(\cdot, t'); 0\big] (x, t+z-t') 
-
S\big[F(\cdot, t'); 0\big] (x, t-t')
\Big\} dt' \right|^2 dz dt
\label{W-ivp-te-b1}
\\
&\quad
+ 4
\int_0^T \int_0^{T-t}
z^{-(1+2\mu)} \left|\displaystyle \int_t^{t+z} S\big[F(\cdot, t'); 0\big] (x, t+z-t') dt'\right|^2   dz dt.
\label{W-ivp-te-b2}
\end{align}
For the first term, we apply the Cauchy-Schwarz inequality in the $t'$-integral and then bound $t$ by $T$ to obtain
$$
\eqref{W-ivp-te-b1}
\leq
4 T \int_0^T \int_0^{T-t}
z^{-(1+2\mu)} \int_0^t 
\left| 
S\big[F(\cdot, t'); 0\big] (x, t+z-t') 
-
S\big[F(\cdot, t'); 0\big] (x, t-t')
 \right|^2 dt' dz dt.
$$
Thus, interchanging the $z$-integral with the $t'$-integral and then the $t'$-integral with the $t$-integral, we have
\begin{align}
\eqref{W-ivp-te-b1}
&\leq
4 T \int_0^T  \int_{t'}^T 
\int_0^{T-t}
\frac{\left| 
S\big[F(\cdot, t'); 0\big] (x, t+z-t') 
-
S\big[F(\cdot, t'); 0\big] (x, t-t')
 \right|^2}{z^{1+2\mu}} dz dt  dt'
\nn\\
&=
4 T \int_0^T  \int_{0}^{T-t'} 
\int_0^{T-\rho-t'}
\frac{\left| 
S\big[F(\cdot, t'); 0\big] (x, \rho + z) 
-
S\big[F(\cdot, t'); 0\big] (x, \rho)
 \right|^2}{z^{1+2\mu}} dz d\rho  dt'
 \nn
\end{align}
with last equality due to the change of variable $t = \rho + t'$. Therefore, augmenting the $z\rho$-region by exploiting the fact that $S\big[F(\cdot, t'); 0\big] (x, \rho)$ makes sense for any $\rho\geq 0$, we find
\begin{align}
\eqref{W-ivp-te-b1}
&\leq
4 T \int_0^T  \int_{0}^{T} 
\int_0^{T-\rho}
\frac{\left| 
S\big[F(\cdot, t'); 0\big] (x, \rho + z) 
-
S\big[F(\cdot, t'); 0\big] (x, \rho)
 \right|^2}{z^{1+2\mu}} dz d\rho  dt'
\nn\\
&=
2 T  \int_0^T \no{S\big[F(\cdot, t'); 0\big] (x, \rho)}_{\mu_\rho}^2 dt'
\leq
2 T  \int_0^T \no{S\big[F(\cdot, t'); 0\big] (x, \rho)}_{H_\rho^{\frac{2s+1}{4}}(0, T)}^2 dt'
\label{wtemp1}
\end{align}
which allows us to conclude via the homogeneous time estimate \eqref{U-ivp-te} that
\begin{equation}
\eqref{W-ivp-te-b1}
\leq
2T c_2(s, \alpha, \nu, T)^2
\no{F}_{L_t^2((0, T); H_x^s(\mathbb R))}^2.
\label{W-ivp-te-b1-est}
\end{equation}

The estimation of the term \eqref{W-ivp-te-b2} is more involved and requires us to treat the cases $\mu>\frac 12$ and $\mu< \frac 12$ separately. 
For $\mu>\frac 12$, we use formula \eqref{U-sol} and the fact that $\text{Re}(\omega) \geq 0$ for $k\in\mathbb R$ to infer
\begin{align}
\eqref{W-ivp-te-b2}
&=
\frac{1}{\pi^2} 
\int_0^T \int_0^{T-t} z^{-(1+2\mu)}
 \left|\displaystyle \int_t^{t+z} 
 \int_{\mathbb R} e^{ik x-\omega(t+z-t')}
\mathcal F\{F\}(k, t') dk
 dt'\right|^2  dz dt
\nn\\
&\leq 
\frac{1}{\pi^2} 
\int_0^T \int_0^{T-t} z^{-(1+2\mu)}
 \left(
 \int_t^{t+z} 
 \int_{\mathbb R} \left|\mathcal F\{F\}(k, t')\right| dk
 dt'\right)^2  dz dt.
 \nn
\end{align}
Hence, by the Cauchy-Schwarz inequality in $k$ and $t'$, Fubini's theorem, and the fact that $\mu>\frac 12$ corresponds to $s>\frac 12$ and so $\int_{k \in \mathbb R} \left(1+k^2\right)^{-s}  dk = c_s < \infty$,  we obtain
\begin{align}
\eqref{W-ivp-te-b2}
&\leq
\frac{c_s}{\pi^2} 
\int_0^T \int_0^{T-t} z^{-(1+2\mu)}
 \left(  \int_t^{t+z} \no{F(t')}_{H_x^s(\mathbb R)}  dt'\right)^2  dz dt
\nn\\
& \leq
\frac{c_s}{\pi^2} 
\int_0^T \int_0^{T-t} z^{-2\mu}
 \left(  \int_t^{t+z} \no{F(t')}_{H_x^s(\mathbb R)}^2  dt'\right)  dz dt
 \nn\\
& =
\frac{c_s}{\pi^2}
\int_0^T \no{F(t')}_{H_x^s(\mathbb R)}^2 \int_0^T z^{-2\mu}
 \int_{t'-z}^{t'} dt  dz dt'
\nn\\
& =
\frac{c_s}{\pi^2}
\no{F}_{L_t^2((0, T); H_x^s(\mathbb R))}^2 \int_0^T z^{1-2\mu} dz
 =
\frac{c_s}{\pi^2} \, \frac{T^{2(1-\mu)}}{2(1-\mu)} \no{F}_{L_t^2((0, T); H_x^s(\mathbb R))}^2.
\label{W-ivp-te-b2-est}
\end{align}

If $\mu<\frac 12$, then applying the Cauchy-Schwarz inequality in $t'$ and then interchanging the integrals with respect to $t$ and $z$, we find
\begin{align*} 
\eqref{W-ivp-te-b2}
&\leq
4 \int_0^{T} z^{-2\mu}  \int_0^{T-z} 
 \int_t^{t+z} \left|S\big[F(\cdot, t'); 0\big] (x, t+z-t')\right|^2 dt'  dt dz
\nn\\
&=
4 \int_0^{T} z^{-2\mu}  \int_{z}^T 
 \int_{r-z}^{r} \left|S\big[F(\cdot, t'); 0\big] (x, r-t')\right|^2 dt'  dr dz
\end{align*}
after making the change of variable $t = r - z$. Next, for each $z \in [0, T]$, as shown in  Fig. \ref{regions-f} we have
\begin{align}
\int_{r=z}^T \int_{t'=r-z}^r
&\leq 
\int_{t'=0}^{T-z} \int_{r=t'}^{t'+z}
+
\int_{t'=T-z}^{T} \int_{r=t'}^{T}
\label{regions-ineq}
\\
&\leq 
\int_{t'=0}^{T} \int_{r=t'}^{t'+z}
+
\int_{t'=T-z}^{T} \int_{r=t'}^{T}
=
\int_{t'=0}^{T} \int_{\rho=0}^{z}
+
\int_{t'=T-z}^{T} \int_{\rho=0}^{T-t'}
\nn
\end{align}
where for the last equality we have let $r = \rho + t'$. Hence, 
\begin{align} 
\eqref{W-ivp-te-b2}
&\leq
4 \int_0^{T} z^{-2\mu}  \int_{0}^T 
 \int_{0}^{z} \left|S\big[F(\cdot, t'); 0\big] (x, \rho)\right|^2 d\rho  dt' dz
\nn\\
&\quad
+
4 \int_0^{T} z^{-2\mu}  \int_{T-z}^T 
 \int_{0}^{T-t'} \left|S\big[F(\cdot, t'); 0\big] (x, \rho)\right|^2 d\rho dt' dz 
\nn\\
&\leq
8 \int_0^{T} z^{-2\mu} \int_{0}^{T} \no{S\big[F(\cdot, t'); 0\big] (x, \rho)}_{L_\rho^2(0, T)}^2 dt'  dz
\end{align}
with the last inequality after augmenting the domains of the two $\rho$-integrals and of the second $t'$-integral.
Estimating the $L_\rho^2(0, T)$ norm via \eqref{WUj-te} with $T$ in place of $T-t'$, respectively, we obtain
%
\begin{equation}
\eqref{W-ivp-te-b2}
\leq
\frac{8 \left(1+ \sqrt \nu\, \right)^2 \big(1+\sqrt{T} \, \big)^2}{\sqrt 2 \pi \nu}
\int_0^{T} z^{-2\mu} \int_{0}^{T}  \no{F(t')}_{H_x^{-\frac 12}(\mathbb R)}^2 dt'  dz
\end{equation}
so integrating in $z$ (recall that $\mu<\frac 12$) we infer
\begin{align}\label{W-ivp-te-low}
\eqref{W-ivp-te-b2}
&\leq
\frac{8\left(1+ \sqrt \nu\, \right)^2\big(1+ \sqrt T \, \big)^2 \, T^{1-2\mu}}{\sqrt 2 \pi \nu \left(1-2\mu\right)} \no{F}_{L_t^2((0, T); H_x^{-\frac 12}(\mathbb R))}^2.
\end{align}
In view of estimates \eqref{W-ivp-te-b1-est}, \eqref{W-ivp-te-b2-est} and \eqref{W-ivp-te-low}, inequalities \eqref{W-ivp-te-b1} and \eqref{W-ivp-te-b2} yield
\begin{align}\label{W-ivp-te-b-est}
\no{W(x)}_\mu
&\leq
\Bigg(
2\sqrt T c_2(s, \alpha, \nu, T)^2
+
\frac{c_s \, T^{2(1-\mu)} \chi_{(\frac 12, 1)}(\mu)}{2\pi^2(1-\mu)}
+
\frac{8\left(1+ \sqrt \nu\, \right)^2\big(1+ \sqrt T \, \big)^2 \, T^{1-2\mu}\chi_{(0, \frac 12)}(\mu)}{\sqrt 2 \pi \nu \left(1-2\mu\right)}
\Bigg)^{\frac 12}
\nn\\
&\qquad
\cdot \no{F}_{L_t^2((0, T); H_x^s(\mathbb R))},
\quad 
-\frac 12 < s < \frac 32, \ s\neq \frac 12. 
\end{align}

Returning to the definition \eqref{hm-l2} and combining the estimates \eqref{W-l2}, \eqref{W-te1} and \eqref{W-te2} with the fractional estimate~\eqref{W-ivp-te-b-est}, we obtain the time estimate \eqref{W-ivp-te} for all $s \in \left[-\frac 12, \frac 32\right] \setminus \left\{\frac 12\right\}$ with associated constant 
\begin{equation}\label{c3-def}
\begin{aligned}
c_3(s, \alpha, \nu, T) 
&:= 
c(m, T)
\Bigg[
1 + \frac{\sqrt T\left(1+ \sqrt \nu\, \right) \big(1+\sqrt T \, \big) \left(1+ \sqrt{\alpha^2 + \nu^2}\right)}{2^{\frac 14} \sqrt{\pi \nu}} 
\\
&\hskip -1cm
+ 
\Bigg(
2\sqrt T c_2(s, \alpha, \nu, T)^2
+
\frac{c_s \, T^{2(1-\mu)} \chi_{(\frac 12, 1)}(\mu)}{2\pi^2(1-\mu)}
+
\frac{8\left(1+ \sqrt \nu\, \right)^2\big(1+ \sqrt T \, \big)^2 \, T^{1-2\mu}\chi_{(0, \frac 12)}(\mu)}{\sqrt 2 \pi \nu \left(1-2\mu\right)}
\Bigg)^{\frac 12}\Bigg],
\end{aligned}
\end{equation}
where $c(m, T)$ is the constant in \eqref{hm-l2}, $c_2(s, \alpha, \nu, T)$ is given by \eqref{c2-def}, and we recall that $m = \frac{2s+1}{4}$, $\mu = m - [m]$.
We remark that, although $c(m, T)$ blows up as $T\to 0^+$, this will not be a problem in the contraction mapping argument for which the limit $T\to 0^+$ is relevant, as then we will employ estimate \eqref{W-ivp-te} with $T+1$ in place of $T$. 

\begin{figure}[h!]
\centering
\includegraphics[scale=0.5]{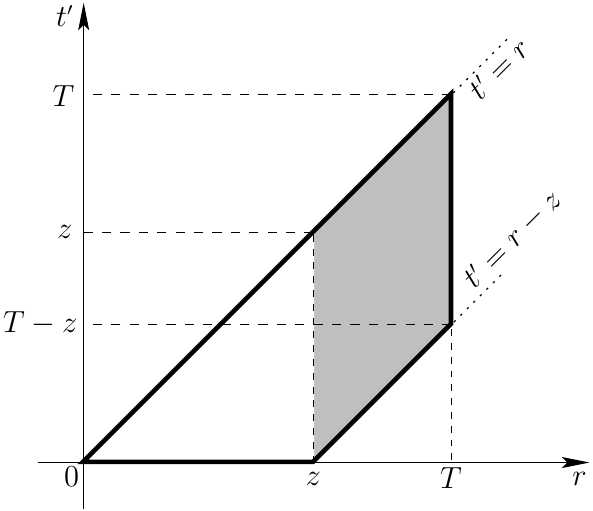}
\caption{
The regions of integration corresponding to the left side (shaded) and the right side (enclosed in bold) of inequality \eqref{regions-ineq}.
}
\label{regions-f}
\end{figure}

\vskip 2mm
\noindent
\textbf{The case of $m\in(1, 2)$ with $m\neq \frac 32$.}
Recalling \eqref{wt}, we compute
\begin{align}
W_t(x,z+t) - W_t(x,t) 
&=
F(x,z+t) - F(x, t)
+
\left(\nu+i\alpha\right)
\int_t^{t+z} S\big[F_{xx}(\cdot, t'); 0\big](x, t+z-t') dt'
\nn\\
&\quad
+
\left(\nu+i\alpha\right)
\int_0^t \left\{S\big[F_{xx}(\cdot, t'); 0\big](x, t+z-t')-S\big[F_{xx}(\cdot, t'); 0\big](x, t-t')\right\} dt'.
\end{align}
Therefore, by the definition \eqref{t-frac-def} of the fractional seminorm and the inequality $|a+b+c|^2 \leq 2a^2 + 2(b+c)^2 \leq 2a^2 + 4b^2 + 4c^2$, 
\begin{align}
&\no{W_t(x)}_\mu^2
\leq
4\int_0^T \int_0^{T-t} z^{-(1+2\mu)} \left|F(x,z+t) - F(x, t)\right|^2 dz dt
\nn\\
&
+8 \left(\alpha^2+\nu^2\right) \int_0^T \int_0^{T-t} z^{-(1+2\mu)} \left| \int_t^{t+z} S\big[F_{xx}(\cdot, t'); 0\big](x, t+z-t') dt' \right| ^2 dz dt
\nn\\
&
+ 8 \left(\alpha^2+\nu^2\right) \int_0^T \int_0^{T-t} z^{-(1+2\mu)} \left| \int_0^t \left\{S\big[F_{xx}(\cdot, t'); 0\big](x, t+z-t')-S\big[F_{xx}(\cdot, t'); 0\big](x, t-t')\right\} dt' \right|^2 dz dt.
\nn
\end{align}
Identifying the first term as $\no{F(x)}_\mu^2$ and handling the remaining two terms like \eqref{W-ivp-te-b1} and \eqref{W-ivp-te-b2} after 
(i)~replacing $F$ by $F_{xx}$, 
(ii) noting that $\frac{2s+1}{4}$ in \eqref{wtemp1} now corresponds to $\frac{2s+1}{4}-1 = \frac{2(s-2)+1}{4}$, 
(iii) noting that $s$ in \eqref{W-ivp-te-b2-est} can now be replaced by $s-2$ (since the range relevant to \eqref{W-ivp-te-b2-est} now is $s>\frac 52$ which implies $s-2 > \frac 12$), 
and (iv) bounding the $H^{-\frac 12}$ norm of $F_{xx}$ in \eqref{W-ivp-te-low} by the $H^{s-2}$ norm of $F_{xx}$ since now $s>\frac 32$, we obtain
\begin{align}
\no{W_t(x)}_\mu^2
&\leq
4 \no{F(x)}_\mu^2 
+
2 \left(\alpha^2+\nu^2\right) \cdot 2T c_2(s-2, \alpha, \nu, T)^2
\no{F_{xx}}_{L_t^2((0, T); H_x^{s-2}(\mathbb R))}^2
\\
&
+
2 \left(\alpha^2+\nu^2\right) \cdot \left(\frac{c_s \, T^{2(1-\mu)} \chi_{(\frac 12, 1)}(\mu)}{2\pi^2(1-\mu)}
+
\frac{4\left(1+ \sqrt \nu\, \right)^2\big(1+ \sqrt T \, \big)^2 \, T^{1-2\mu}\chi_{(0, \frac 12)}(\mu)}{\sqrt 2 \pi \nu \left(1-2\mu\right)} \right)
\no{F_{xx}}_{L_t^2((0, T); H_x^{s-2}(\mathbb R))}^2.
\nn
\end{align}
Finally, since for $\frac 32 < s < \frac 72$ we have $\mu = m-1=\frac{2s+1}{4} - 1$, we deduce 
\begin{align}\label{wt-frac}
\no{W_t(x)}_\mu
&\leq
\sqrt{2\left(\alpha^2+\nu^2\right)}\, 
\Bigg(
\frac{c_s \, T^{2(1-\mu)} \chi_{(\frac 12, 1)}(\mu)}{2\pi^2(1-\mu)}
+
\frac{4\left(1+ \sqrt \nu\, \right)^2\big(1+ \sqrt T \, \big)^2 \, T^{1-2\mu}\chi_{(0, \frac 12)}(\mu)}{\sqrt 2 \pi \nu \left(1-2\mu\right)}  
+
2T c_2(s-2, \alpha, \nu, T)^2
\Bigg)^{\frac 12}
\nn\\
&\quad
\cdot
\no{F}_{L_t^2((0, T); H_x^2(\mathbb R))}
+
2 \no{F(x)}_{W^{m-1}(0, T)},
\quad
\frac 32 < s < \frac 72, \ s\neq \frac 52.
\end{align}

Returning to the definition \eqref{hm-l2}, we combine \eqref{W-l2}, \eqref{W-te1}, \eqref{W-te2} and \eqref{wtt-est} with the fractional estimate~\eqref{wt-frac} to deduce the time estimate~\eqref{W-ivp-te2} for all $s \in \left(\frac 32, \frac 72\right] \setminus \left\{\frac 52\right\}$ and with associated constant 
\begin{align}\label{c4-def}
&
c_4(s, \alpha, \nu, T)
:=
c(m, T) \Bigg[2 + \sqrt{\alpha^2+\nu^2} + \frac{\sqrt T\left(1+ \sqrt \nu\, \right) \big(1+\sqrt T \, \big) \left(1+ 2\sqrt{\alpha^2 + \nu^2}\right)}{2^{\frac 14} \sqrt{\pi \nu}} 
\\
&+
\sqrt{2\left(\alpha^2+\nu^2\right)}\, 
\Bigg(
\frac{c_s \, T^{2(1-\mu)} \chi_{(\frac 12, 1)}(\mu)}{2\pi^2(1-\mu)}
+
\frac{4\left(1+ \sqrt \nu\, \right)^2\big(1+ \sqrt T \, \big)^2 \, T^{1-2\mu}\chi_{(0, \frac 12)}(\mu)}{\sqrt 2 \pi \nu \left(1-2\mu\right)}  
+
2T c_2(s-2, \alpha, \nu, T)^2
\Bigg)^{\frac 12}\Bigg],
\nn
\end{align}
where $c(m, T)$ is the constant in \eqref{hm-l2}, $c_2(s, \alpha, \nu, T)$ is given by \eqref{c2-def}, and $m = \frac{2s+1}{4}$, $\mu = m - [m]$.
\end{proof}

\begin{remark}[Time estimate and  solution space for well-posedness]
The norm of $F(x, t)$ in $W_t^{\frac{2s+1}{4}-1}(0, T)$, which arises in the estimation of $W(x, t)$ in $H_t^{\frac{2s+1}{4}}(0, T)$ when $s>\frac 32$ (see \eqref{W-h2} and \eqref{wt-frac}), introduces the need for a temporal Sobolev space in the solution space  of the nonlinear problem used in Section \ref{lwp-s}. 
There are two options in this regard: either include the space $C_x([0, L]; H_t^{\frac{2s+1}{4}-1}(0, T))$ in the well-posedness solution space, or use inequality \eqref{T-power-ineq} to replace the $H_t^{\frac{2s+1}{4}-1}(0, T)$ norm with the $H_t^{\frac{2s+1}{4}}(0, T)$ norm and instead include the space $C_x([0, L]; H_t^{\frac{2s+1}{4}}(0, T))$ in the solution space. We shall end up with the latter option, thus leading to the solution space \eqref{xt-def}, as it has the advantage of generating a positive power of $T$ that will be crucial in the contraction mapping argument and, furthermore, in the relevant range of $s>\frac 32$ it satisfies the algebra property (while the space $H_t^{\frac{2s+1}{4}-1}(0, T)$ does not). 
\end{remark}


\subsection{Construction of compactly supported boundary data}
Our goal is to combine the various linear estimates derived so far in order to deduce a corresponding estimate for the forced linear problem \eqref{lcgl-ibvp}, which will be used in the next section to establish the local well-posedness of the nonlinear problem~\eqref{cgl-ibvp} for $s \in \left(\frac 12, \frac 32\right) \cup \left(\frac 32, \frac 52\right)$, which is precisely the range given in Theorem \ref{lwp-t}. Note that the corresponding ranges for the Sobolev exponents of the Dirichlet and Neumann data $g_0$ and $h_1$ are, respectively, $\frac{2s+1}{4} \in \left(\frac 12, 1\right) \cup \left(1, \frac 32\right)$ and $\frac{2s-1}{4} \in \left(0, \frac 12\right) \cup \left(\frac 12, 1\right)$. Hence, by the Sobolev embedding theorem, $g_0 \in H^{\frac{2s+1}{4}}(0, T)$ is a continuous function, while $h_1 \in H^{\frac{2s-1}{4}}(0, T)$ is continuous for $s\in \left(\frac 32, \frac 52\right)$ but not necessarily for $s\in \left(\frac 12, \frac 32\right)$.

With the above in mind, we revisit the decomposition~\eqref{lin-sup} and {additionally, for each $x\in\mathbb R$, we introduce the zero extension $F_0(t)$ of $F(t)$ from $(0, T)$ to $\mathbb R$. Then, in view of inequality \eqref{U0F-def}, we readily have
\begin{equation}\label{ff0}
\no{F_0}_{L_t^p((0, T+1); H_x^s(\mathbb R))}
=
\no{F}_{L_t^p((0, T); H_x^s(\mathbb R))}
\leq 2\no{f}_{L_t^p((0, T); H_x^s(0, L))}, \quad 1\leq p \leq \infty, \ T>0.
\end{equation}
In addition, for $\frac 32<s<\frac 52$ (which amounts to $0<\frac{2s+1}{4}-1<\frac 12$), invoking the zero extension result of Theorem~11.4 in \cite{lm1972} (see also Theorem 1.43 in \cite{l2023} for the precise dependence of the constant on $T$) 
\begin{equation}\label{T-power-ineq0}
\no{F_0(x)}_{W_t^{\frac{2s+1}{4}-1}(0, T+1)}
\leq
\no{F_0(x)}_{W_t^{\frac{2s+1}{4}-1}(\mathbb R)}
\lesssim
T^{-\left(\frac{2s+1}{4}-1\right)} \no{F(x)}_{W_t^{\frac{2s+1}{4}-1}(0, T)}
\simeq
\no{F(x)}_{H_t^{\frac{2s+1}{4}-1}(0, T)}
\end{equation}
with the last equality due to \eqref{hm-l2} since $c(m, T) \simeq T^{-m}$.
Now, according to Lemma 3.7 in \cite{bo2016}, 
\begin{equation}
\no{F}_{H^{m-1}(0, T)} \leq T^{\frac{1}{1+m}} \no{F}_{H^m(0, T)}, \quad m>1. 
\end{equation}
Using this result with $m= \frac{2s+1}{4}$, we overall obtain 
\begin{equation}\label{T-power-ineq}
\no{F_0(x)}_{W_t^{\frac{2s+1}{4}-1}(0, T+1)}
\lesssim 
T^{\frac{4}{2s+5}}\no{F(x)}_{H_t^{\frac{2s+1}{4}}(0, T)}
=
T^{\frac{4}{2s+5}}\no{f(x)}_{H_t^{\frac{2s+1}{4}}(0, T)}, \quad x \in (0, L).
\end{equation}

Consider now the homogeneous Cauchy problem \eqref{U-ivp} and the forced Cauchy problem \eqref{W-ivp} but with $T$ replaced by $T+1$ and $F$ replaced by $F_0$.
Then, the decomposition~\eqref{lin-sup} can be written as
\begin{equation}\label{lin-sup2}
S\big[u_0, g_0, h_1; f\big]
= 
S\big[U_0; 0\big]\big|_{x\in (0, L), t\in (0, T)}
+
S\big[0; F_0\big]\big|_{x\in (0, L), t \in (0, T)}
+
S\big[0, \psi_0, \psi_1; 0\big].
\end{equation}
Now, thanks to the time estimates of Theorems \ref{U-ivp-t} and \ref{W-ivp-t} as well as the extension inequalities \eqref{U0F-def},  \eqref{ff0} and \eqref{T-power-ineq},} the boundary data~\eqref{psi-def} of the third component in \eqref{lin-sup2} actually belong to the same Sobolev spaces with the original boundary data. 
More precisely, we have that $\psi_0 \in H^{\frac{2s+1}{4}}(0, T)$ and $\psi_1 \in H^{\frac{2s-1}{4}}(0, T)$ with the estimates 
\begin{align}
\no{\psi_0}_{H^{\frac{2s+1}{4}}(0, T)}
&\leq
\no{g_0}_{H^{\frac{2s+1}{4}}(0, T)}
+
2c_2(s, \alpha, \nu, {T+1}) \no{u_0}_{H^s(0,L)}
\nn\\
&\quad
+
2c_3(s, \alpha, \nu, {T+1}) {\no{f}_{L_t^2((0, T); H_x^s(0, L))}}, 
\quad
\frac 12 < s < \frac 32, 
\label{mbc-est}
\\
\no{\psi_0}_{H^{\frac{2s+1}{4}}(0, T)}
&\leq
\no{g_0}_{H^{\frac{2s+1}{4}}(0, T)}
+
2c_2(s, \alpha, \nu, {T+1}) \no{u_0}_{H^s(0,L)}
+
2c_4(s, \alpha, \nu, {T+1})  
\Big(
\no{f}_{L_t^2((0, T); H_x^s(0, L))}
\nn\\
&\quad
+
T^{\frac{4}{2s+5}} \sup_{x\in [0, L]} \no{f(x)}_{H_t^{\frac{2s+1}{4}}(0, T)}
\Big),
\quad
\frac 32 < s < \frac 52, 
\label{mbc-est-2}
\\
\no{\psi_1}_{H^{\frac{2s-1}{4}}(0, T)}
&\leq
\no{h_1}_{H^{\frac{2s-1}{4}}(0, T)}
+
2c_2(s-1, \alpha, \nu, {T+1}) \no{u_0}_{H^s(0, L)}
\nn\\
&\quad
+
2c_3(s-1, \alpha, \nu, {T+1}) {\no{f}_{L_t^2((0, T); H_x^s(0, L))}},
\quad
\frac 12 < s < \frac 52, \ s\neq \frac 32.
\label{mbc-est-3}
\end{align}
In addition, by the compatibility conditions \eqref{comp-cond} and the continuity of Theorems \ref{U-ivp-t} and \ref{W-ivp-t}, we have 
\begin{equation}
\psi_0(0) := g_0(0) - u_0(0) - 0  = 0, \  s \in \left(\tfrac 12, \tfrac 32\right) \cup \left(\tfrac 32, \tfrac 52\right), 
\quad
\psi_1(0) := h_1(0) - u_0'(L) - 0  = 0, \  s \in \left(\tfrac 32, \tfrac 52\right). 
\end{equation}

Next, for $s \in \left(\frac 12, \frac 32\right) \cup \left(\frac 32, \frac 52\right)$, we extend the Dirichlet datum $\psi_0$ from $(0, T)$ to $(0, T+1)$ (actually, any interval $(0, T')$ with $T'>T$ works) in a way that the resulting extension also vanishes at $t=T+1$ while maintaining the Sobolev regularity of the original function. 
First, similarly to \eqref{U0F-def}, we let $\Psi^0 \in H^{\frac{2s+1}{4}}(\mathbb R)$ be an extension of $\psi_0 \in H^{\frac{2s+1}{4}}(0, T)$ such that
$$
\no{\Psi^0}_{H^{\frac{2s+1}{4}}(\mathbb R)} \leq 2 \no{\psi_0}_{H^{\frac{2s+1}{4}}(0, T)}. 
$$
Then, for $\theta \in C_c^\infty(\mathbb R)$ satisfying $\no{\theta}_{L^\infty(\mathbb R)} \leq 1$, $\theta \equiv 1$ on $[-T, T]$,  $\theta \equiv 0$ on $(-T-1, T+1)^c$ and $\no{\theta}_{H^{\frac{2s+1}{4}}(\mathbb R)} \lesssim 1+\sqrt T$, we let $\Psi_\theta^0(t) := \theta(t) \Psi^0(t)$ and use 
 the algebra property in $H^{\frac{2s+1}{4}}(\mathbb R)$ (recall that $s>\frac 12$) to infer\footnote{Since $\theta \in C_c^\infty(\mathbb R)$, for $s\in\left(\frac 12, \frac 52\right)$ we have $\no{\theta}_{H^{\frac{2s+1}{4}}(\mathbb R)} \leq \no{\theta}_{H^2(\mathbb R)} = \no{\theta}_{L^2(-T', T')} + \no{\theta'}_{L^2(-T', T')}+ \no{\theta''}_{L^2(-T', T')}$ and for a suitable choice of $T'$, e.g. $T'=T+1$, the first norm is bounded by $\sqrt{2T'} = \sqrt{2(T+1)}$ while the second and third norms can be bounded by $\sqrt 2 \max_{[T, T']}|\theta'| = \sqrt 2 \max_{[T, T+1]}|\theta'| \lesssim 1$ and $\sqrt 2 \max_{[T, T']}|\theta''| = \sqrt 2 \max_{[T, T+1]}|\theta''| \lesssim 1$, respectively, for a suitable choice of $\theta$ since the derivatives of $\theta$ are itself continuous and so they attains their maximum values on the interval of fixed width $[T, T+1]$, which can be adjusted to values smaller than $1$ since there is enough room to go from $\theta = 1$ to $\theta = 0$.} 
$$
\no{\Psi_\theta^0}_{H^{\frac{2s+1}{4}}(\mathbb R)} 
\leq
c_s \big(1+\sqrt T \, \big)
\no{\Psi^0}_{H^{\frac{2s+1}{4}}(\mathbb R)} 
\leq
2c_s \big(1+\sqrt T \, \big) \no{\psi_0}_{H^{\frac{2s+1}{4}}(0, T)}.
$$
Moreover, importantly, $\Psi_\theta^0(0) = \Psi_\theta^0(T+1) = 0$. Hence, since $\frac{2s+1}{4} \in \left(\frac 12, 1\right) \cup \left(1, \frac 32\right)$, Theorem~11.5 of \cite{lm1972} implies that $\Psi_\theta^0 \in H_0^{\frac{2s+1}{4}}(0, T+1)$. In turn, according to Theorem 11.4 of \cite{lm1972}, the function 
\begin{equation}\label{Psi0-def}
\Psi_0(t) := \left\{
\begin{array}{ll}
\Psi_\theta^0(t), &t \in (0, T+1),
\\
0, &t \in (0, T+1)^c,
\end{array}
\right.
\end{equation}
belongs to $H^{\frac{2s+1}{4}}(\mathbb R)$ and, in view of the above construction,  satisfies the estimate
\begin{equation}\label{Psi0-est}
\no{\Psi_0}_{H^{\frac{2s+1}{4}}(\mathbb R)} 
\leq
c_{s, T} c_s \big(1+\sqrt T \, \big) \no{\psi_0}_{H^{\frac{2s+1}{4}}(0, T)},
\end{equation}
where $c_{s, T} = c(s, T+1)$ emerges from the zero extension theorem and remains bounded as $T\to 0^+$.
Therefore, $\Psi_0$ is an extension of $\psi_0$ with compact support in $(0, T+1)$ and the same Sobolev regularity with $\psi_0$. 

For the Neumann datum $\psi_1$, the above construction is not necessary for $s \in \left(\frac 12, \frac 32\right)$ since that interval corresponds to $\frac{2s-1}{4} \in \left(0, \frac 12\right)$ and so Theorem 11.4 of \cite{lm1972} readily implies that the extension of $\psi_1$ by zero outside $(0, T)$ satisfies the analogue of \eqref{Psi0-est}. 
On the other hand, for $s \in \left(\frac 32, \frac 52\right)$ we have $\frac{2s-1}{4} \in \left(\frac 12, 1\right)$ and a construction analogous to \eqref{Psi0-def} is required. 
Overall, the function
\begin{equation}\label{Psi1-def}
\Psi_1(t) := 
\left\{
\begin{aligned}
&
\left\{
\begin{array}{ll}
\psi_1(t), &t \in (0, T),
\\
0, &t \in [0, T]^c,
\end{array}
\right.
\ s \in \left(\tfrac 12, \tfrac 32\right),
\\
&\left\{
\begin{array}{ll}
\Psi_\theta^1(t), &t \in (0, T+1),
\\
0, &t \in (0, T+1)^c,
\end{array}
\right. \quad s \in \left(\tfrac 32, \tfrac 52\right),
\end{aligned}
\right.
\end{equation}
with $\Psi_\theta^1$ defined analogously to $\Psi_\theta^0$ is an extension of $\psi_1$ outside $(0, T)$ that belongs to $H^{\frac{2s-1}{4}}(\mathbb R)$ and satisfies 
\begin{equation}\label{Psi1-est}
\no{\Psi_1}_{H^{\frac{2s-1}{4}}(\mathbb R)} 
\leq
\left\{
\begin{aligned}
&{c_{s, T}} \no{\psi_1}_{H^{\frac{2s-1}{4}}(0, T)}, \quad
s \in \left(\tfrac 12, \tfrac 32\right),
\\
&{c_{s, T}} \big(1+\sqrt T \, \big) \no{\psi_1}_{H^{\frac{2s-1}{4}}(0, T)}, \quad
s \in \left(\tfrac 32, \tfrac 52\right),
\end{aligned}
\right.
\end{equation}
{with $c_{s, T}$ remaining bounded as $T\to 0^+$.}

\subsection{Estimation of the forced linear CGL equation on a finite interval}
Suppose $s \in \left(\frac 12, \frac 32\right) \cup \left(\frac 32, \frac 52\right)$. Recall the notation $S\big[u_0, g_0, h_1; f\big]$ introduced in \eqref{s-not} for the solution to the initial-boundary value problem~\eqref{lcgl-ibvp-w}. Then, by the definitions \eqref{Psi0-def} and \eqref{Psi1-def} of $\Psi_0$ and $\Psi_1$, the solution  $S\big[0, \psi_0, \psi_1; 0\big]$ on $(0, L) \times (0, T)$ is the restriction on $(0, T)$ of the solution $S\big[0, \Psi_0, \Psi_1; 0\big]$ on $(0, L) \times (0, T+1)$. The problem associated with the latter solution amounts precisely to the reduced initial-boundary value problem \eqref{pure-ibvp} with $g(t) = \Psi_0(t)$, $h(t) = \Psi_1(t)$ and $T'=T+1$. Importantly, the compact support assumption for $g$ and $h$ as stated in \eqref{pure-ibvp} is indeed satisfied by $\Psi_0$ and $\Psi_1$. Therefore, letting
\begin{equation}\label{xt-def}
X_T = C_t([0, T]; H_x^s(0, L)) \cap C_x([0, L]; H_t^{\frac{2s+1}{4}}(0, T))
\end{equation}
we can employ Theorems~\ref{pure-ibvp-t} and \ref{pure-te-t} to infer 
\begin{equation}
\no{S\big[0, \Psi_0, \Psi_1; 0\big]}_{X_{T+1}}
\leq
c_1 \,  
\big(
\no{\Psi_0}_{H^{\frac{2s+1}{4}}(\mathbb R)} + \no{\Psi_1}_{H^{\frac{2s-1}{4}}(\mathbb R)}
\big),
\end{equation}
where $c_1=c_1(s, \lambda, L, \alpha, \nu, T) $ is the positive constant resulting after combining estimates \eqref{pure-se} and \eqref{pure-te} with $T'=T+1$. 
Note that $c_1$ remains bounded as $T\to 0^+$ (since the constants in \eqref{pure-se} and \eqref{pure-te} involve $T'=T+1\to 1\neq 0$ in that limit). Then, since $S\big[0, \psi_0, \psi_1; 0\big] \equiv S\big[0, \Psi_0, \Psi_1; 0\big]$ on $(0, L) \times (0, T)$ and also $\Psi_0$ and $\Psi_1$ satisfy~\eqref{Psi0-est} and \eqref{Psi1-est}, respectively, we conclude that
\begin{equation}\label{psi01-est}
\no{S\big[0, \psi_0, \psi_1; 0\big]}_{X_T}
\leq
c_5 \, 
\big(
\no{\psi_0}_{H^{\frac{2s+1}{4}}(0, T)} + \no{\psi_1}_{H^{\frac{2s-1}{4}}(0, T)}
\big),
\end{equation}
where the constant $c_5 = c_5(s, \lambda, L, \alpha, \nu, T) = c_1(s, \lambda, L, \alpha, \nu, T)  \,  c_{s, T} \big(1+\sqrt T \, \big)$ remains bounded as $T\to 0^+$.

Combining the linear superposition \eqref{lin-sup2} with the space and time estimates of Theorems \ref{U-ivp-t} and \ref{W-ivp-t},  estimate~\eqref{psi01-est}, the extension inequalities \eqref{U0F-def}, \eqref{ff0} and \eqref{T-power-ineq}, and the boundary data estimates \eqref{mbc-est}-\eqref{mbc-est-3}, we readily infer estimates for the solution $w = S\big[u_0, g_0, h_1; f\big]$ of the forced linear initial-boundary value problem~\eqref{lcgl-ibvp-w}. Specifically, {with the constants $c_2, c_3, c_4, c_5$ as in \eqref{mbc-est}-\eqref{mbc-est-3} and \eqref{psi01-est},} for $s\in \left(\frac 12, \frac 32\right)$ we have
\begin{align}\label{w-se}
\no{S\big[u_0, g_0, h_1; f\big]}_{X_T}
&\leq
2 \, \Big(1 + c_2  + c_5  \left( c_2(s)  + c_2(s-1)\right) \Big) 
\no{u_0}_{H^s(0, L)} 
\nn\\
&\quad
+ 
2 \, \left(\sqrt T + c_3  + c_5  \left( c_3(s) + c_3(s-1)\right) \right)
\no{f}_{L_t^2((0, T); H_x^s(0, L))}
\nn\\
&\quad
+ 
c_5 
\Big(
\no{g_0}_{H^{\frac{2s+1}{4}}(0, T)} 
+
\no{h_1}_{H^{\frac{2s-1}{4}}(0, T)}
\Big) 
\end{align}
and for $s\in \left(\frac 32, \frac 52\right)$ we have
\begin{align}\label{w-se-2}
\no{S\big[u_0, g_0, h_1; f\big]}_{X_T}
&\leq
2 \, \Big(1 + c_2  + c_5  \left( c_2(s) + c_2(s-1)\right) \Big) 
\no{u_0}_{H^s(0, L)} 
\nn\\
&\quad
+ 
2 \, \Big(\sqrt T + c_4  + c_5  \left(c_4(s) + c_3(s-1)\right) \Big)
\no{f}_{L_t^2((0, T); H_x^s(0, L))}
\nn\\
&\quad
+  2 T^{\frac{4}{2s+5}} c_4 \left(1 + c_5 \right) \sup_{x\in [0, L]} \no{f(x)}_{H_t^{\frac{2s+1}{4}}(0, T)}
\nn\\
&\quad
+ 
c_5 
\Big(
\no{g_0}_{H^{\frac{2s+1}{4}}(0, T)} 
+
\no{h_1}_{H^{\frac{2s-1}{4}}(0, T)}
\Big).
\end{align}

In view of the transformation \eqref{gamma-trans}, estimates \eqref{w-se} and \eqref{w-se-2} readily imply corresponding estimates for the solution to the finite interval problem \eqref{lcgl-ibvp} for the forced linear CGL equation. In particular, denoting the solution to that problem by $u = \Phi\big[u_0, a, b; \mathfrak f\big]$, we have via  \eqref{gamma-trans} that
\begin{equation}\label{Phi-def}
\Phi\big[u_0, a, b; \mathfrak f\big](x, t) = e^{\gamma t} S\big[u_0, e^{-\gamma t} a, e^{-\gamma t} b; e^{-\gamma t} \mathfrak f\big](x, t). 
\end{equation}

\begin{lemma}
Let $\gamma \geq 0$. Then, for any $m\geq 0$, 
\begin{align}
\no{e^{-\gamma t} \varphi}_{H^m(0, T)} &\leq c_{m, \gamma} \no{\varphi}_{H^m(0, T)},
\label{135}
\\
\no{e^{\gamma t} \varphi}_{H^m(0, T)} &\leq c_{m, \gamma} \, e^{\gamma T} \no{\varphi}_{H^m(0, T)}.
\label{136}
\end{align}
\end{lemma}

\begin{proof}
Suppose first that $m\in \mathbb N_0$. Then, according to the $L^2$ characterization \eqref{hm-l2}, 
$$
\no{e^{-\gamma t} \varphi}_{H^m(0, T)}
=
c(m, T) \no{e^{-\gamma t} \varphi}_{W^m(0, T)}
=
c(m, T) \sum_{j=0}^{m} \big\| \left(e^{-\gamma t} \varphi\right)^{(j)} \big\|_{L^2(0, T)}.
$$
By the product rule, 
$$
\left(e^{-\gamma t} \varphi\right)^{(j)} = \sum_{\ell=0}^j {j \choose \ell} \left(\p_t^{j-\ell}e^{-\gamma t}\right) \varphi^{(\ell)}
=
\sum_{\ell=0}^j {j \choose \ell} (-\gamma)^{j-\ell} e^{-\gamma t} \varphi^{(\ell)}.
$$
Hence, by the triangle inequality and the fact that $e^{-\gamma t} \leq 1$ for $\gamma, t \geq 0$,
$$
\no{\left(e^{-\gamma t} \varphi\right)^{(j)}}_{L^2(0, T)}
\leq
\sum_{\ell=0}^j {j \choose \ell} |\gamma|^{j-\ell} 
\no{e^{-\gamma t} \varphi^{(\ell)}}_{L^2(0, T)}
\leq
c_{j, \gamma}  \sum_{\ell=0}^j \big\|\varphi^{(\ell)}\big\|_{L^2(0, T)}.
$$
In turn, 
\begin{align*}
\no{e^{-\gamma t} \varphi}_{H^m(0, T)}
&\leq
c(m, T) \sum_{j=0}^{m} c_{j, \gamma}  \sum_{\ell=0}^j \big\|\varphi^{(\ell)}\big\|_{L^2(0, T)}
\nn\\
&\leq
c_{m, \gamma} \, c(m, T) \sum_{j=0}^m \big\|\varphi^{(j)}\big\|_{L^2(0, T)}
=
c_{m, \gamma}
\no{\varphi}_{H^m(0, T)},
\end{align*}
which amounts to \eqref{135} for $m\in \mathbb N_0$. The case of general $m\geq 0$ follows via interpolation. The proof of \eqref{136} is entirely analogous. 
\end{proof}

Combining estimates \eqref{w-se} and \eqref{w-se-2} with inequalities \eqref{135} and \eqref{136} for $m=\frac{2s\pm 1}{4}$, 
%
%
we deduce the following crucial estimates for the solution of problem~\eqref{lcgl-ibvp}:
\begin{align}\label{u-se}
\no{\Phi\big[u_0, a, b; \mathfrak f\big]}_{X_T}
&\leq
\left(1+c_{s, \gamma}\right) e^{\gamma T} 
\bigg[
2 \, \Big(1 + c_2  + c_5  \left( c_2(s)  + c_2(s-1)\right) \Big) 
\no{u_0}_{H^s(0, L)} 
\nn\\
&\quad
+ 
2 \, \left(\sqrt T + c_3  + c_5  \left( c_3(s) + c_3(s-1)\right) \right)
\no{\mathfrak f}_{L_t^2((0, T); H_x^s(0, L))}
\nn\\
&\quad
+ 
c_5 c_{s, \gamma}
\Big( 
\no{a}_{H^{\frac{2s+1}{4}}(0, T)} 
+
\no{b}_{H^{\frac{2s-1}{4}}(0, T)}
\Big)\bigg], \quad \frac 12 < s < \frac 32,
\end{align}
and
\begin{align}\label{u-se-2}
\no{\Phi\big[u_0, a, b; \mathfrak f\big]}_{X_T}
&\leq
\left(1+c_{s, \gamma}\right) e^{\gamma T} 
\bigg[
2 \, \Big(1 + c_2  + c_5  \left( c_2(s) + c_2(s-1)\right) \Big) 
\no{u_0}_{H^s(0, L)} 
\nn\\
&\quad
+ 
2 \, \Big(\sqrt T + c_4  + c_5  \left(c_4(s) + c_3(s-1)\right) \Big)
\no{\mathfrak f}_{L_t^2((0, T); H_x^s(0, L))}
\nn\\
&\quad
+ 2 T^{\frac{4}{2s+5}} c_4 \left(1 + c_5 \right) c_{s, \gamma} \sup_{x\in [0, L]} \no{\mathfrak f(x)}_{H_t^{\frac{2s+1}{4}}(0, T)}
\nn\\
&\quad
+ 
c_5 
c_{s, \gamma}
\Big( 
\no{a}_{H^{\frac{2s+1}{4}}(0, T)} 
+
\no{b}_{H^{\frac{2s-1}{4}}(0, T)}
\Big)\bigg], \quad \frac 32 < s < \frac 52.
\end{align}
As we shall see next, the linear estimates \eqref{u-se} and \eqref{u-se-2} provide the basis for the contraction mapping argument leading to the proof of the local well-posedness Theorem \ref{lwp-t} for the CGL equation on a finite interval. 

\subsection{Local well-posedness of CGL on a finite interval}
\label{lwp-s}

We now combine the central linear estimates~\eqref{u-se} and \eqref{u-se-2} with a contraction mapping argument in order to establish the local well-posedness result of Theorem~\ref{lwp-t} for the Dirichlet-Neumann initial-boundary value problem~\eqref{cgl-ibvp} of the nonlinear Ginzburg-Landau equation on a finite interval. 

We say that $u$ is a solution to the nonlinear initial-boundary value problem \eqref{cgl-ibvp} if and only if $u$ solves the integral equation $u = \Phi\big[u_0, a, b; -\left(\kappa + i\beta\right) |u|^p u\big]$, where $\Phi$ is defined through the unified transform solution formula for the forced linear counterpart \eqref{lcgl-ibvp} of problem \eqref{cgl-ibvp} according to~\eqref{Phi-def}. With this definition in hand, proving that $u$ satisfies problem \eqref{cgl-ibvp} amounts to proving that $u$ is a fixed point of the map
\begin{equation}\label{it-map}
u(x, t) \mapsto \Phi\big[u_0, a, b; -\left(\kappa + i\beta\right) |u|^p u\big].
\end{equation}
Next, for $s \in \left(\frac 12, \frac 32\right) \cup \left(\frac 32, \frac 52\right)$, we shall show the existence and uniqueness of such a fixed point in the space $X_T$ defined by~\eqref{xt-def} for an appropriate choice of lifespan $T>0$ determined below.
\\[2mm]
\textit{Existence.}
Let $B(0, \rho_T) \subset X_T$ denote the closed ball in $X_T$ with center at the origin and radius  
\begin{equation}\label{rst-def}
\rho_T
:= 
2\mathcal B \, 
\Big( 
\no{u_0}_{H^s(0, L)} 
+
\no{a}_{H^{\frac{2s+1}{4}}(0, T)} 
+
\no{b}_{H^{\frac{2s-1}{4}}(0, T)}
\Big)
\end{equation}
where $\mathcal B = \mathcal B(s, \lambda, L, \alpha, \nu, \gamma, T)$ is given by 
\begin{equation}\label{cr-def}
\begin{aligned}
\mathcal B 
=
\left(1+c_{s, \gamma}\right) e^{\gamma T} 
\left\{
2 \, \left[1 + c_2  + c_5 \left( c_2(s)  + c_2(s-1)\right) \right] 
+ 
c_5 
c_{s, \gamma}
\right\}.
\end{aligned}
\end{equation}
Then, for $s\in \left(\frac 12, \frac 32\right)$ and $u \in B(0, \rho_T)$, the linear estimate \eqref{u-se} with $\mathfrak f = \mathfrak f(u) = -\left(\kappa + i\beta\right) |u|^p u$ and $p$ satisfying the conditions \eqref{p-def} can be combined with Lemma 3.1 of \cite{bo2016}, which provides a generalization of Lemma~4.10.2 in \cite{c2003} and is valid for $s>\frac 12$, in order to handle the nonlinearity and infer the estimate
\begin{equation}\label{Phi-se}
\begin{aligned}
\no{\Phi\big[u_0, a, b; -\left(\kappa + i\beta\right) |u|^p u\big](t)}_{X_T}
&\leq
\frac{\rho_T}{2}
+ 
2 \left(1+c_{s, \gamma}\right) e^{\gamma T} c_{s, p} \sqrt{\kappa^2+\beta^2} 
\\
&\quad
\cdot 
\sqrt T \left(\sqrt T + c_3  + c_5  \left( c_3(s) + c_3(s-1)\right) \right) \rho_T^{p+1}.
\end{aligned}
\end{equation}
Similarly, for $s\in \left(\frac 32, \frac 52\right)$, the linear estimate \eqref{u-se-2} yields
\begin{equation}\label{Phi-se-2}
\begin{aligned}
\no{\Phi\big[u_0, a, b; -\left(\kappa + i\beta\right) |u|^p u\big](t)}_{X_T}
&\leq
\frac{\rho_T}{2}
+ 
\left(1+c_{s, \gamma}\right) e^{\gamma T} c_{s, p} \sqrt{\kappa^2+\beta^2} 
\\
&\quad
\cdot 
\Big[2 \sqrt T \, \Big(\sqrt T + c_4  + c_5  \left(c_4(s) + c_3(s-1)\right) \Big)
\\
&\quad
+ 2 T^{\frac{4}{2s+5}} c_4 \left(1 + c_5 \right) c_{s, \gamma} \Big] \rho_T^{p+1}.
\end{aligned}
\end{equation}

In view of the above two estimates, choosing $T>0$ such that the right sides of \eqref{Phi-se} and \eqref{Phi-se-2} are bounded above by $\rho_T$, i.e. imposing the conditions
\begin{align}
&4 \left(1+c_{s, \gamma}\right) e^{\gamma T} c_{s, p} \sqrt{\kappa^2+\beta^2} 
\sqrt T \left(\sqrt T + c_3  + c_5  \left( c_3(s) + c_3(s-1)\right) \right) \, \rho_T^p
\leq
1, 
\quad
\frac 12 < s < \frac 32, 
\label{T-into}
\\
&2 \left(1+c_{s, \gamma}\right) e^{\gamma T} c_{s, p} \sqrt{\kappa^2+\beta^2} 
\Big[2 \sqrt T \, \Big(\sqrt T + c_4  + c_5  \left(c_4(s) + c_3(s-1)\right) \Big)
\nn\\
&\hspace*{4.5cm}
+ 2 T^{\frac{4}{2s+5}} c_4 \left(1 + c_5 \right) c_{s, \gamma} \Big]  \, \rho_T^p
\leq
1, 
\quad
\frac 32 < s < \frac 52, 
\label{T-into-2}
\end{align}
we can ensure that the map \eqref{it-map} takes $B(0, \rho_T)$ to $B(0, \rho_T)$. It is important to note that the conditions \eqref{T-into} and~\eqref{T-into-2} can be fulfilled by an appropriate $T>0$ \textbf{without a smallness condition on the data} due to the fact that the constants $c_2, \ldots, c_5$ are as in \eqref{mbc-est}-\eqref{mbc-est-3} and \eqref{psi01-est} and hence remain bounded as $T\to 0^+$ while, additionally, the presence of the factors $\sqrt T$ and $T^{\frac{4}{2s+5}}$ in \eqref{T-into} and~\eqref{T-into-2} allows one to make the relevant coefficients of $\rho_T^p$  arbitrarily small.  

In addition, for $s\in \left(\frac 12, \frac 32\right)$ and any two $u_1, u_2 \in B(0, \rho_T)$, the linear estimate \eqref{u-se} with $\mathfrak f = \mathfrak f(u) = -\left(\kappa + i\beta\right) \left(|u_1|^p u_1 - |u_2|^p u_2\right)$ together with Lemma 3.1 of \cite{bo2016} imply
\begin{align}\label{Phi-diff}
&\quad
\no{\Phi\big[u_0, a, b; -\left(\kappa + i\beta\right) |u_1|^p u_1\big](t) - \Phi\big[u_0, a, b; -\left(\kappa + i\beta\right) |u_2|^p u_2\big](t)}_{X_T}
\nn\\
&\leq
2 \left(1+ c_{s, \gamma}\right) e^{\gamma T} c_{s, p} \sqrt{\kappa^2+\beta^2} \, 
\sqrt T \left(\sqrt T + c_3  + c_5  \left( c_3(s) + c_3(s-1)\right) \right) \, 
2 \rho_T^p \no{u_1-u_2}_{X_T}.
\end{align}
Similarly, for $s\in \left(\frac 32, \frac 52\right)$ the linear estimate \eqref{u-se-2} yields
\begin{align}\label{Phi-diff-2}
&\quad
\no{\Phi\big[u_0, a, b; -\left(\kappa + i\beta\right) |u_1|^p u_1\big](t) - \Phi\big[u_0, a, b; -\left(\kappa + i\beta\right) |u_2|^p u_2\big](t)}_{X_T}
\nn\\
&\leq
\left(1+ c_{s, \gamma}\right) e^{\gamma T} c_{s, p} \sqrt{\kappa^2+\beta^2} \, 
\Big[2 \sqrt T \, \Big(\sqrt T + c_4  + c_5  \left(c_4(s) + c_3(s-1)\right) \Big)
\nn\\
&\quad
+ {2 T^{\frac{4}{2s+5}} c_4} \left(1 + c_5 \right) c_{s, \gamma} \Big]
\, 
2 \rho_T^p \no{u_1-u_2}_{X_T}.
\end{align}
Hence, further restricting $T>0$ so that the coefficient of $\no{u_1-u_2}_{X_T}$ on the right side of \eqref{Phi-diff} and \eqref{Phi-diff-2} is strictly less than one, i.e. imposing the conditions
\begin{align}
&4 \left(1+ c_{s, \gamma}\right) e^{\gamma T} c_{s, p} \sqrt{\kappa^2+\beta^2} \, 
\sqrt T \Big(\sqrt T + c_3  + c_5 \left( c_3(s) + c_3(s-1)\right) \Big) \, 
\rho_T^p  < 1,
\quad
\frac 12 < s < \frac 32,
\label{T-contr}
\\
&2\left(1+ c_{s, \gamma}\right) e^{\gamma T} c_{s, p} \sqrt{\kappa^2+\beta^2} \, 
\Big[2 \sqrt T \, \Big(\sqrt T + c_4  + c_5  \left(c_4(s) + c_3(s-1)\right) \Big)
\nn\\
&\hspace*{4.5cm}
+ {2 T^{\frac{4}{2s+5}} c_4} \left(1 + c_5 \right) c_{s, \gamma} \Big] \, 
\rho_T^p  < 1,
\quad
\frac 32 < s < \frac 52,
\label{T-contr-2}
\end{align}
which, as in the case of \eqref{T-into} and \eqref{T-into-2} earlier, are realizable \textbf{without a smallness condition on the data} due to the presence of the factors $\sqrt T$ and $T^{\frac{4}{2s+5}}$ and the fact that the constants $c_3, c_4, c_5$ are taken from \eqref{mbc-est}-\eqref{mbc-est-3} and \eqref{psi01-est} and hence remain bounded as $T\to 0^+$, we can ensure that the map \eqref{it-map} is a contraction in $B(0, \rho_T)$. In turn, Banach's fixed point theorem implies the existence of a unique fixed point of this map in $B(0, \rho_T)$ which, according to our definition of solution earlier, corresponds to a  unique solution in $B(0, \rho_T)$ for the CGL initial-boundary value problem \eqref{cgl-ibvp}. This completes the proof of the existence part of Theorem~\ref{lwp-t}.
\\[2mm]
\textit{Continuous dependence on the data.}
The argument for this part of the proof is similar to the one for the existence part (see, for example, \cite{mmo2024} for more details).
\\[2mm]
\textit{Extending uniqueness to $C_t([0, T]; H_x^s(0, L))$.}
The proof of existence readily implied uniqueness in the ball $B(0, \rho_T) \subset X_T$, with $\rho_T>0$ given by \eqref{rst-def} and $T>0$ satisfying \eqref{T-into}, \eqref{T-into-2}, \eqref{T-contr} and \eqref{T-contr-2}. 
We now complete the proof of the local Hadamard well-posedness result given in Theorem \ref{lwp-t} by extending uniqueness to the larger space $C_t([0, T]; H_x^s(0, L)) \supset X_T \supset B(0, \rho_T)$.

Suppose $u, v \in C_t([0, T]; H_x^s(0, L))$ are two solutions associated with the same triplet of data $(u_0, a, b)$. Then, the difference $w:=u-v$ satisfies the problem
\begin{equation*}
\begin{aligned}
&w_t - \left(\nu + i \alpha\right) w_{xx} - \gamma w + \left(\kappa + i\beta\right) \left(|u|^p u - |v|^p v\right) = 0, \quad  (x, t) \in (0, L) \times (0, T),
\\
&w(x,0) = 0,
\\
&w(0, t) = 0, \quad w_x(L, t) = 0.
\end{aligned} 
\end{equation*}
If $u, v$ are sufficiently smooth, then we can proceed via an energy estimate as follows. Multiplying by $\overline w$ the equation satisfied by $w$ and then integrating in $x$ while enforcing the boundary conditions for $w$, we have
$$
\int_0^L \overline w w_t \, dx
=
-\left(\nu + i \alpha\right) \no{w_x(t)}_{L_x^2(0, L)}^2
+
\gamma \no{w(t)}_{L_x^2(0, L)}^2
-
\left(\kappa + i\beta\right) \int_0^L \overline w \left(|u|^p u - |v|^p v\right) dx.
$$
The real part of this equation yields
\begin{align*}
\frac 12 \frac{d}{dt} \no{w(t)}_{L_x^2(0, L)}^2
&=
- \nu \no{w_x(t)}_{L_x^2(0, L)}^2
+ \gamma \no{w(t)}_{L_x^2(0, L)}^2
\nn\\
&\quad
- \kappa \, \text{Re} \int_0^L \overline w \left(|u|^p u - |v|^p v\right) dx
+ \beta \, \text{Im} \int_0^L \overline w \left(|u|^p u - |v|^p v\right) dx
\nn\\
&\leq
\gamma \no{w(t)}_{L_x^2(0, L)}^2
- \kappa \, \text{Re} \int_0^L \overline w \left(|u|^p u - |v|^p v\right) dx
+ \beta \, \text{Im} \int_0^L \overline w \left(|u|^p u - |v|^p v\right) dx
\end{align*}
with the last inequality due to the fact that $\nu>0$.
Bounding the right side via the triangle inequality and then using the standard inequality $\left| |u|^p u - |v|^p v\right| \leq c_p \left(|u|^p + |v|^p\right) |u-v|$, $p\geq 0$ (trivial for $p=0$, otherwise see for example the proof of Lemma 4.2 in \cite{hkmms2024}, which is actually valid for any $p>0$), we have
\begin{align*}
\frac 12 \frac{d}{dt} \no{w(t)}_{L_x^2(0, L)}^2
&\leq
\gamma \no{w(t)}_{L_x^2(0, L)}^2
+ \left(\kappa + |\beta|\right) \, \bigg| \int_0^L \overline w \left(|u|^p u - |v|^p v\right) dx\bigg|
\nn\\
&\lesssim
\gamma \no{w(t)}_{L_x^2(0, L)}^2
+ \left(\kappa + |\beta|\right) \int_0^L \left(|u|^p + |v|^p\right) |w|^2 dx 
\nn\\
&\lesssim
\left[
\gamma 
+ \left(\kappa + |\beta|\right) 
\left(\no{|u|^p}_{L_x^\infty(0, L)}  + \no{|v|^p}_{L_x^\infty(0, L)}\right) 
\right]
\no{w(t)}_{L_x^2(0, L)}^2.
\end{align*}
Bounding the $L^\infty$ norms by  $H^s$ norms thanks to the Sobolev embedding theorem (since $s>\frac 12$) and then using Lemma 3.1 of \cite{bo2016} to move the power $p$ outside the norms, we find
\begin{equation*}
\begin{aligned}
\frac 12 \frac{d}{dt} \no{w(t)}_{L_x^2(0, L)}^2
&\lesssim
\left[
\gamma 
+ \left(\kappa + |\beta|\right) 
\left(\no{u(t)}_{H_x^s(0, L)}^p  + \no{v(t)}_{H_x^s(0, L)}^p\right) 
\right]
\no{w(t)}_{L_x^2(0, L)}^2
\\
&\leq
\left[
\gamma 
+ 
2\left(\kappa + |\beta|\right) \rho_T^p \right]
\no{w(t)}_{L_x^2(0, L)}^2.
\end{aligned}
\end{equation*}
Integrating this differential inequality, we obtain $\no{w(t)}_{L_x^2(0, L)}^2 \leq \no{w(0)}_{L_x^2(0, L)}^2 e^{2\left[\gamma + 2\left(\kappa + |\beta|\right) \rho_T^p \right] t}$ which in view of the initial condition $w(0) = 0$ allows us to conclude that $\no{w(t)}_{L_x^2(0, L)} \equiv 0$ or, equivalently (since $s>\frac 12$), $u\equiv v$. 
The case of rough $u, v$ an be treated via mollification along the lines of the arguments used in the proof of Proposition 1.4 in \cite{h2005}.

\subsection{Global solutions of the open-loop problem}

We now prove Theorem \ref{gwp-t}. We assume $u_0\in H^1(0,L)$, $a,b\in H_{\text{loc}}^1(\mathbb{R}_+)$, $\kappa,\nu>0$, $\gamma\ge 0$, $0<p\le 2$ if $\beta\ge 0$ and $0<p<\frac{4}{5}$ if $\beta<0$.  
Furthermore, without loss of generality we take $\alpha>0$. This is not a necessary assumption and is made only for convenience. The case $\alpha<0$ can be handled by changing $\alpha$ with $-\alpha$ in certain arguments in the proofs below.
\\[2mm]
\textbf{The case of $\beta\ge 0$.}
Multiplying the CGL equation by the conjugate $\bar{u}$ of $u$, integrating over $(0,L)$, taking the real part of the resulting expression and using the boundary conditions $u(0,t)=a(t)$ and $u_x(L,t)=b(t)$, we have
 \begin{equation}
 \begin{aligned}
 \frac12\frac{d}{dt}\no{u(t)}_{L_x^2(0,L)}^2
 &-\gamma\no{u(t)}_{L^2_x(0,L)}^2
 +\nu\no{u_x(t)}_{L_x^2(0,L)}^2+\kappa\no{u(t)}_{L_x^{p+2}(0,L)}^{p+2}\\
 &-\text{Re}\left[(\nu+i\alpha) \left(b(t)\bar{u}(L,t)-u_x(0,t)\bar{a}(t)\right)\right]=0.
 \end{aligned}
 \end{equation}
Integrating this equation in $t$, we find
\begin{align}
E_0(t):=\no{u(t)}_{L^2_x(0,L)}^2
&\leq
\no{u_0}_{L_x^2(0,L)}^2+2\gamma\int_0^t\no{u(s)}_{L^2_x(0,L)}^2ds \nn\\
&\quad
-2\nu\int_0^t\no{u_x(s)}_{L^2_x(0,L)}^2ds - 2\kappa\int_0^t\no{u(s)}_{L_x^{p+2}(0,L)}^{p+2}ds\nn\\
&\quad
+2 \int_0^t \text{Re}\left[(\nu+i\alpha) \left(b(s)\bar{u}(L,s)-u_x(0,s)\bar{a}(s)\right)\right] ds
\end{align}
so using the fact that $\kappa, \nu > 0$ as well as the Cauchy-Schwarz inequality for the last term, we obtain
\begin{equation}\label{L2defest0}
E_0(t)
\leq
\no{u_0}_{L_x^2(0,L)}^2+2\gamma\int_0^t E_0(s) ds
+2 \left|\nu+i\alpha\right| \left(\no{b}_{L^2(0,t)} D_L(t)+\no{a}_{L^2(0,t)} N_0(t)\right)
\end{equation}
where 
\begin{equation}
D_L(t):=\no{u(L,\cdot)}_{L^2(0,t)}, \quad N_0(t):=\no{u_x(0,\cdot)}_{L^2(0,t)}.
\end{equation}  
Young's inequality $wz \leq \frac{w^q}{q} + \frac{z^{q'}}{q'}$ for $w =\left(\epsilon q\right)^{\frac 1q} \tilde w$ and $z= c_\epsilon \tilde z$ with $c_\epsilon =  \left(\epsilon q\right)^{-\frac{q'}{q}} (q')^{-1}$ 
yields what is known as the epsilon-Young inequality for $\tilde w$ and $\tilde z$, namely $\tilde w \tilde z \leq \epsilon \tilde w^q + c_\epsilon \tilde z^{q'}$. 
Using this inequality for $\tilde w =  \no{b}_{L^2(0,T)}$, $\tilde z = D_L(t)$ and $q=2$ as well as for $\tilde w = \no{a}_{L^2(0,T)}$, $\tilde z =   N_0(t)$ and $q=2$, we obtain
\begin{equation}\label{L2defest}
E_0(t)
\leq
\no{u_0}_{L_x^2(0,L)}^2+2\gamma\int_0^t E_0(s) ds
+2 \left|\nu+i\alpha\right| \left(\epsilon D_L(t)^2 + c_\epsilon \no{b}_{L^2(0,t)}^2  + \epsilon N_0(t)^2 + c_\epsilon \no{a}_{L^2(0,t)}^2\right)
\end{equation}

Concerning $D_L$, we note that  
\begin{align}\label{uLt}
\left|u(L,t)\right|^2&=\frac{1}{L}\int_0^L \p_x \left(xu(x,t)\bar{u}(x,t)\right)dx=\frac{1}{L}\int_0^L \left(\left|u(x,t)\right|^2+2x\, \text{Re}(u\bar u_x)\right) dx
\nn\\
&\le \frac{1}{L}\left(\no{u(t)}_{L^2_x(0,L)}^2+2L\no{u(t)}_{L_x^2(0,L)}\no{u_x(t)}_{L_x^2(0,L)}\right)\le c\no{u(t)}_{H^1_x(0,L)}^2
\end{align}
with $c = 2 + \frac 1L$ depending only on $L$.
Thus, we obtain
\begin{equation}\label{DLtsquare}
D_L(t)^2 \equiv \int_0^t|u(L,s)|^2ds\le c\int_0^t\no{u(s)}_{H^1_x(0,L)}^2ds,\quad 0\le t\le T.
\end{equation}

Proceeding to $N_0$, we observe that
\begin{equation*}
    \left|u_x(0,t)\right|^2  =-\frac{1}{L}\int_0^L \p_x \left((L-x)|u_x(x,t)|^2\right) dx =\frac{1}{L}\int_0^L|u_x(x,t)|^2dx-\frac{2}{L}\text{Re}\int_0^L(L-x)u_{xx}\bar{u}_xdx
\end{equation*}
so by Young's inequality $wz \leq \frac{w^q}{q} + \frac{z^{q'}}{q'}$ with $w= |u_{xx}(x,t)|$, $z=|u_{x}(x,t)|$, $q=2$ we obtain
\begin{equation}
\left|u_x(0,t)\right|^2
 \le c\left(\no{u_x(t)}_{L_x^2(0,L)}^2+\no{u_{xx}(t)}_{L_x^2(0,L)}^2\right)
\end{equation}
where $c$ depends on $L$.
Therefore, 
\begin{equation}\label{N0tsquare}
N_0^2(t)\equiv \int_0^t|u_x(0,s)|^2ds \le c\int_0^t\left(\no{u_x(s)}_{L_x^2(0,L)}^2+\no{u_{xx}(s)}_{L_x^2(0,L)}^2\right)ds.
\end{equation}

Next, we form a different energy by multiplying the CGL equation by $-\bar{u}_{xx}$ and integrating over $(0,L)$ to get
\begin{equation}\label{mult_uxx}
\begin{aligned}
-\int_0^Lu_t\bar{u}_{xx}dx + (\nu+i\alpha)\int_0^L|u_{xx}|^2dx-\left(\kappa+i\beta\right)\int_0^L|u|^pu\bar{u}_{xx}dx+\gamma\int_0^Lu\bar{u}_{xx}dx=0.
\end{aligned}
\end{equation}
The first term can be rewritten through integration by parts in $x$ as
\begin{equation}
-\int_0^Lu_t\bar{u}_{xx}dx = -u_t(L,t)\bar{b}(t)+a'(t)\bar{u}_x(0,t) + \int_0^Lu_{xt}\bar{u}_xdx.
\end{equation}
Moreover, the third term can be rewritten through integration by parts in $x$ as
\begin{align}\label{ibp-en}
\int_0^L|u|^pu\bar{u}_{xx}dx 
= 
\int_0^L \left(u^{\frac{p+2}{2}}\bar{u}^{\frac{p}{2}}\right) \bar{u}_{xx}dx
&=
\left|u(L,t)\right|^pu(L,t)\bar{b}(t)-|a(t)|^pa(t)\bar{u}_x(0,t) 
\nn\\
&\quad
-\int_0^L\left(\tfrac{p+2}{2}|u|^p|u_x|^2+\tfrac{p}{2}u^2|u|^{p-2}\bar{u}_x^2\right)dx
\end{align}
and the fourth term  can be expressed as
\begin{equation}
\int_0^Lu\bar{u}_{xx}dx =  u(L,t)\bar{b}(t)-a(t)\bar{u}_x(0,t) -  \int_0^L|u_x|^2dx.
\end{equation}
Also, observing that
$
\text{Re}\left(u^2 \bar u_x^2\right) + |u|^2 |u_x|^2
=
\frac 12 \left(u^2 \bar u_x^2 + \bar u^2 u_x^2\right) + |u|^2 |u_x|^2
=
\frac 12 \left(u \bar u_x + \bar u u_x\right)^2 
=
2 \left(\text{Re}(u \bar u_x)\right)^2 \geq 0
$,
we have
\begin{equation}\label{re-ineq}
\text{Re}\int_0^L\left(\tfrac{p+2}{2}|u|^p|u_x|^2+\tfrac{p}{2}u^2|u|^{p-2}\bar{u}_x^2\right)dx\ge \int_0^L\left(\tfrac{p+2}{2}-\tfrac{p}{2}\right)|u|^p|u_x|^2dx\ge 0.
\end{equation}
In view of the above calculations and the fact that $\kappa, \nu>0$, we take the real parts in \eqref{mult_uxx} to obtain
\begin{align}\label{mul_H1est}
&\quad
\frac{1}{2}\frac{d}{dt}\no{u_x(t)}^2_{L^2_x(0,L)} + \nu\no{u_{xx}(t)}^2_{L_x^2(0,L)}
\nn\\
&\leq 
\text{Re}\left(u_t(L,t)\bar{b}(t)-a'(t)\bar{u}_x(0,t)\right)
+\text{Re}\left[\left(\kappa+i\beta\right)\left(\left|u(L,t)\right|^pu(L,t)\bar{b}(t)-|a(t)|^pa(t)\bar{u}_x(0,t)\right)\right]
\nn\\
&\quad
-\gamma \, \text{Re}\left(u(L,t)\bar{b}(t)-a(t)\bar{u}_x(0,t)\right)
+\beta \, \text{Im}\int_0^L \tfrac{p}{2}u^2|u|^{p-2}\bar{u}_x^2 dx + \gamma \no{u_x(t)}_{L^2_x(0,L)}^2.
\end{align}

Next, we estimate the time integrals of the boundary terms on the right side of \eqref{mul_H1est}.  For the first of these integrals, using the Cauchy-Schwarz inequality and recalling the definition of $D_L(t)$, we have
\begin{align}
\text{Re}\int_0^tu_t(L,s)\bar{b}(s)ds
&=
\text{Re}\left(u(L,t)\bar{b}(t)-u_0(L)\bar{b}(0)-\int_0^tu(L,s)\overline{b'}(s)ds\right)
\nn\\
&\leq
\left|u(L, t)\right| \no{b}_{L^\infty(0, T)} + \left|u_0(L)\right| \left|b(0)\right|
+
D_L(t) \no{b'}_{L^2(0, T)}
\end{align}
so, in view of \eqref{uLt} and the Sobolev embedding theorem,
\begin{equation}
\text{Re}\int_0^tu_t(L,s)\bar{b}(s)ds
\leq
c \no{u(t)}_{H_x^1(0,L)} \no{b}_{H^1(0, T)} + c 
+
D_L(t)  \no{b}_{H^1(0, T)}.
\end{equation}
%
Employing the epsilon-Young  inequality  $w z \leq \epsilon w^q + c_\epsilon z^{q'}$ for $w =  \no{u(t)}_{H_x^1(0,L)}$, $z = c \no{b}_{H^1(0,T)}$ and $q=2$, as well as the standard Young's inequality for $w = D_L(t)$, $z =  \no{b}_{H^1(0,T)}$ and $q=2$, we obtain
\begin{align}\label{be1}
\text{Re}\int_0^tu_t(L,s)\bar{b}(s)ds
&\leq
c  + \epsilon \no{u(t)}_{H_x^1(0,L)}^2 + c^2 c_\epsilon  \no{b}_{H^1(0, T)}^2 
+
\frac 12 D_L(t)^2 + \frac 12  \no{b}_{H^1(0, T)}^2
\nn\\
&\leq
c  +  \epsilon \no{u(t)}_{H_x^1(0,L)}^2 + \left(1+ c^2 c_\epsilon\right) \no{b}_{H^1(0, T)}^2 
+
c \int_0^t\no{u(s)}_{H_x^1(0,L)}^2ds
\end{align}
after using \eqref{DLtsquare} in the last step.
Similarly, by the definition of $N_0(t)$, the Cauchy-Schwarz inequality and the epsilon-Young inequality, 
\begin{align}\label{be2}
-\text{Re}\left(\int_0^ta'(s)\bar{u}_x(0,s)ds\right)
&\le  
N_0(t)\no{a'}_{L^2(0,T)}
\le 
\epsilon N_0^2(t)+c_\epsilon\no{a'}_{L^2(0,T)}^2
\nn\\
&\le \epsilon \int_0^t\left(\no{u_{xx}(s)}_{L_x^2(0,L)}^2+\no{u_x(s)}_{L_x^2(0,L)}^2\right)ds +c_\epsilon\no{a}_{H^1(0,T)}^2.
\end{align}

Concerning the second boundary term in \eqref{mul_H1est}, we begin by noting that
\begin{align}
\left|u(L,t)\right|^{p}u(L,t) 
&= \frac{1}{L}\int_0^L \p_x \left(x \left|u(x,t)\right|^{p} u(x,t)\right) dx 
\nn\\
&= \frac{1}{L}\int_0^L \left(\left|u(x,t)\right|^pu(x,t) + x \, \tfrac{p+2}{2}\left|u(x,t)\right|^pu_x(x,t)+ x \, \tfrac
{p}{2}\left|u(x,t)\right|^{p-2} \left(u(x,t)\right)^2 \bar{u}_x(x,t)\right)dx 
\nn
\end{align}
so by H\"older's inequality 
\begin{equation}\label{nonlinbdrest}
\left|u(L,t)\right|^{p+1}
\leq c\left(\no{u(t)}_{L^{p+2}_x(0,L)}^{p+1}+\no{u(t)}_{L_x^{p+2}(0,L)}^p \no{u_x(t)}_{L_x^{\frac{p+2}{2}}(0,L)}\right)
\end{equation}
where $c = \frac 1L + p + 1$.
Young's inequality $wz \leq \frac{w^q}{q} + \frac{z^{q'}}{q'}$ applied twice, once with $w= \no{u(t)}_{L_x^{p+2}(0,L)}^p$, $z=\no{u_x(t)}_{L_x^{\frac{p+2}{2}}(0,L)}$, $q=\frac{p+2}{p}$ and once with $w= \no{u(t)}_{L_x^{p+2}(0,L)}^{p+1}$, $z=1$, $q=\frac{p+2}{p+1}$, further yields
\begin{align}\label{nonlinbdrest2}
\int_0^t \left|u(L,t)\right|^{p+1} dt
&\le c\int_0^t\left(\no{u(s)}_{L^{p+2}_x(0,L)}^{p+1}+\no{u(s)}_{L_x^{p+2}(0,L)}^p\no{u_x(s)}_{L_x^{\frac{p+2}{2}}(0,L)}\right)ds
\nn\\
&\le c + c\int_0^t\left(\no{u(s)}_{L^{p+2}_x(0,L)}^{p+2}+\no{u_x(s)}_{L_x^{\frac{p+2}{2}}(0,L)}^{\frac{p+2}{2}}\right)ds, \quad 0\le t\le T,
\end{align}
where $c$ depends on $L$, $p$ and $T$.
%
%
Assuming that $p\leq 2$ and using \eqref{nonlinbdrest2}, we have
\begin{align}\label{upubbar}
\text{Re}\left(\left(\kappa+i\beta\right)\int_0^t|u(L,s)|^pu(L,s)\bar{b}(s)ds\right)
&\leq
\left|\kappa+i\beta\right| \no{b}_{L^\infty(0,T)} \int_0^t \left|u(L,s)\right|^{p+1} ds 
\nn\\
& \le c + c\int_0^t\left(\no{u(s)}_{L^{p+2}_x(0,L)}^{p+2}+\no{u_x(s)}_{L_x^{2}(0,L)}^{2}\right) ds,
\end{align}
where $c$ depends on $\kappa,\beta,\no{b}_{L^\infty_t(0,L)},L, p, T$. Moreover, by the Cauchy-Schwarz and epsilon-Young inequalities,
    \begin{align}\label{apauxbar}
-\text{Re}\left(\left(\kappa+i\beta\right)\int_0^t |a(s)|^pa(s)\bar{u}_x(0,s) ds\right) 
&\le \left|\kappa + i \beta\right| N_0(t)\no{a}_{L^{2p+2}(0,T)}^{p+1}
\nn\\
 &\le \epsilon N_0^2(t)+c_\epsilon \no{a}_{L^{2p+2}(0,T)}^{2p+2}.
    \end{align}

Finally, regarding the third boundary term in \eqref{mul_H1est}, using the Cauchy-Schwarz and epsilon-Young inequalities we find
\begin{align}\label{be3}
-\gamma\text{Re}\int_0^t\left(u(L,t)\bar{b}(t)-a(t)u_x(0,t)\right) ds 
&\le \gamma D_L(t)\no{b}_{L^2(0,T)}+\gamma N_0(t) \no{a}_{L^2(0,T)}
\nn\\
& \le \epsilon D_L^2(t)+\epsilon N_0^2(t) + c_\epsilon\no{a}_{L^2(0,T)}^2 + c_\epsilon\no{b}_{L^2(0,T)}^2.
\end{align}

We proceed to the formation of yet one more energy by multiplying the CGL equation by $|u|^p\bar{u}$ and integrating over $(0,L)$ to obtain
\begin{equation}\label{mult_u_to_p}
        \text{Re}\int_0^L\left(u_t|u|^p\bar{u}-(\nu+i\alpha) u_{xx}|u|^p\bar{u}+\left(\kappa+i\beta\right)|u|^{2p+2}-\gamma|u|^{p+2}\right)dx=0.
\end{equation}
The first term can be written as
\begin{equation}
\text{Re}\int_0^Lu_t|u|^p\bar{u}dx
=
\frac 12 \int_0^L |u|^p  \, \p_t(|u|^2) dx
=
\frac{1}{p+2} \frac{d}{dt} \int_0^L \left(|u|^2\right)^{\frac{p+2}{2}} dx
=
\frac{1}{p+2}\frac{d}{dt}\no{u}_{L_x^{p+2}(0,L)}^{p+2}
\end{equation}
while for the second term \eqref{ibp-en} yields
    \begin{align}
        -\text{Re}\Big((\nu+i\alpha)\int_0^Lu_{xx}|u|^p\bar{u}dx\Big)
        &=
         - \nu \text{Re} \int_0^L \bar u_{xx}|u|^p u dx - \alpha \, \text{Im} \int_0^L \bar u_{xx}|u|^p u dx
\nn\\
        &=\nu\text{Re}\int_0^L\left(\tfrac{p+2}{2}|u|^p|u_x|^2+\tfrac{p}{2}u^2|u|^{p-2}\bar{u}_x^2\right)dx
       + \alpha \, \text{Im}\int_0^L \tfrac{p}{2}u^2|u|^{p-2}\bar{u}_x^2 dx\nn\\
        &\quad -\text{Re}\left[ (\nu-i\alpha)\left(\left|u(L,t)\right|^pu(L,t)\bar{b}(t)-|a(t)|^pa(t)\bar{u}_x(0,t)\right)\right].
    \end{align}
As in \eqref{upubbar} and \eqref{apauxbar}, we have
\begin{align}\label{upubbar2}
       &\quad \left|\text{Re}\int_0^t\left((\nu-i\alpha)\left(|u(L,s)|^pu(L,s)\bar{b}(s)-|a(s)|^pa(s)\bar{u}_x(0,s)\right)\right)ds\right|\nn\\
        &\le c + c\int_0^t\left(\no{u(s)}_{L^{p+2}_x(0,L)}^{p+2}+\no{u_x(s)}_{L_x^{2}(0,L)}^{2}\right) ds +\epsilon N_0^2(t)+c_\epsilon \no{a}_{L^{2p+2}(0,T)}^{2p+2}.
    \end{align}
Combining \eqref{mult_u_to_p}-\eqref{upubbar2} while also recalling \eqref{re-ineq} and the fact that $\kappa, \nu >0$, we infer
    \begin{align}\label{mult_u_to_p2}
       \frac{1}{p+2}\frac{d}{dt}\no{u}_{L_x^{p+2}(0,L)}^{p+2}\le &-\alpha \, \text{Im}\int_0^L \tfrac{p}{2}u^2|u|^{p-2}\bar{u}_x^2 dx
       -\kappa\no{u}_{L^{2p+2}_x(0,L)}^{2p+2}+\gamma\no{u}_{L_x^{p+2}(0,L)}^{p+2}\nn\\
        &+\text{Re}\left((\nu-i\alpha)\left(\left|u(L,t)\right|^pu(L,t)\bar{b}(t)-|a(t)|^pa(t)\bar{u}_x(0,t)\right)\right).
    \end{align}

Multiplying \eqref{mul_H1est} by $\alpha$, \eqref{mult_u_to_p2} by $\beta$ and adding the resulting inequalities, we find
    \begin{align}\label{e1d-eq}
&\quad
\frac{\alpha}{2}\frac{d}{dt}\no{u_x(t)}^2_{L^2_x(0,L)} + \frac{\beta}{p+2}\frac{d}{dt}\no{u(t)}_{L_x^{p+2}(0,L)}^{p+2}
\nn\\
&\le -\beta\kappa\no{u(t)}_{L^{2p+2}_x(0,L)}^{2p+2}+\beta\gamma\no{u(t)}_{L_x^{p+2}(0,L)}^{p+2}
-\alpha\nu\no{u_{xx}(t)}_{L^2_x(0,L)}^2+\alpha\gamma \no{u_x(t)}_{L^2_x(0,L)}^2 
\nn\\
&\quad
+ \alpha \, \text{Re}\left(u_t(L,t)\bar{b}(t)-a'(t)\bar{u}_x(0,t)\right)
+ \alpha \,  \text{Re}\left[\left(\kappa+i\beta\right)\left(\left|u(L,t)\right|^pu(L,t)\bar{b}(t)-|a(t)|^pa(t)\bar{u}_x(0,t)\right)\right]
\nn\\
&\quad
- \alpha  \gamma \, \text{Re}\left(u(L,t)\bar{b}(t)-a(t)\bar{u}_x(0,t)\right)
+\beta \, \text{Re}\left((\nu-i\alpha)\left(\left|u(L,t)\right|^pu(L,t)\bar{b}(t)-|a(t)|^pa(t)\bar{u}_x(0,t)\right)\right).
    \end{align}
Letting
$$
E_1(t):=\frac{\alpha}{2}\no{u_x(t)}^2_{L^2_x(0,L)} + \frac{\beta}{p+2}\no{u(t)}_{L_x^{p+2}(0,L)}^{p+2},
$$
we integrate equation \eqref{e1d-eq} in $t$ to obtain
\begin{align}\label{e1d-eq2}
E_1(t)
&\le
E_1(0) -\beta\kappa \int_0^t \no{u(s)}_{L^{2p+2}_x(0,L)}^{2p+2} ds +\beta\gamma \int_0^t \no{u(s)}_{L_x^{p+2}(0,L)}^{p+2} ds
\nn\\
&\quad 
-\alpha \nu \int_0^t \no{u_{xx}(s)}_{L^2_x(0,L)}^2 ds+\alpha\gamma  \int_0^t \no{u_x(s)}_{L^2_x(0,L)}^2 ds
\nn\\
&\quad 
+ \alpha \, \text{Re}  \int_0^t \left(u_t(L,s)\bar{b}(s)-a'(s)\bar{u}_x(0,s)\right) ds
\nn\\
&\quad 
 + \alpha \,  \text{Re} \left( \left(\kappa+i\beta\right) \int_0^t  \left(\left|u(L,s)\right|^pu(L,s)\bar{b}(s)-|a(s)|^p a(s)\bar{u}_x(0,s)\right) ds \right)
\nn\\
&\quad  
- \alpha  \gamma \, \text{Re} \int_0^t \left(u(L,s)\bar{b}(s)-a(s)\bar{u}_x(0,s)\right) ds
\nn\\
&\quad  +\beta \, \text{Re}  \left((\nu-i\alpha) \int_0^t \left(\left|u(L,s)\right|^pu(L,s)\bar{b}(s)-|a(s)|^pa(s)\bar{u}_x(0,s)\right)  ds \right).
\end{align}
Hence, setting $E:=E_0+E_1$ and using inequality \eqref{L2defest} for $E_0$ along with the estimates \eqref{be1}, \eqref{be2}, \eqref{upubbar}, \eqref{apauxbar}, \eqref{be3} and \eqref{upubbar2} for the various boundary terms, we have
\begin{align}\label{EtE0}
E(t)
&\le
E(0) +2\gamma\int_0^t E_0(s) ds -\beta\kappa \int_0^t \no{u(s)}_{L^{2p+2}_x(0,L)}^{2p+2} ds +\beta\gamma \int_0^t \no{u(s)}_{L_x^{p+2}(0,L)}^{p+2} ds
\nn\\
&\quad 
-\alpha \nu \int_0^t \no{u_{xx}(s)}_{L^2_x(0,L)}^2 ds+\alpha\gamma  \int_0^t \no{u_x(s)}_{L^2_x(0,L)}^2 ds
\nn\\
&\quad
+2 \left|\nu+i\alpha\right| \left(\epsilon D_L(t)^2 + c_\epsilon \no{b}_{L^2(0,t)}^2  + \epsilon N_0(t)^2 + c_\epsilon \no{a}_{L^2(0,t)}^2\right)
\nn\\
&\quad 
+ \alpha \, \bigg(c  +  \epsilon \no{u(t)}_{H_x^1(0,L)}^2 + \left(1+ c^2 c_\epsilon\right) \no{b}_{H^1(0, T)}^2 
+
c \int_0^t\no{u(s)}_{H_x^1(0,L)}^2ds
\nn\\
&\hspace*{1.5cm}
+
\epsilon \int_0^t\left(\no{u_{xx}(s)}_{L_x^2(0,L)}^2+\no{u_x(s)}_{L_x^2(0,L)}^2\right)ds + c_\epsilon\no{a}_{H^1(0,T)}^2
\bigg)
\nn\\
&\quad 
 + \alpha
 \left(
 c + c\int_0^t\left(\no{u(s)}_{L^{p+2}_x(0,L)}^{p+2}+\no{u_x(s)}_{L_x^{2}(0,L)}^{2}\right) ds
+
\epsilon N_0^2(t)+c_\epsilon \no{a}_{L^{2p+2}(0,T)}^{2p+2}
\right)
\nn\\
&\quad  
+ \alpha  \left(\epsilon D_L^2(t)+\epsilon N_0^2(t) + c_\epsilon\no{a}_{L^2(0,T)}^2 + c_\epsilon\no{b}_{L^2(0,T)}^2\right)
\nn\\
&\quad  +\beta  \left(c + c\int_0^t\left(\no{u(s)}_{L^{p+2}_x(0,L)}^{p+2}+\no{u_x(s)}_{L_x^{2}(0,L)}^{2}\right) ds +\epsilon N_0^2(t)+c_\epsilon \no{a}_{L^{2p+2}(0,T)}^{2p+2}\right).
\end{align}
Choosing $\epsilon$ sufficiently small, using the bounds \eqref{DLtsquare} and \eqref{N0tsquare} for $D_L(t)$ and $N_0(t)$, and eliminating the third and fifth term on the right side of \eqref{EtE0} as we are working under the assumption of $\alpha > 0$ and $\beta \geq 0$, we deduce 
\begin{equation}
E(t)\le c + c\int_0^tE(s)ds, \quad 0\le t\le T,
\end{equation}
where $c$ depends on the fixed parameters $\epsilon, \no{a}_{H^1(0,T)},\no{b}_{H^1(0,T)},p,L,T, \kappa, \beta, \alpha, \nu$. Gronwall's inequality then implies that 
\begin{equation}
E(t)\le c \, e^{cT}, \quad 0\le t\le T,
\end{equation}
therefore, the local solution $H^1$ solution extends globally in time.

\begin{remark}[Range of $p$] It is well-known that the range of admissible values of $p$ for global well-posedness of the CGL equation is $(0,\infty)$ when one considers homogeneous boundary conditions. However, the situation is more complex in presence of an inhomogeneous Neumann boundary condition. The restriction $p\le 2$ (with $\beta>0$) in the above analysis is due to the estimate of the boundary terms that involve nonlinear traces such as \eqref{upubbar} and \eqref{upubbar2}.  If one were to consider an inhomogeneous Dirichlet initial-boundary value problem with boundary conditions $u(0,t)=a(t)$ and $u(L,t)=b(t)$ (instead of $u_x(0,t)=b(t)$), then one would have been able to extend the range of admissible values of $p$ for global well-posedness to $(0,\infty)$ in the case of $\beta>0$. Another alternative in order to extend the range of $p$ to $(0,\infty)$ when $\beta>0$ is to assume that $b(t)\equiv 0$, i.e. to use the homogeneous Neumann  boundary condition $u_x(0,t)=0$ at the right endpoint, while the left endpoint is still subject to an inhomogeneous Dirichlet boundary condition.
\end{remark}

\noindent
\textbf{The case of $\beta<0$.}
Similarly to \eqref{EtE0}, using the fact that $\beta<0$ while $\kappa, \nu >0$, we have
    \begin{align}
        &\quad\frac{1}{2}\no{u(t)}_{L_x^2(0,L)}^2+\frac{\alpha}{2}\no{u_x(t)}_{L_x^2(0,L)}^2 \nn\\
        &\le \frac{1}{2}\no{u_0}_{L^2(0,L)}^2 +\frac{\alpha}{2}\no{u_0'}_{L^2(0,L)}^2 +\alpha\gamma\int_0^t\no{u_x(s)}_{L^2_x(0,L)}^2ds-\beta \kappa\int_0^t\no{u(s)}_{L^{2p+2}_x(0,L)}^{2p+2}ds\nn\\
        &\quad-\alpha\nu\int_0^t\no{u_{xx}(s)}_{L^2_x(0,L)}^2ds -\frac{\beta}{p+2}\no{u(t)}_{L_x^{p+2}(0,L)}^{p+2}+\gamma\int_0^t\no{u(s)}_{L^2_x(0,L)}^2ds\nn\\
        &\quad+ c + c\int_0^t\left(\no{u(s)}_{L^{p+2}_x(0,L)}^{p+2}+\no{u_x(s)}_{L_x^{2}(0,L)}^{2}\right)+\epsilon N_0^2(t)+c_\epsilon \no{a}_{L^{2p+2}(0,T)}^{2p+2}\nn\\
         &\quad+\epsilon D_L^2(t)+\epsilon N_0^2(t) + c_\epsilon\no{a}_{L^2(0,T)}^2 + c_\epsilon\no{b}_{L^2(0,T)}^2\nn\\
         &\quad+c+\epsilon\no{u(t)}_{H_x^1(0,L)}^2+(1+c_\epsilon)\no{b}_{H_t^1(0,T)}^2+c\int_0^t\no{u(s)}_{H_x^1(0,L)}^2ds\nn\\
         &\quad+\epsilon \int_0^t\left(\no{u_{xx}(s)}_{L_x^2(0,L)}^2+\no{u_x(s)}_{L_x^2(0,L)}^2\right)ds +c_\epsilon\no{a}_{H^1(0,T)}^2\nn\\
         &\quad {-(\alpha\kappa+\beta\nu)\text{Re}\int_0^t\int_0^L\left(\tfrac{p+2}{2}|u(s)|^p|u_x(s)|^2+\tfrac{p}{2}u^2(s)|u(s)|^{p-2}\bar{u}_x^2(s)\right)dxds.}
         \label{be4}
    \end{align}
By the (inhomogeneous version of) Gagliardo-Nirenberg inequality, i.e., with the constant of estimate also depending on $L$,
\begin{equation}\label{gagfirst}
        -\frac{\beta}{p+2}\no{u(t)}_{L_x^{p+2}(0,L)}^{p+2}\le c_{p,L} \no{u(t)}_{H_x^1(0,L)}^{\frac{p}{2}} \no{u(t)}_{L^2_x(0,L)}^{\frac{p+4}{2}}
\end{equation}
so using the epsilon-Young inequality twice, once for the product $\no{u(t)}_{H_x^1(0,L)}^{\frac{p}{2}} \no{u(t)}_{L^2_x(0,L)}^{\frac{p+4}{2}}$ 
with $q=\frac 4p$ (this step requires $p<4$) and once again for the resulting product $c_\epsilon \no{u(t)}_{L^2_x(0,L)}^{\frac{2(p+4)}{4-p}}$ with $q=\frac{2(4-p)}{p+4}$ (this step requires {$p<\frac 43$}), we obtain
\begin{equation}
        -\frac{\beta}{p+2}\no{u(t)}_{L_x^{p+2}(0,L)}^{p+2}\le  c_\epsilon +\epsilon\no{u(t)}_{H_x^1(0,L)}^{2}+\epsilon\no{u(t)}_{L^2_x(0,L)}^{4}.
\end{equation}
Via the same argument, assuming that {$p<\frac{4}{5}$} we have
\begin{equation}
    \begin{aligned}
        -\beta \kappa\no{u(t)}_{L^{2p+2}_x(0,L)}^{2p+2} \le c_{p,L}\no{u(t)}_{H_x^2(0,L)}^{\frac{p}{2}}\no{u(t)}_{L^2_x(0,L)}^{\frac{3p+4}{2}}\le c_\epsilon+\epsilon\no{u(t)}_{H_x^2(0,L)}^{2}+\epsilon\no{u(t)}_{L^2_x(0,L)}^{4}.
    \end{aligned}
\end{equation}
In order to handle the term $\no{u(t)}_{L^2_x(0,L)}^{4}$, we observe that $D_L$ and $N_0$ are non-decreasing. Therefore, by the integral form of Gronwall's inequality applied to \eqref{L2defest0}, we infer
$$\no{u(t)}_{L^2_x(0,L)}^2\le \left(\no{u_0}_{L_x^2(0,L)}^2+2\left|\nu+i \alpha\right| \left(\no{b}_{L^2(0,T)}D_L(t)+\no{a}_{L^2(0,T)}N_0(t)\right)\right)e^{2\gamma t},\quad 0\leq t \le T.$$
Taking squares of both sides and using $(A+B)^2\le  2 \, (A^2+B^2)$, we find
$$\no{u(t)}_{L^2_x(0,L)}^4\lesssim 2\left(\no{u_0}_{L_x^2(0,L)}^4+8\left(\nu^2+\alpha^2\right)\left(\no{b}_{L^2(0,T)}^2D_L^2(t)+\no{a}_{L^2(0,T)}^2N^2_0(t)\right)\right)e^{4\gamma T}$$
i.e. we can control the \textit{fourth power} of the $L^2$ norm of $u$ by just the \textit{square} of $N_0$ via the bound
\begin{equation}\label{l2-est-b<0}
\no{u(t)}_{L^2_x(0,L)}^4\le c+cD_L^2(t)+cN_0^2(t),
\end{equation}
where $c$ denotes a constant depending on the fixed parameters  $T,\nu,\alpha,\no{a}_{L^2(0,T)}, \no{b}_{L^2(0,T)}$ and $\no{u_0}_{L^2_x(0,L)}$.

\begin{remark}
    It is important to notice that even if we do not have full control of the $L^2$ norm, we still have a mild control over it, see the fourth power on the left side of inequality \eqref{l2-est-b<0}. That is, although the above bound of the $L^2$ norm is not uniform in $t$ (as it would have been in the case of zero boundary conditions), it still allows us to control that norm mildly thanks to the bounds \eqref{DLtsquare} and \eqref{N0tsquare} for $D_L(t)$ and $N_0(t)$.
\end{remark}

Furthermore,  by H\"older's inequality,
\begin{align}
-(\alpha\kappa+\beta\nu)\text{Re}\int_0^L\left(\tfrac{p+2}{2}|u|^p|u_x|^2+\tfrac{p}{2}u^2|u|^{p-2}\bar{u}_x^2\right)dx 
&\le (p+1)|\alpha\kappa+\beta\nu|\int_0^L|u|^p|u_x|^2dx
\nn\\
&\le (p+1)|\alpha\kappa+\beta\nu|\no{u(t)}_{L^{p+2}_x(0,L)}^p\no{u_x(t)}_{L_x^{p+2}(0,L)}^2
\end{align}
and the right side can be handled via the Gagliardo-Nirenberg and epsilon-Young inequalities as above, after assuming that {$p<\frac{4}{5}$}. More precisely, by the Gagliardo-Nirenberg inequality,
$$
\no{u_x}_{L^{p+2}} \leq \no{u}_{H^{1, p+2}} \lesssim \no{u}_{L^2}^{\frac{p+4}{4(p+2)}} \no{u}_{H^2}^{\frac{3p+4}{4(p+2)}},
\quad
\no{u}_{L^{p+2}} \lesssim \no{u}_{L^2}^{\frac{3p+8}{4(p+2)}} \no{u}_{H^2}^{\frac{p}{4(p+2)}}
$$
hence, if $p<\frac 45$, the epsilon-Young inequality yields
\begin{equation}
\no{u(t)}_{L^{p+2}_x(0,L)}^p\no{u_x(t)}_{L_x^{p+2}(0,L)}^2
\le c\no{u(t)}_{H_x^2(0,L)}^{1+\frac{p}{4}}\no{u(t)}_{L^2(0,L)}^{1+\frac{3p}{4}}
\le c_\epsilon+\epsilon\no{u(t)}_{H_x^2(0,L)}^{2}+\epsilon\no{u(t)}_{L^2_x(0,L)}^{4}.
\end{equation}
Again, the fourth power of the $L^2$ norm of $u$ on the right side can be controlled only by the square of $N_0$ thanks to the bound \eqref{l2-est-b<0}.

In view of the above calculations, choosing $\epsilon$ sufficiently small so that the various time integrals of the $L^2$ norm of $u_{xx}$ are overall multiplied by a negative coefficient, we eliminate the terms on the right side of~\eqref{be4} that have a negative sign to infer the estimate 
$$\no{u(t)}_{H^1_x(0,L)}^2\le c + c\int_0^t\no{u(s)}_{H^1_x(0,L)}^2ds, \quad 0\le t \le T.$$ Therefore,  by Gronwall's inequality,   local $H^1$ solutions can be extended globally if $\beta<0$ and $p\in (0,\frac{4}{5})$.

\begin{remark}
    For homogeneous boundary value problems, the $L^2$ norm of the solution is automatically bounded on any interval $[0,T]$.  However, from the $L^2$ estimate we see that this is no longer readily available in the case of inhomogeneous boundary conditions. This is a big challenge for controlling the nonlinearity when $\beta<0$  since one generally uses Gagliardo-Nirenberg inequalities together with boundedness of the $L^2$ norm to achieve this goal. Even in that case, $p$ is usually restricted to the range $(0,4)$ unless initial datum is assumed to be small.  For the inhomogeneous initial-boundary value problem, it is a highly nontrivial problem to prove global well-posedness when $\beta<0$ in absence of an apriori bound on the $L^2$ norm. 
 Hence, in the presence of inhomogeneous boundary conditions, it is more reasonable to expect a global well-posedness result for only mild cases of $p$.
\end{remark}

\section{Chaos suppression}
\label{stab-s}
In this section, we design a boundary feedback stabilizer for the CGL equation with the property that the controller uses only a finite number of Fourier modes of the state.  More precisely, we are concerned with the stabilization of the following  nonlinear model:
\begin{eqnarray} \label{pde_nonlin}
	\begin{cases}
		u_t - (\nu + i\alpha) u_{xx} + (\kappa + i\beta) |u|^p u - \gamma u = 0, &(x,t) \in(0,L) \times (0,T), \\
		u(0,t) = 0, \quad u_x(L,t) = b(t), & t \in (0,T),\\
		u(x,0) = u_0(x), &x \in (0,L),
	\end{cases}
\end{eqnarray}
where, $\alpha, \beta \in \mathbb{R}$, $\kappa, \gamma > 0$, $0 < p < 4$. $b(t) = b(P_N u(\cdot,t))$ is a nonlocal Neumann type feedback control input that involves finitely many Fourier modes of $u$. {Fig. \ref{fig:finite-dimensional-feedback} illustrates the finite-dimensional closed-loop feedback mechanism, emphasizing that the system is controlled through only finitely many Fourier modes.

\begin{figure}[!ht]
\centering

\begin{tikzpicture}[
    >=Latex,
    line cap=round,
    line join=round,
    block/.style={
        draw,
        rounded corners=3pt,
        align=center,
        minimum height=1.45cm,
        text width=4.5cm,
        fill=gray!8
    },
    arrow/.style={-{Latex[length=2.8mm,width=2mm]}, semithick}
]

\node[block, text width=4.8cm] (plant) at (0,0)
{Evolution of\\[1mm]
the state $u$};

\node[block, text width=4.8cm] (ctrl) at (-4,-2.5)
{Construction of the\\[1mm] controller $b(t)=b(P_Nu(\cdot,t))$};

\node[block] (obs) at (4,-2.5)
{In-domain modal\\[1mm]
observables $P_Nu$};

\coordinate (rightBottom) at ($(obs.east)+(1.8,0)$);
\coordinate (rightTop) at (rightBottom |- plant.east);

\coordinate (leftA) at ($(ctrl.west)+(-1.8,0)$);
\coordinate (leftB) at ($(leftA |- plant.west)$);

\draw[arrow] (plant.east) -- (rightTop)
              -- (rightBottom)
              -- (obs.east);

\draw[arrow] (obs.west) -- (ctrl.east);

\draw[arrow] (ctrl.west) -- (leftA)
              -- (leftB)
              -- (plant.west);

\end{tikzpicture}

\caption{Schematic illustration of the finite-dimensional feedback mechanism. The system is fed back through finitely many observed modes, represented by $P_Nu$, while the controller is constructed by means of a backstepping-type design.}
\label{fig:finite-dimensional-feedback}
\end{figure}

To develop a finite-dimensional controller, we follow a variation of the classical backstepping approach. Standard backstepping controller design (see e.g., \cite{krsticbook,bssurvey} and references therein) is based on two key elements: (i) a target model which features an interior damping term that cancels the destabilizing effects and therefore acts as a stabilizing mechanism, and (ii) a suitable coordinate transformation, typically a Volterra-type integral transformation of the second kind, which reflects the effect of the control input of the form $b(u)$ in the original model as an interior damping effect in the target model. Then, the stability analysis is performed for the target system. The invertibility of the transformation in the appropriate function space, which can be established by standard functional analytic tools, allows the resulting stability property to be transferred back to the original system.

The main difference between our method and the standard backstepping approach lies in the novel choice of these key elements, which enables the control mechanism to be expressed in a finite-dimensional form.
\begin{itemize}
    \item [(i)] In contrast to the classical backstepping strategy, we choose the target model with a finite-rank interior damping term involving the projection $P_N$ with the coefficient $\mu > 0$. Since the instability of the zero equilibrium is caused by only finitely many modal components that grow in time, a finite-dimensional interior damping effect can suppress these unstable modes under suitable choices of sufficiently large $\mu$ and $N$.
    \item [(ii)] We define the backstepping transformation involving a projection operator $P_N$. Defining the transformation in this way is crucial for two reasons. First, it enables the control input to be expressed in a finite-dimensional form. Second, the transformation allows the effect of the finite-dimensional control input $b(P_N u)$ in the original model to be reflected to the target model as a finite-rank interior damping mechanism so that we recover the target model with the structure described in."
\end{itemize}

Fig. \ref{fig:finite-dimensional-backstepping} illustrates this process, showing how the original model and the target model are connected through the transformations $T_N$ and $T_N^{-1}$.

\begin{figure}[!ht]
    \centering
\begin{tikzpicture}[
    >=Latex,
    line cap=round,
    line join=round,
    block/.style={
        draw,
        rounded corners=3pt,
        align=center,
        minimum height=1.45cm,
        fill=gray!8
    },
    arrow/.style={-{Latex[length=2.8mm,width=2mm]}, semithick}
]

\node[block, text width=5cm] (orig) at (-5,0)
{Original model with \\[1mm]
boundary control $b(P_Nu)$};

\node[block, text width=5cm] (target) at (5,0)
{Target model with \\[1mm]
in-domain damping $\mu P_N w$};

\draw[arrow, bend left=20] (orig) to
node[above, align=center]
{$w=T_N^{-1}u$}
(target);

\draw[arrow, bend left=20] (target) to
node[below, align=center]
{$u=T_Nw$}
(orig);
\end{tikzpicture}

\caption{Schematic illustration of the finite-dimensional backstepping strategy. The original model with state $u$ and the target model with state $w$ are connected through the transformation $T_N$ involving projections and its bounded inverse $T_N^{-1}$. The boundedness of $T_N$ and $T_N^{-1}$ are used to transfer the stability properties of the target system to the original one. The transformation, together with its inverse, is also used to construct the finite-dimensional feedback law.}
\label{fig:finite-dimensional-backstepping}
\end{figure}

\begin{remark}
        We note that the controller designed via the proposed finite-dimensional backstepping strategy can be viewed as a finite-dimensional truncation of the feedback law arising from the classical backstepping transformation. Indeed, when one takes $N = \infty$, the projection operator $P_N$ becomes the identity operator and the resulting control law takes the form of the classical full-state backstepping feedback law. From a practical point of view, our approach is more efficient than the full-state backstepping controller design, since the design requires only finitely many modal observations of the state.
    \end{remark}
}
\noindent\textbf{Outline of finite-dimensional controller design.}  To construct a finite-dimensional controller and ensure the well-posedness of the closed-loop system, we proceed as follows. First, in Section \ref{bs-trans}, we introduce the key elements of the control design: a Volterra-type transformation that incorporates a projection operator and a target model which features a finite-parameter interior damping term that acts as a stabilizing mechanism. We derive sufficient conditions for the kernel of the transformation to ensure that it effectively maps the original system to the target model. Next, in Section \ref{bdd-inv}, we establish the bounded invertibility of the aforementioned transformation on the spaces $H^\ell(0,L)$, $\ell = 0, 1, 2$. This guarantees that the target model and the original system exhibit the same long-term behavior. A critical challenge arises due to the projection operator in the sought-after transformation, which may not be invertible for certain values of $\mu$ and $N$. We exclude these problematic values, defining the remaining pairs as admissible.  Section \ref{ctrl-design} presents the finite-dimensional feedback control law. In Section \ref{wp-s}, the well-posedness for general inhomogeneous boundary conditions, $u(0,t) = a(t)$ and $u_x(L,t) = b(t)$ was established through analysis of a continuous data-to-solution map 
\[
u \mapsto {\Phi}[u_0, a, b, f(u)].
\]
Our controller is of feedback type and is actuated only at the right endpoint of the domain, while we impose homogeneous boundary condition at the left endpoint. Thus, in this setting, the data-to-solution map takes the closed-loop form 
\[
u \mapsto \Phi[u_0, 0, b_u, f(u)],
\]
where $b_u=b(P_Nu(t))$ is the sought-after feedback that depends on the state $u$. Moreover, due to the nature of the boundary controller, $b_u$ is nonlocal.  In above formulation, one may simply take $f(u) = 0$ for the linearized problem. In Section \ref{local-wp}, we establish the local-in-time well-posedness of both the linear and nonlinear closed-loop systems using a contraction argument on the mapping $\Phi$, without requiring any smallness condition on initial datum. In Section \ref{stabilization}, we extend these results to demonstrate the stabilization of both linear and nonlinear systems. Notably, this step also extends the local-in-time solution to a global one. Uniqueness of global in-time solution is proved in Section \ref{uniqueness-closed-loop}.
\subsection{Kernel and the linear target model}
\label{bs-trans}
First, we will study the linearized model below:
\begin{eqnarray} \label{pde_lin}
	\begin{cases}
		u_t - (\nu + i\alpha) u_{xx} - \gamma u = 0, &(x,t) \in(0,L) \times (0,T), \\
		u(0,t) = 0, \quad u_x(L,t) = b_u(t), & t \in (0,T),\\
		u(x,0) = u_0(x), &x \in (0,L).
	\end{cases}
\end{eqnarray}
We set the transformation
\begin{equation}  \label{backstepping-trans}
	\begin{split}
		u(x,t) &= w(x,t) + \int_0^x k(x,y) P_N w(y,t) dy \\
		&=: (T_N w)(x,t),
	\end{split}
\end{equation}
involving the projection operator $P_N : L^2(0,L) \to \text{span}\{e_j\}_{j=1}^N$. If the kernel $k$ satisfies the following boundary value problem (see Appendix \ref{app-kernel} for details)
\begin{eqnarray}\label{kernel_pde}
	\begin{cases}
		(\nu + i \alpha) (k_{xx} - k_{yy}) + \mu k = 0, \\
		k(x,0) = 0, \quad k(x,x) = -\dfrac{\mu x}{2(\nu + i \alpha)},
	\end{cases}
    (x,y) \in \mathcal T
\end{eqnarray}
where $\mathcal T := \{ (x,y) \, | \, 0 \leq x \leq L, 0 \leq y \leq x \}$, then \eqref{pde_lin} can be obtained from the following so called ``linear target model'' under the transformation \eqref{backstepping-trans}, 
\begin{eqnarray} \label{pde_tarlin}
	\begin{cases}
		w_t - (\nu + i\alpha) w_{xx} - \gamma w + \mu P_N w = 0, &(x,t) \in(0,L) \times (0,T), \\
		w(0,t) = w_x(L,t) = 0, & t \in (0,T),\\
		w(x,0) = w_0(x), &x \in (0,L).
	\end{cases}
\end{eqnarray}
Here, the initial datum $w_0$ depends on $u_0$ and its explicit representation will be given in Section \ref{bdd-inv}. 

Observe that the target model \eqref{pde_tarlin} retains the general structure of the linear plant \eqref{pde_lin} together with homogeneous versions of the same type of boundary conditions in \eqref{pde_lin} and an additional finite rank term $\mu P_N w$, involving projection of the state $w$ onto the finite-dimensional space $\text{span}\{e_j\}_{j=1}^N$. The purpose of this term is to eliminate destabilizing, possibly chaotic, effects that occur due to the choice of problem parameters. Notice that the differential operator
\begin{equation*}
    A = (\nu + i\alpha) \partial_x^2 + \gamma I
\end{equation*}
with domain $D(A) = \{ \varphi \in H^2(0,L) \, | \, \varphi(0) = \varphi^\prime(L) = 0 \}$ has a discrete spectrum consisting of a countable set of eigenvalues. Moreover, only finitely many of these eigenvalues lie in the right half of the complex plane. Therefore, for suitably chosen $\mu$ and $N$, the term $\mu P_N w$ is capable of shifting the eigenvalues with real parts to the left complex half-plane which eventually amounts to stabilize the associated unstable Fourier modes and asymptotically steer solutions uniformly to the zero equilibrium. 

It is well-known that the kernel model \eqref{kernel_pde} can be transformed into an integral equation and employing the method of successive approximations (see, e.g., \cite[Section 4]{krsticbook}), a solution to \eqref{kernel_pde} can be expressed in series form as follows:
\begin{equation}
	\label{ksolrep}
	k(x,y) = -\frac{\mu y}{2 (\nu + i \alpha)} \sum\limits_{m=0}^\infty \left(-\frac{\mu}{4 (\nu + i \alpha)}\right)^m \frac{(x^2 - y^2)^m}{m! (m+1)!},
\end{equation}
in which the sum is uniformly and absolutely convergent on $\mathcal T$.

\subsection{Bounded invertibility of $T_N$}
\label{bdd-inv}
It is crucial to establish the invertibility of the transformation \eqref{backstepping-trans} with a bounded inverse on relevant spaces, as this plays two essential roles in our analysis. First, the boundedness of $T_N$ together with the boundedness of its inverse $T_N^{-1}$ ensures that the original model and the target model have the same asymptotic behavior. Second, the explicit representation of $T_N^{-1}$ obtained at the end of this section enables the boundary feedback law to be expressed in terms of only finitely many modal components. Note that if the transformation \eqref{backstepping-trans} did not involve a projection operator, one could prove bounded invertibility on $L^2$ based Sobolev spaces by employing standard functional analytical tools (e.g., see \cite{Liu03,OzBa}). However, in the present case, there exist certain conditions on $\mu$ and $N$, as detailed below, for which the transformation is not necessarily surjective, hence it may not invertible. This motivates us to make the following definition to exclude ill-suited $(\mu,N)$ pairs, and classify the remaining ones as admissible.

\begin{definition}[Admissible decay rate-mode pair]\label{ratemodepair} Let $\mu>0$ and $N\in\mathbb{Z}_+$. If $T_N$ is invertible with a bounded inverse, then we say $(\mu,N)$ is an admissible decay rate-mode pair.
\end{definition}

For given $\mu$, it is important to know that whether there exists an $N\in\mathbb{Z}_+$ such that $(\mu,N)$ is admissible. Indeed, the situation is better and one can show as in the following lemma that each $\mu$ is associated with infinitely many admissible $(\mu,N)$ pairs.

\begin{lemma} \label{invlem}
    Let $\mu>0$, $k = k(x,y;\mu)$ be a smooth kernel solving \eqref{kernel_pde} and $T_N$ be given by \eqref{backstepping-trans}. Then, there exist infinitely many $N \in \mathbb{Z}_+$ such that $T_N$ is invertible on $H^\ell(0,L)$, $\ell \geq 0$.
\end{lemma}

\begin{proof}
    Suppose to the contrary that there exists some $\mu_0 > 0$ for which $T_N$ is invertible only for finitely many values of $N$. This implies existence of some $N_0 \in\mathbb{Z}_+$ such that, for all $N > N_0$, $T_N$ fails to be invertible. Then, for all $N > N_0$, $0 \in \sigma(T_N)$. Let $T$ be the standard backstepping transformation. Note that $T_N \to T$ uniformly in operator norm in relevant $L^2$ based Sobolev spaces. Therefore, the above argument implies $0 \in \sigma(T)$, from which we infer $T$ is not invertible. However, this cannot be true since $T$ is a Volterra type transformation of second kind with a smooth kernel.
\end{proof}

It is important to know how to choose admissible decay rate-mode pairs in practice. For this purpose, we state the following lemma which provides a sufficient condition to guarantee the invertibility of $T_N$, therefore the pair $(\mu,N)$ to be an admissible pair.

\begin{lemma}
    \label{invlem-2}
    Let $\Upsilon_j : L^2(0,L) \to H^\ell(0,L)$, $\ell\in \{0,1,2\}$, be the linear bounded operators defined by
    \begin{equation}
    \label{Phij}
    \begin{split}
		\Upsilon_0 &:= 0, \\
		\Upsilon_{j}\varphi &= (I - \Upsilon_{j-1})[K_{j}\varphi] -\frac{\left((I - \Upsilon_{j-1})[K_{j}\varphi],e_{j}\right)_2}{1+\left((I-\Upsilon_{j-1})[K e_{j}],e_{j}\right)_2} \times(I - \Upsilon_{j-1})[K e_{j}],
    \end{split}
	\end{equation}
    for $j = 1, \dotsc, N$, where $(K \varphi)(x) := \int_0^x k(x,y;\mu) \varphi (y) dy$ and satisfies
    \begin{equation} \label{invexists}
	   \left ((I-\Upsilon_{j-1})[K e_{j}],e_{j}\right)_2 \neq -1, \quad j = 1, \dotsc, N. 
    \end{equation} Then, $T_N$ is invertible with $T_N^{-1}=I - \Upsilon_N$.
\end{lemma}

{The proof of Lemma \ref{invlem-2} follows by induction. The ideas for this induction argument can be found in \cite[Section 2.2]{KaOzY}. For the sake of completeness, we briefly present the first steps of the induction argument to clarify how the condition \eqref{invexists} arises and how the bounded operator $\Upsilon_N$ can be constructed provided that \eqref{invexists} holds true. To this end, let us consider the first step, i.e., take $N = 1$, and construct the bounded inverse of $T_1 = I + K_1$. Let $\psi \in L^2(0,L)$ and set
    \begin{equation} \label{invlem-1}
        \varphi = T_1 \psi = \psi + K_1 \psi.
    \end{equation}
    Note that $\varphi \in L^2(0,L)$ since the backstepping kernel $k$ is a smooth function. Let us set $v = K_1 \psi$. Then, we write $\psi = \varphi - v$ and substitute this into \eqref{invlem-1} to get
    \begin{equation} \label{invlem-22}
        v = K_1 \varphi - K_1 v = K_1 \varphi - \hat{v}_1 Ke_1.
    \end{equation}
    Let us take $L^2(0,L)-$inner product of both sides by $e_1$. Then, we obtain
    \begin{equation} \label{invlem-3}
        \hat{v}_1 (Te_1,e_1)_2 = \hat\varphi_1 (K e_1 ,e_1)_2
    \end{equation}
    Now, suppose $\mu$ be such that $(Te_1,e_1)_2 = 0$. Then, it follows from \eqref{invlem-3} that we have $\hat\varphi_1 = (\varphi,e_1)_2 = 0$, which implies that the range, $\mathcal{R}(T_1)$ of $T_1$ is $\mathcal{R}(T_1) = \text{span}\{e_1\}^\perp \subsetneq L^2(0,1)$. Therefore, $T_1$ fails to be surjective and consequently does not admit an inverse. Now, suppose that $(Te_1,e_1)_2 \neq 0$. Substituting $\hat{v}_1$ from \eqref{invlem-3} into \eqref{invlem-22} we obtain
    \begin{equation}
        v = K_1 \varphi - \frac{\hat\varphi_1 (Ke_1, e_1)_2}{(Te_1,e_1)_2} Ke_1 = K_1 \varphi - \frac{(Ke_1, e_1)_2}{(Te_1,e_1)_2} K_1 \varphi = \frac{1}{(Te_1,e_1)_2} K_1 \varphi.
    \end{equation}
    Hence,
    \begin{equation*}
        \psi = \varphi - v = \varphi -\frac{1}{(Te_1,e_1)_2} K_1 \varphi = \left(I - \frac{1}{(Te_1,e_1)_2}K_1\right)\varphi.
    \end{equation*}
    Now, we define the operator
    \begin{equation} \label{invlem-4}
        \Upsilon_1 := \frac{1}{(Te_1,e_1)_2}K_1,
    \end{equation}
    so that $\psi = (I - \Upsilon_1)\varphi$. It can be verified that $(I - \Upsilon_1)$ is a left and right inverse for the operator $T_1$.
    
    Now, let us consider the case $N = 2$ and construct the bounded inverse of $T_2$. We choose $\mu > 0$ such that $(Te_1,e_1)_2 \neq 0$. Therefore, $T_1^{-1} = I - \Upsilon_1$ exists as a bounded operator, where $\Upsilon_1$ is given by \eqref{invlem-4}. Let $\psi \in L^2(0,L)$ and define its range under $T_2$ as $\varphi = T_2 \psi = \psi + K_2 \psi$. Let us also set $v = K_2 \psi$ and write $\varphi = \psi + v$. Following the same arguments to those in \eqref{invlem-22} and also using that $K_2 v = K_1 v + \hat{v}_2 Ke_2$, we obtain
    \begin{equation} \label{invlem-5}
        v = K_2 \varphi - K_2 v \Rightarrow (I + K_1)v = K_2 \varphi - \hat v_2 Ke_2.
    \end{equation}
    Since $(Te_1,e_1)_2 \neq 0$, the inverse $T_1^{-1} = I - \Upsilon_1$ exists as a bounded operator so applying $T_1^{-1}$ to both sides of \eqref{invlem-4} yields
    \begin{equation} \label{invlem-5.5}
        v = (I - \Upsilon_1)[K_2 \varphi] - \hat v_2 (I - \Upsilon_1)[Ke_2].
    \end{equation}
    Now let us take $L^2-$inner product of both sides with $e_2$. Then, we obtain
    \begin{equation} \label{invlem-6}
        \hat v_2 \left(1 + \left( (I - \Upsilon_1)[Ke_2], e_2 \right)_2\right)= ( (I - \Upsilon_1)[K_2 \varphi] , e_2 )_2.
    \end{equation}
    If $\mu$ were such that $1 + \left( (I - \Upsilon_1)[Ke_2], e_2 \right)_2 = 0$, then the right hand side of \eqref{invlem-6} would become also zero and one can write
    \begin{equation} \label{invlem-7}
        \begin{split}
            0 &= ((I - \Upsilon_1)[K_2 \varphi] , e_2 )_2 \\
            &= \hat \varphi_1 ( (I - \Upsilon_1)[Ke_1] , e_2)_2 + \hat \varphi_2 ( (I - \Upsilon_1) [K e_2], e_2 )_2 \\
            &= \hat \varphi_1 ( (I - \Upsilon_1)[Ke_1] , e_2)_2 - \hat{\varphi}_2 \\
            &= \hat \varphi_1 ( (I - \Upsilon_1)[Te_1 - e_1] , e_2)_2 - \hat{\varphi}_2 \\
            &= \hat \varphi_1 ( \Upsilon_1 e_1 , e_2)_2 - \hat \varphi_2,
        \end{split}
    \end{equation}
    where we used linearity of $T_1^{-1}$ and $Te_1 = T_1 e_1$. Now, using the definition of $\Upsilon_1$ given by \eqref{invlem-4}, we deduce from \eqref{invlem-7} that
    \begin{equation*}
        0 = \hat \varphi_1 \frac{( K e_1 , e_2 )_2}{( T e_1 , e_1 )_2} - \hat \varphi_2 \Rightarrow \hat \varphi_1 \frac{( T e_1 , e_2 )_2}{( T e_1 , e_1 )_2} - \hat \varphi_2 = 0
    \end{equation*}
    or equivalently
    \begin{equation*}
        (\varphi,\alpha_2 e_1 - \alpha_1 e_2)_2 = 0, \quad \alpha_1 = (Te_1,e_1), \quad \alpha_2 = (Te_1,e_2).
    \end{equation*}
    Consequently, if $\hat v_2 \left(1 + \left( (I - \Upsilon_1)[Ke_2], e_2 \right)\right)_2 = 0$, then the range $\mathcal{R}(T_2)$ of $T_2$ is given by
    \begin{equation*}
        \mathcal{R}(T_2) = \left\{\varphi \in L^2(0,L) \, \bigg| \, (\varphi,\alpha_2 e_1 - \alpha_1 e_2)_2 = 0\right\} \subsetneq L^2(0,L).
    \end{equation*}
    Therefore, $T_2$ is not surjective, hence does not admit an inverse. Now, suppose that $\mu > 0$ is such that $1 + \left( (I - \Upsilon_1)[Ke_2], e_2 \right)_2 \neq 0$. Then, from \eqref{invlem-6}, we can find $\hat{v}_2$ as
    \begin{equation} \label{invlem-8}
        \hat{v}_2 = \frac{( (I - \Upsilon_1)[K_2 \varphi] , e_2 )_2}{1 + \left( (I - \Upsilon_1)[Ke_2], e_2 \right)_2}.
    \end{equation}
    Substituting \eqref{invlem-8} into \eqref{invlem-5.5}, we obtain
    \begin{equation*}
        v = (I - \Upsilon_1)[K_2 \varphi] - \frac{( (I - \Upsilon_1)[K_2 \varphi] , e_2 )_2}{1 + \left( (I - \Upsilon_1)[Ke_2], e_2 \right)_2} (I - \Upsilon_1)[Ke_2].
    \end{equation*}
    Taking into account the relation $\varphi = \psi + v$, which implies $\psi = \varphi - v$, we are motivated to define the operator $\Upsilon_2$ as
    \begin{equation*}
        \Upsilon_2 \varphi = (I - \Upsilon_1)[K_2 \varphi] - \frac{( (I - \Upsilon_1)[K_2 \varphi] , e_2 )_2}{1 + \left( (I - \Upsilon_1)[Ke_2], e_2 \right)_2} (I - \Upsilon_1)[Ke_2],
    \end{equation*}
    which is linear and bounded. Since $\varphi=T_2\psi$ and $\psi=(I-\Upsilon_2)\varphi$, the operator $I-\Upsilon_2$ is the candidate for the inverse of $T_2$. Indeed, one can verify by a direct computation that $(I-\Upsilon_2)T_2 = T_2(I-\Upsilon_2)=I$. Hence, $I-\Upsilon_2$ is the bounded inverse of $T_2$, provided that $1 + \left( (I - \Upsilon_1)[Ke_2], e_2 \right)_2 \neq 0$ holds true.
    
    Proceeding inductively and following the similar arguments as above, one can obtain the bounded linear operators $\Upsilon_j$, $j=1,\dotsc,N$, in the form
    \begin{equation*}
        \Upsilon_{j}\varphi = (I - \Upsilon_{j-1})[K_j \varphi] -\frac{\left((I - \Upsilon_{j-1})[K_{j}\varphi],e_{j}\right)_2}{1+\left((I-\Upsilon_{j-1})[K e_{j}],e_{j}\right)_2} \times(I - \Upsilon_{j-1})[K e_{j}],
    \end{equation*}
    provided that the corresponding denominators do not vanish.
    \begin{remark}
    \label{admissible-discuss}
    We would like to note that in practice it is very unlikely to find a pair $(\mu,N)$ which violates the condition
    \begin{equation} \label{admis-cond}
        1+\left((I-\Upsilon_{j-1})[K e_{j}],e_{j}\right)_2 \neq 0, \quad j = 1, \dotsc, N.
    \end{equation}
    It is a theoretical results which in practice occurs very rare. Moreover, the backstepping kernel $k$ depends continuously on $\mu$, therefore the scalar quantity appearing in \eqref{admis-cond} also varies continuously with respect to $\mu$. Thus, if \eqref{admis-cond} is violated, then one can slightly perturb $\mu$ so that the new pair becomes admissible.
    \end{remark}
    To support the discussion in Remark \ref{admissible-discuss}, we provide below a numerical illustration showing that how the scalar quantity in \eqref{admis-cond} varies with respect to $\mu$ for $N = 3$.
    \paragraph{A numerical illustration of the admissibility condition.} Consider a Ginzburg-Landau model with coefficients $\nu + i \alpha = 1 + 3i$ and $\gamma = 63$, posed on the unit interval, $L = 1$. Then $\lambda_3 = \frac{25\pi^2}{4}$, $\lambda_4 = \frac{49\pi^2}{4}$ so that
    \begin{equation*}
        \lambda_3 < \frac{\gamma}{\nu} = 63 < \lambda_{4}.
    \end{equation*}
    Hence, the instability level of the problem is $3$. In view of Theorem 1.3 (which will be proved below in Section 3.5), the zero equilibrium can be stabilized by a feedback law involving $N = 3$ Fourier modes with $\mu$ chosen from the interval
    \begin{equation} \label{mu-test}
        2(\gamma - \nu \lambda_1) \left(1 - \frac{1}{2N+1}\right)^{-1} \simeq 141.24 < \mu < 2\nu \lambda_{N+1} \simeq 241.81.
    \end{equation}
    To this end, suppose we want to implement the three-mode feedback law so that we need to consider the operator $T_3$ and guarantee its bounded invertibility. For this purpose, we consider the functions
    \begin{equation*}
        d_j: \mu \to 1+\left((I-\Upsilon_{j-1})[K e_{j}],e_{j}\right)_2
    \end{equation*}
    for each $j = 1,2,3$ where $\mu$ varies over the interval \eqref{mu-test}. Since the zeros of these functions correspond to the values of $\mu$ for which the sufficient condition in Lemma 3.2 fails, we investigate their behavior over the interval \eqref{mu-test}. The corresponding plots are represented in the Fig. \ref{fig:admisfigs}. The numerical results indicate that none of the functions $d_j$, $j=1,2,3$, vanishes in the prescribed range of $\mu$. Hence, for the considered example, any value of $\mu$ in the interval \eqref{mu-test} can be chosen without violating the admissibility condition in Lemma 3.2.
    \begin{figure}[!ht]
    \begin{subfigure}[b]{0.32\textwidth}
		\centering
		\includegraphics[width=\textwidth]{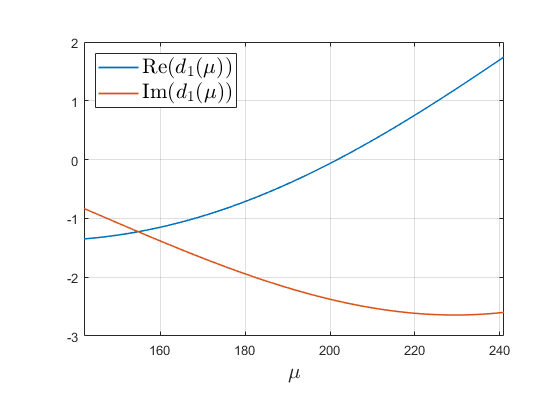}
	\end{subfigure}
	\begin{subfigure}[b]{0.32\textwidth}
		\centering
		\includegraphics[width=\textwidth]{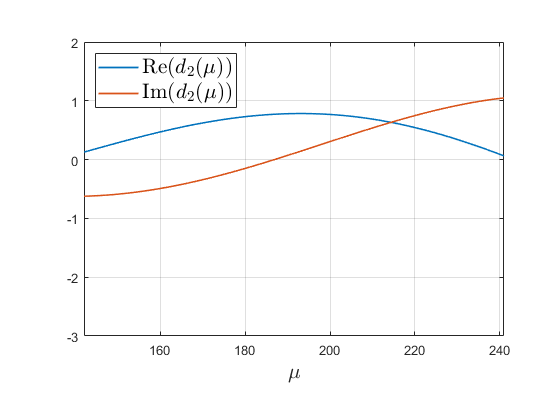}
	\end{subfigure}
	\begin{subfigure}[b]{0.32\textwidth}
		\centering
		\includegraphics[width=\textwidth]{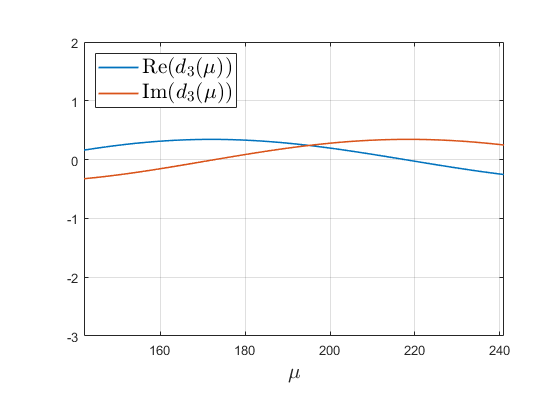}
	\end{subfigure}
    \caption{Real and imaginary parts of the functions $d_j(\mu)$, $j=1,2,3$ plotted over the interval \eqref{mu-test}. The plots indicate that there is no value of $\mu$ in the interval \eqref{mu-test} for which any of the functions $d_1(\mu)$, $d_2(\mu)$ or $d_3(\mu)$ vanishes. Hence, the condition \eqref{admis-cond} for the bounded invertibility of $T_3$ is satisfied throughout the considered range.}
    \label{fig:admisfigs}
\end{figure}}

The recursive representation obtained in Lemma~\ref{invlem-2} has several important consequences in our study, which we summarize in the following remarks.
\begin{remark}
    The representation $T_N^{-1} = I - \Upsilon_N$ will play a crucial role in the next section. Indeed, formulation of $\Upsilon_N$ given by \eqref{Phij} allows us to express the control law in terms of only finitely many Fourier modes which is one of the main objectives of this study.
\end{remark}
\begin{remark}
    Lemma \ref{invlem-2} provides an explicit representation of the inverse and this allows us to construct the controller numerically in Section \ref{numerical}.
\end{remark}

\subsection{Explicit representation of the boundary controller}
\label{ctrl-design}
To compute the boundary feedback actuated at the Neumann boundary condition, we first differentiate \eqref{backstepping-trans} with respect to $x$:
\begin{equation*}
	\begin{split}
		u_x(x,t) &= w_x(x,t) + \frac{\partial}{\partial x} \int_0^x k(x,y) P_N w(y,t) dy \\
		&= w_x(x,t) + \int_0^x k_x(x,y) P_N w(y,t) dy + k(x,x) P_N w(x,t).
	\end{split}
\end{equation*}
Then, taking $x = L$ and noting that $w_x(L,t) = 0$, $\displaystyle k(L,L) = -\frac{\mu L}{2 (\nu + i \alpha)}$, we get
\begin{equation*}
	u_x(L,t) = \int_0^L k_x(L,y) P_N w(y,t) dy -\frac{\mu L}{2 (\nu + i \alpha)} P_N w(x,t) \big|_{x = L}.
\end{equation*}
Using bounded invertibility of the transformation \eqref{backstepping-trans}, we replace $w$ with $w = (I - \Upsilon_N)u$, so that we can express the control input in feedback form as follows
\begin{equation} \label{ctrl1}
    u_x(L,t) = \int_0^L k_x(L,y) P_N [(I - \Upsilon_N)u](y,t) dy -\frac{\mu L}{2 (\nu + i \alpha)} P_N [(I - \Upsilon_N)u](x,t) \big|_{x = L}.
\end{equation}
Observe from \eqref{Phij} that $\Upsilon_N$ is a composition with the projection operator. Therefore, $\Upsilon_N u = \Upsilon_N P_N u$, and that  $$P_N[(I - \Upsilon_N)u] = (I - P_N \Upsilon_N)[P_N u] =: \Gamma_N [P_N u].$$ Hence, it follows from \eqref{ctrl1} that
\begin{equation}\label{controller}
    u_x(L,t) = \int_0^L k_x(L,y) \Gamma_N [P_N u](y,t) dy - \frac{\mu L}{2 (\nu + i \alpha)} \Gamma_N [P_N u](x,t) \big|_{x = L}.
\end{equation}
Consequently, the final form of the control law is of feedback type and it involves finitely many Fourier modes of the state of solution.

\subsection{Local solutions of the non-local boundary value problem in fractional Sobolev spaces}
\label{local-wp}
In this section, we establish the local existence of solutions for the initial-boundary value problem \eqref{pde_nonlin} with $$b(t)=b_u(t):=\int_0^L k_x(L,y) \Gamma_N [P_N u](y,t) dy - \frac{\mu L}{2 (\nu + i \alpha)} \Gamma_N [P_N u](x,t) \big|_{x = L}$$ in fractional Sobolev spaces. We consider the following initial-boundary value problem associated with the linear and nonlinear complex Ginzburg-Landau equations:
\begin{equation}\label{cgl-ibvp-nonlocal}
	\begin{aligned}
		&u_t - \left(\nu + i \alpha\right) u_{xx} - \gamma u + f(u) = 0, \quad (x, t) \in (0, L) \times (0, T),
		\\
		&u(x,0) = u_0(x),
		\\
		&u(0, t) = 0, \quad u_x(L, t) = b_u(t):=\int_0^L\xi(y)(\Gamma_NP_N u)(y,t)dy+\zeta(\Gamma_NP_N u)(L,t),
	\end{aligned} 
\end{equation}
where $\nu> 0$, $\alpha \in \mathbb R$, $\gamma \geq 0$, $\displaystyle\zeta=- \frac{\mu L}{2 (\nu + i \alpha)}\in \mathbb{C}$, $\xi=k_x(L,\cdot)\in C^\infty([0,L])$, $\Gamma_N$ is the linear mapping defined in the previous section, which acts as a bounded operator on Sobolev spaces $H_x^{s}(0,L)$ for every $s\ge 0$. We will simultaneously treat both the linear problem ($f(u)\equiv 0\Leftrightarrow \beta=\kappa=0$) and the nonlinear problem ($f(u)=\left(\kappa + i\beta\right) |u|^p u$, $\beta\in\mathbb{R}$, $\kappa,p>0$).  To this end, we make use of the solution map below:
\begin{equation}\label{it-map2}
	u(x, t) \mapsto \Phi\left[u_0, 0, b_u; f(u)\right],
\end{equation} where $\Phi$ is given by \eqref{Phi-def}.
Here, we assume $\frac 12 < s < \frac 32$, and look for solutions in the space $$X_T := C_t([0, T]; H_x^s(0, L)) \cap C_x([0, L]; H_t^{\frac{2s+1}{4}}(0, T))$$ for suitable $T>0$.  

\begin{lemma} Let $\frac 12 < s < \frac 32$ and $u\in X_T$, then $b_u\in H_t^{\frac{2s-1}{4}}(0,T)$. Furthermore, we have the estimate
	\begin{equation}
		\|b_u\|_{H_t^{\frac{2s-1}{4}}(0,T)}\lesssim T^{\frac{2}{2s+5}}\no{u}_{X_T}.
	\end{equation} 
\end{lemma}
\begin{proof} We write $b_u(t):=a_u(t)+\zeta(\Gamma_NP_N u)(L,t)$, where $a_u(t):=\int_0^L\xi(y)(\Gamma_NP_N u)(y,t)dy$.
	Taking $L^2_t(0,T)$ norm of $a_u$, we get
	\begin{equation}
		\begin{aligned}		
			\|a_u\|_{L_t^2(0,T)}^2&=\left\|\int_0^L\xi(y)(\Gamma_NP_N u)(y,\cdot)dy\right\|_{L_t^2(0,T)}^2=\int_0^T\left|\int_0^L\xi(y)(\Gamma_NP_N u)(y,t)dy\right|^2dt\\
			&\le \int_0^T\left(\int_0^L|\xi(y)|^2dy\right)\left(\int_0^L|(\Gamma_NP_N u)(y,t)|^2dy\right)dt \\
            &= \|\xi\|_{L^2(0,L)}^2\int_0^T\|(\Gamma_NP_N u)(\cdot,t)\|_{L_x^2(0,L)}^2dt\\
			&\le \|\xi\|_{L^2(0,L)}^2\int_0^T\|\Gamma_N\|_{L_x^2(0,L)\rightarrow L_x^2(0,L)}^2\|u(\cdot,t)\|_{L_x^2(0,L)}^2dt\\
			&=\|\xi\|_{L^2(0,L)}^2\|\Gamma_N\|_{L_x^2(0,L)\rightarrow L_x^2(0,L)}^2\int_0^L\int_0^T\left|u(x,t)\right|^2dtdx\\
			& \le L\|\xi\|_{L^2(0,L)}^2\|\Gamma_N\|_{L_x^2(0,L)\rightarrow L_x^2(0,L)}^2\no{u}_{L^\infty_x(0,L;L^2_x(0,T))}^2=c(\xi,\Gamma_N,L)\no{u}_{L^\infty_x(0,L;L^2_t(0,T))}^2.
		\end{aligned} 
	\end{equation} Taking square roots, we establish the $L_t^2(0,T)$ level estimate:
	\begin{equation}\label{auL2}
		\|a_u\|_{L_t^{2}(0,T)}\lesssim \no{u}_{L^\infty_x(0,L;L^2_t(0,T))}.
	\end{equation}
Next, we estimate $L^2_t(0,T)$ norm of $\displaystyle \frac{d}{dt}a_u$ by using similar arguments, and we have
\begin{equation}
		\begin{aligned}		
			\|a_u'\|_{L_t^2(0,T)}^2&=\left\|\int_0^L\xi(y)(\Gamma_NP_N u_t)(y,\cdot)dy\right\|_{L_t^2(0,T)}^2\\
            &\le \|\xi\|_{L^2(0,L)}^2\|\Gamma_N\|_{L_x^2(0,L)\rightarrow L_x^2(0,L)}^2\int_0^L\int_0^T|u_t(x,t)|^2dtdx\\
            &\le c(\xi,\Gamma_N,L)\|u_t\|_{L^\infty_x(0,L;L^2_t(0,T))}^2\le c(\xi,\Gamma_N,L)\no{u}_{L^\infty_x(0,L;H^1_t(0,T))}^2,
		\end{aligned} 
	\end{equation} and in view of \eqref{auL2}, we obtain
    	\begin{equation}\label{auH1}
		\|a_u\|_{H^1_t(0,T)}\lesssim \no{u}_{L^\infty_x(0,L;H^1_t(0,T))}.
	\end{equation}
    Moreover, interpolating between \eqref{auL2} and \eqref{auH1}, we establish
        	\begin{equation}\label{auHs}
		\|a_u\|_{H^{\frac{2s-1}{4}}_t(0,T)}\lesssim \no{u}_{L^\infty_x(0,L;H^{\frac{2s-1}{4}}_t(0,T))}.
	\end{equation}
Now, we wish to estimate the trace term in $b_u$ given below:
\begin{equation}
		\begin{aligned}		
			\|\zeta\Gamma_NP_Nu(L,\cdot)\|_{L_t^2(0,T)}^2&=|\zeta|^2\int_0^T\left|(\Gamma_NP_N u)(L,t)\right|^2dt.
		\end{aligned} 
	\end{equation}
Recall that $\Gamma_N = I-P_N\Upsilon_N$, therefore we have $$\Gamma_NP_Nu(L,t)=P_Nu(L,t)-P_N\Upsilon_Nu(L,t)=P_N\Gamma_Nu(L,t).$$ It is trivial to handle the first term at the right side above. So let us consider the term $P_N\Upsilon_Nu(L,t)$. We write $$P_N\Upsilon_Nu(L,t)= \sum_{j=1}^Nc_j(t)e_j(L);\quad c_j(t):=\langle \Upsilon_Nu(\cdot,t),e_j(\cdot)\rangle_2.$$
Observe that, in view of Cauchy-Schwarz inequality and boundedness of $\Upsilon_N$, we get $$|c_j(t)|\le \|\Upsilon_Nu\|_2\cdot \|e_j\|_2\le\|\Upsilon_N\|_{L^2_x(0,L)\rightarrow L^2_x(0,L)}\cdot \|u(\cdot,t)\|_2\cdot \|e_j\|_2.$$
Hence, 
\begin{equation}
		\begin{aligned}		\int_0^T\left|P_N\Upsilon_Nu(L,t)\right|^2dt& = \int_0^T\left|\sum_{j=1}^Nc_j(t)e_j(L)\right|^2dt\lesssim \sum_{j=1}^N|e_j(L)|^2\int_0^T\left|c_j(t)\right|^2dt\\
        &\le \sum_{j=1}^N|e_j(L)|^2\|e_j\|_2^2\|\Upsilon_N\|_{L^2_x(0,L)\rightarrow L^2_x(0,L)}^2\int_0^T\|u(\cdot,t)\|_2^2dt\\
        &\lesssim \no{u}_{L^\infty_x(0,L;L^2_t(0,T))}^2.
        \end{aligned} 
	\end{equation}
Similarly, $$\int_0^T\left|P_N\Upsilon_Nu_t(L,t)\right|^2dt\lesssim \no{u}_{L^\infty_x(0,L;H^1_t(0,T))}^2.$$
Combining above estimates, interpolating, we deduce
$$\|\zeta\Gamma_NP_N u(L,\cdot)\|_{H_t^{\frac{2s-1}{4}}(0,T)}\lesssim \no{u(L,\cdot)}_{H_t^{\frac{2s-1}{4}}(0,T)}+\no{u}_{L^\infty_x(0,L;H^{\frac{2s-1}{4}}_t(0,T))}\lesssim \no{u}_{L^\infty_x(0,L;H^{\frac{2s-1}{4}}_t(0,T))}.$$
We use \cite[Lemma 3.7]{bo2016} to get
$$\no{u}_{L^\infty_x(0,L;H^{\frac{2s-1}{4}}_t(0,T))}\le T^{\frac{2}{2s+5}}\no{u}_{L^\infty_x(0,L;H^{\frac{2s+1}{4}}_t(0,T))}.$$
\end{proof}

Let $\overline{B(0, \rho_T)} \subset X_T$ denote the closed ball in $X_T$ with center at zero and radius  
\begin{equation}\label{rst-defnew}
\begin{aligned}
\rho_T
:= 
2\mathcal B 
\no{u_0}_{H^s(0, L)}
\end{aligned}
\end{equation}
where
\begin{equation}\label{cr-defnew}
\mathcal B = \mathcal B(s, \lambda, L, \alpha, \nu, \gamma, T) 
=
2\left(1+c_{s, \gamma}\right) e^{\gamma T} 
\left[1 + c_2  + c_5 \left( c_2(s)  + c_2(s-1)\right) \right] .
\end{equation}
Then, for every $u \in \overline{B(0, \rho_T)}$, estimate \eqref{u-se} with $a\equiv 0$, $b=b_u$, $\mathfrak f = -\left(\kappa + i\beta\right) |u|^p u$ and $p$ satisfying \eqref{p-def} can be combined with Lemma 3.1 of \cite{bo2016} to yield 
\begin{align}\label{Phi-senew}
\no{\Phi\big[u_0, 0, b_u; -\left(\kappa + i\beta\right) |u|^p u\big](t)}_{X_T}
&\leq
\frac{\rho_T}{2}
+
c_5 c_{s, \gamma} T^{\frac{2}{2s+5}} \rho_T
+ 
2 c_{s, p} \left(1+c_{s, \gamma}\right) e^{\gamma T} (\kappa^2+\beta^2)^{\frac 12} 
\nn\\
&\quad
\cdot \sqrt T \left(\sqrt T + c_3  + c_5  \left( c_3(s) + c_3(s-1)\right) \right) \rho_T^{p+1}.
\end{align}
Therefore, choosing $T>0$ such that the right side of \eqref{Phi-senew} is less than $\rho_T$, namely
\begin{equation}\label{T-intonew}
2c_5 c_{s, \gamma} T^{\frac{2}{2s+5}}
+ 
4 c_{s, p} \left(1+c_{s, \gamma}\right) e^{\gamma T} (\kappa^2+\beta^2)^{\frac 12} 
\sqrt T \left(\sqrt T + c_3  + c_5  \left( c_3(s) + c_3(s-1)\right) \right) \rho_T^p
<
1, 
\end{equation}
we deduce that the solution map takes $\overline{B(0, \rho_T)}$ into $\overline{B(0, \rho_T)}$. It is important to note that the condition \eqref{T-intonew} can be fulfilled by an appropriate $T>0$ \textit{without a smallness condition on the data} due to the fact that the relevant constants are well-defined for all $T\geq 0$ and the left side can be made as small as necessary due to the presence of positive powers of $T$.  

Furthermore, for any $u_1, u_2 \in \overline{B(0, \rho_T)}$, estimate \eqref{u-se} for $b=b_{u_1}-b_{u_2}$ and $\mathfrak f = -\left(\kappa + i\beta\right) \left(|u_1|^p u_1 - |u_2|^p u_2\right)$ together with Lemma 3.1 of \cite{bo2016} imply
\begin{align}\label{Phi-diffnew}
&\quad
\no{\Phi\big[u_0, 0, b_{u_1}; -\left(\kappa + i\beta\right) |u_1|^p u_1\big](t) - \Phi\big[u_0, 0, b_{u_2}; -\left(\kappa + i\beta\right) |u_2|^p u_2\big](t)}_{X_T}
\nn\\
&=
\no{\Phi\big[0, 0, b_{u_1}-b_{u_2}; -\left(\kappa + i\beta\right) \left(|u_1|^p u_1 - |u_2|^p u_2\right)\big](t)}_{X_T}
\nn\\
&\leq
\Big[2 c_{s, p} (1+ c_{s, \gamma}) e^{\gamma T} (\kappa^2+\beta^2)^{\frac 12} 
\sqrt T \left(\sqrt T + c_3  + c_5  \left( c_3(s) + c_3(s-1)\right) \right)
2 \rho_T^p
+
c_5 c_{s, \gamma} T^{\frac{2}{2s+5}}\Big]\no{u_1-u_2}_{X_T}.
\end{align}
Hence, further restricting $T>0$ so that the coefficient of $\no{u_1-u_2}_{X_T}$ on the right side is strictly less than one, namely
\begin{equation}\label{T-contrnew}
2 c_{s, p} (1+ c_{s, \gamma}) e^{\gamma T} (\kappa^2+\beta^2)^{\frac 12} 
\sqrt T \left(\sqrt T + c_3  + c_5  \left( c_3(s) + c_3(s-1)\right) \right)
2 \rho_T^p
+
c_5 c_{s, \gamma} T^{\frac{2}{2s+5}}<1
\end{equation}
which, as in the case of \eqref{T-intonew} before, is realizable \textit{without a smallness condition on the data}, we conclude that the solution map is a contraction in $\overline{B(0, \rho_T)}$. In turn, Banach's fixed point theorem implies the existence of a unique fixed point of the solution map in $\overline{B(0, \rho_T)}$ which, as noted above, corresponds to a  unique solution in $\overline{B(0, \rho_T)}$ to the nonlinear initial-boundary value problem \eqref{cgl-ibvp-nonlocal}. This completes the proof of local existence. Global existence (for energy solutions) and global uniqueness will follow from the stabilization results in the subsequent sections.

We have the following result.

\begin{proposition}[Local solutions of the closed-loop problem in fractional spaces]
    Let $s\in (\frac{1}{2},\frac32)$, $s$ and $p$ satisfy the conditions of Theorem \ref{lwp-t}, and $T>0$ satisfy \eqref{T-intonew}, \eqref{Phi-diffnew}, and \eqref{T-contrnew}. Then, for initial data $u_0 \in H^s(0, L)$ satisfying the compatibility condition $u_0(0)=0$, the nonlinear initial-boundary value problem \eqref{cgl-ibvp-nonlocal} has a solution 
$
u \in X_T.
$ 
\end{proposition}

\subsection{Stabilization of zero equilibrium by a finite dimensional boundary feedback controller}
\label{stabilization}
In this part, we derive uniform exponential decay estimates in $H^1$ level under the non-local control law \eqref{controller}.  The uniform bounds we derive also extends the local in-time solutions as global solutions. A Poincar\'e-type inequality will be particularly useful in the estimates established here, which is stated below.  
\\

\noindent \paragraph{\bf Poincar\'e-type inequality} 
Let $w \in \{ \varphi \in H^1(0,L) \, | \, \varphi(0) = \varphi^\prime(0) = 0\}$ and denote $j$th Fourier mode of $w$ as $\hat w_j := (w,e_j)_2$. Then,
\begin{equation*}
	\begin{split}
		\|w - P_N w\|^2 &= \sum_{ j = N+1}^\infty |\hat{w}_j|^2 \\
		&= \sum_{ j = N+1}^\infty \frac{\lambda_j}{\lambda_j} |\hat{w}_j|^2, \quad \lambda_j = (2j - 1)^2 \lambda_1, \quad \lambda_1 = \frac{\pi^2}{4L^2}  \\
		&\leq \lambda_{N+1}^{-1} \sum_{N+1}^\infty \lambda_j |\hat{w}_j|^2 \|e_j\|^2 \\
		&= \lambda_{N+1}^{-1} \sum_{N+1}^\infty |\hat{w}_j|^2 \|e_j^\prime\|^2 \\
		&= \lambda_{N+1}^{-1} \left\|\sum_{N+1}^\infty \hat{w}_je_j^\prime \right\|^2 \\
		&\leq \lambda_{N+1}^{-1} \|w^\prime\|^2.
	\end{split}
\end{equation*}

\subsubsection{Linearized model (Proof of Theorem \ref{thm-stab-linear})} \label{linstab} Let $u_0 \in H^1(0,L)$. Then, bounded invertibility of the  transformation \eqref{backstepping-trans} implies $w_0 = T_N^{-1}u_0 \in H^1(0,L)$. We take $L^2$ inner product of the main equation \eqref{pde_tarlin} by $2w$:
\begin{equation} \label{l2-inner}
	2 \int_0^L w_t \overline w dx - 2(\nu + i\alpha) \int_0^L w_{xx} \overline w dx - 2 \gamma \int_0^L |w|^2 dx + 2\mu \int_0^L \overline w P_N w dx = 0.
\end{equation}
Integrating the second term by parts and taking the real parts in \eqref{l2-inner}, we get
\begin{equation} \label{l2-inner-2}
	\frac{d}{dt} \|w(t)\|_{L^2(0,L)}^2 + 2\nu \|w_x(t)\|_{L^2(0,L)}^2 - 2\gamma \|w(t)\|_{L^2(0,L)}^2 + 2\mu \int_0^L \overline w P_N w dx = 0.
\end{equation}
Using Cauchy-Schwarz inequality, Cauchy's inequality with $\epsilon_1 > 0$ and Poincar\'e-type inequality, the last term on the left side can be bounded from below as
\begin{equation}\label{ctrl_term}
	\begin{split}
		2\mu \int_0^L \overline w P_Nw dx
		&= -2\mu \int_0^L \overline w \left(w - P_Nw\right) dx + 2\mu\|w(t)\|_{L^2(0,L)}^2 \\
		&\geq (2\mu - \epsilon_1\mu) \|w(t)\|_{L^2(0,L)}^2 - \frac{\mu}{\epsilon_1} \|\left(w-P_Nw\right)(t)\|_{L^2(0,L)}^2 \\
		&\geq (2\mu - \epsilon_1\mu) \|w(t)\|_{L^2(0,L)}^2 - \frac{\mu}{\epsilon_1 \lambda_1 (2N + 1)^2} \|w_x(t)\|_{L^2(0,L)}.
	\end{split}
\end{equation}
Consequently, from \eqref{l2-inner-2} and \eqref{ctrl_term}, it follows that
\begin{equation}
\begin{aligned}
	\frac{d}{dt} \|w(t)\|_{L^2(0,L)}^2 &+ 2\nu \|w_x(t)\|_{L^2(0,L)}^2 - 2\gamma \|w(t)\|_{L^2(0,L)}^2
	\\ & + (2\mu - \epsilon_1\mu) \|w(t)\|_{L^2(0,L)}^2 - \frac{\mu}{\epsilon_1 \lambda_1 (2N + 1)^2} \|w_x(t)\|_{L^2(0,L)} \leq 0
    \end{aligned}
\end{equation}
or equivalently,
\begin{equation}\label{l2-inner-3}
	\frac{d}{dt} \|w(t)\|_{L^2(0,L)}^2 + \left(2\nu - \frac{\mu}{\epsilon_1 \lambda_1 (2N + 1)^2}\right) \|w_x(t)\|_{L^2(0,L)}^2 + \left(2\mu - \epsilon_1\mu - 2\gamma\right) \|w(t)\|_{L^2(0,L)}^2 \leq 0.
\end{equation}
Suppose $\epsilon_1$ is chosen with
\begin{equation}
	2\nu - \frac{\mu}{\epsilon_1 \lambda_1 (2N + 1)^2} > 0.
\end{equation}
Then, employing Poincar\'e's inequality, we get
\begin{equation} \label{l2-dec1}
	\frac{d}{dt} \|w(t)\|_{L^2(0,L)}^2 + 2\left[\nu\lambda_1 - \gamma + \mu \left(1 - \frac{\epsilon_1}{2} - \frac{1}{2 \epsilon_1 (2N + 1)^2} \right)\right] \|w(t)\|_{L^2(0,L)}^2 \leq 0.
\end{equation}
We deduce that \eqref{l2-dec1} yields an exponential decay provided
\begin{eqnarray}
	\label{cond1}
	&\text{(i)}& 2\nu - \frac{\mu}{\epsilon_1 \lambda_1 (2N + 1)^2} > 0, \\
	\label{cond2}
	&\text{(ii)}& \nu\lambda_1 - \gamma + \mu \left(1 - \frac{\epsilon_1}{2} - \frac{1}{2 \epsilon_1 (2N + 1)^2} \right) > 0.
\end{eqnarray}

Next, let us take $L^2$ inner product of the main equation of \eqref{pde_tarlin} by $-2w_{xx}$
\begin{equation} \label{h1-inner}
	-2 \int_0^L w_t \overline w_{xx} dx + 2(\nu + i\alpha) \int_0^L |w_{xx}|^2 dx + 2 \gamma \int_0^L w \overline w_{xx} dx - 2\mu \int_0^L \overline w_{xx} P_N w dx = 0.
\end{equation}
Integrating the first and the third terms by parts, and then taking the real parts of each term in \eqref{h1-inner}, we get
\begin{equation} \label{h1-inner-2}
	\frac{d}{dt} \|w_x(t)\|_{L^2(0,L)}^2 + 2\nu \|w_{xx}(t)\|_{L^2(0,L)}^2 dx - 2 \gamma \|w_x(t)\|_{L^2(0,L)} - 2\mu \int_0^L \overline w_{xx} P_N w dx = 0.
\end{equation}
We apply Cauchy-Schwarz inequality, Cauchy's inequality with $\epsilon_2 > 0$ to get
\begin{equation} \label{h1-est}
	\begin{split}
		-2\mu \int_0^L \overline w_{xx} P_Nw(x,t) dx
		=& 2\mu \int_0^L \overline w_{xx} (w-P_Nw) dx - 2\mu \int_0^L \overline w_{xx} w dx \\
		\geq& -\mu\epsilon_2 \|w_{xx}(t)\|_{L^2(0,L)}^2 - \frac{\mu}{\epsilon_2} \|(w-P_Nw)(t)\|_{L^2(0,L)}^2+ 2\mu  \|w_x(t)\|_{L^2(0,L)}^2.
	\end{split}
\end{equation}
Now we apply Poincar\'e-type inequality for the second term of \eqref{h1-est} to obtain the following upper bound
\begin{equation}
	\label{ctrl_term-h1}
	-2\mu \int_0^L \overline w_{xx} P_Nw dx \geq -\mu\epsilon_2 \|w_{xx}(t)\|_{L^2(0,L)}^2 
	+ \left(2\mu - \frac{\mu}{\epsilon_2\lambda_1 (2N+1)^2}\right)\|w_x(t)\|_{L^2(0,L)}^2.
\end{equation}
Then, we combine the \eqref{ctrl_term-h1} with \eqref{h1-inner-2} to write 
\begin{equation} \label{h1-inner-3}
	\frac{d}{dt} \|w_x(t)\|_{L^2(0,L)}^2 + 2\left(\nu - \frac{\mu\epsilon_2}{2}\right) \|w_{xx}(t)\|_{L^2(0,L)}^2 dx + 2\left(\mu - \gamma - \frac{\mu}{2\epsilon_2\lambda_1 (2N+1)^2}\right) \|w_x(t)\|_{L^2(0,L)}^2 \leq 0.
\end{equation}
Suppose $\epsilon_2$ is such that
\begin{equation*}
	\nu - \frac{\mu\epsilon_2}{2} > 0.
\end{equation*}
Applying Poincar\'e inequality to the second term in \eqref{h1-inner-3} yields the following estimate
\begin{equation} \label{h1-inner-4}
	\frac{d}{dt} \|w_x(t)\|_{L^2(0,L)}^2 + 2\left[\nu \lambda_1 - \gamma  + \mu\left(1 - \frac{\epsilon_2 \lambda_1}{2} - \frac{1}{2\epsilon_2\lambda_1 (2N+1)^2}\right)\right] \|w_x(t)\|_{L^2(0,L)}^2 \leq 0.
\end{equation}
Therefore, \eqref{h1-inner-4} will provide decay if
\begin{eqnarray}
	\label{cond3}
	&\text{(i)}& \nu - \frac{\mu\epsilon_2}{2} > 0, \\
	\label{cond4}
	&\text{(ii)}& \nu \lambda_1 - \gamma  + \mu\left(1 - \frac{\epsilon_2 \lambda_1}{2} - \frac{1}{2\epsilon_2\lambda_1 (2N+1)^2}\right) > 0.
\end{eqnarray}

\paragraph{\textit{Rapid stabilization.}} In order to maximize exponential decay rate, we want to choose $\epsilon_1$ and $\epsilon_2$ so that the sums
\begin{equation*}
    1 - \frac{\epsilon_1}{2} - \frac{1}{2 \epsilon_1 (2N + 1)^2}, \quad 1 - \frac{\epsilon_2 \lambda_1}{2} - \frac{1}{2\epsilon_2\lambda_1 (2N+1)^2}
\end{equation*}
inside the inner parenthesis in \eqref{l2-dec1} and \eqref{h1-inner-4}, respectively, attain their maximum. This is achieved for $\epsilon_1 = \frac{1}{2N + 1}$ and $\epsilon_2 = \frac{1}{\lambda_1 (2N+1)}$. Then, the conditions \eqref{cond1}, \eqref{cond3} and \eqref{cond2}, \eqref{cond4} coincide, respectively, and we obtain the following conditions on $N$:
\begin{equation}
    \begin{split}
        N &> \frac{\mu}{4 \nu \lambda_1} - \frac{1}{2}, \\
        N &> \frac{\mu}{2(\mu + \nu \lambda_1 - \gamma)} - \frac{1}{2}.
    \end{split}
\end{equation}
Note also that for these choices of $\epsilon_1$ and $\epsilon_2$, the inequalities \eqref{l2-dec1} and \eqref{h1-inner-4} become very similar and consequently, we get
\begin{equation*}
    \frac{d}{dt} \|w(t)\|_{H^1(0,L)}^2 + 2\left[\nu \lambda_1 - \gamma  + \mu\left(1 - \frac{1}{2N + 1} \right)\right] \|w(t)\|_{H^1(0,L)}^2 \leq 0.
\end{equation*}
Employing Gronwall's inequality, we obtain the following result: Given $\mu > \gamma - \nu \lambda_1$, if
\begin{equation} \label{n-cond}
	N > \max \left\{\frac{\mu}{4 \nu \lambda_1} - \frac{1}{2},\frac{\mu}{2(\mu + \nu \lambda_1 - \gamma)} - \frac{1}{2} \right\},
\end{equation}
then the exponential decay estimate
\begin{equation} \label{l2-dec-est-w-lin}
	\|w(t)\|_{H^1(0,L)} \leq e^{-\eta_1 t} \|w_0\|_{H^1(0,L)}, \quad \eta_1 = \nu\lambda_1 - \gamma + \mu \left(1 - \frac{1}{2N+1}\right), \quad \text{a.e. }t > 0,
\end{equation}
holds. Here the decay rate $\eta_1$ can be made as large as desired by choosing $\mu$ sufficiently large. Observe that the condition $\mu > \gamma - \nu \lambda_1$ is hidden in \eqref{l2-dec-est-w-lin} as $\eta_1 > 0$ must hold for any $N$.

\paragraph{\textit{Minimal number of Fourier modes.}} Suppose that instability level of the problem is $M$, i.e., $\lambda_M \leq \frac{\gamma}{\nu} < \lambda_{M+1}$. We want to minimize the number of Fourier modes $N$, that still guarantees exponential decay of solutions to zero at $H^1$-level. First, from conditions \eqref{cond1}-\eqref{cond2}, we can write
\begin{equation}\label{minimal_mu1}
	\frac{\gamma - \nu \lambda_1}{1 - \dfrac{\epsilon_1}{2} - \dfrac{1}{2 \epsilon_1 (2N+1)^2}} < \mu <  2\epsilon_1 \nu\lambda_1 (2N + 1)^2.
\end{equation}
To guarantee that such $\mu$ exists, we must have
\begin{equation}
	\frac{\gamma - \nu \lambda_1}{1 - \dfrac{\epsilon_1}{2} - \dfrac{1}{2 \epsilon_1 (2N+1)^2}} < 2\epsilon_1 \nu\lambda_1 (2N + 1)^2
\end{equation}
or equivalently
\begin{equation} \label{min-fourier-2}
	\frac{\gamma}{\nu} < \lambda_{N+1} (2\epsilon_1 - \epsilon_1^2),
\end{equation}
provided that $0 < \epsilon_1 < 2$. Similarly, from conditions \eqref{cond3}-\eqref{cond4}, we can write
\begin{equation}\label{minimal_mu2}
    \dfrac{\gamma - \nu\lambda_1}{1 - \dfrac{\epsilon_2 \lambda _1}{2} - \dfrac{1}{2 \epsilon_2 \lambda_1 (2N+1)^2}} < \mu < \frac{2\nu}{\epsilon_2},
\end{equation}
which leads to
\begin{equation} \label{min-fourier-h1}
	\frac{\gamma}{\nu} < \frac{2}{\epsilon_2} - \frac{1}{\epsilon_2^2 \lambda_{N+1}},
\end{equation}
provided that $\frac{1}{2\lambda_{N+1}} < \epsilon_2$. Notice that the ratios on the left sides of \eqref{min-fourier-2} and \eqref{min-fourier-h1} are same and determines the instability level which is greater than or equal to $\lambda_M$. Notice also that for $\epsilon_1 = 1$ and $\epsilon_2 = \lambda_{N+1}^{-1}$, right sides of \eqref{min-fourier-2} and \eqref{min-fourier-h1} attain their maximum, respectively, which both become $\lambda_{N + 1}$. Consequently, for each case, our criterion becomes
\begin{equation} \label{min_value_N}
    \lambda_M < \frac{\gamma}{\nu} < \lambda_{N + 1}.
\end{equation}
From this inequality, we infer that the minimum value for $N$ that the inequality \eqref{min_value_N} remains true is $M$. In other words, if the instability level of the model is $M$, then it suffices to use $M$ Fourier modes to guarantee that the norm estimates \eqref{l2-dec1} and \eqref{h1-inner-4} hold true. Note that choosing $\epsilon_1 = 1$ and $\epsilon_2 = \lambda_{N+1}^{-1}$, the inequalities \eqref{l2-dec1} and \eqref{h1-inner-4} become similar and Gronwall's inequality yields the following exponential decay estimate
\begin{equation} \label{h1-dec-est-w-lin-2}
	\|w(t)\|_{H^1(0,L)} \leq e^{-\eta_2 t} \|w_0\|_{H^1(0,L)}, \quad \eta_2 = \nu\lambda_1 - \gamma + \frac{\mu}{2} \left(1 - \frac{1}{(2N+1)^2}\right), \quad \text{a.e. }t > 0.
\end{equation}
Note also substituting the choices $\epsilon_1$, $\epsilon_2$ into \eqref{minimal_mu1} and \eqref{minimal_mu2} lead to the same following condition on $\mu$:
\begin{equation}\label{mu-cond}
	2(\gamma - \nu \lambda_1) \left(1 - \frac{1}{(2N+1)}\right)^{-1}< \mu < 2\nu \lambda_{N+1}.
\end{equation}
Returning to our estimate \eqref{l2-dec1} and taking $\epsilon_1 = 1$, we obtain the following result: Let $\lambda_M \leq \frac{\gamma}{\nu} < \lambda_{M+1}$. If $\mu$ satisfies \eqref{mu-cond}, then it suffices to use $N = M$ Fourier modes to guarantee the exponential decay of solutions of \eqref{pde_tarlin} in $L^2$ level to the zero equilibrium. Moreover, the following exponential decay estimate holds true,
\begin{equation} \label{l2-dec-est-w-lin-2}
	\|w(t)\|_{H^1(0,L)} \leq e^{-\eta_2 t} \|w_0\|_{H^1(0,L)}, \quad \eta_2 = \nu\lambda_1 - \gamma + \frac{\mu}{2} \left(1 - \frac{1}{(2N+1)^2}\right), \quad \text{a.e. }t > 0. 
\end{equation}

As a last step, we show that the decay estimates \eqref{l2-dec-est-w-lin} and \eqref{l2-dec-est-w-lin-2} we derived for the linear target model \eqref{pde_tarlin} are also true for the linear plant \eqref{pde_lin}. Let $\eta^*$ denote either of the decay rates $\eta_1$ or $\eta_2$, given by \eqref{l2-dec-est-w-lin} or \eqref{l2-dec-est-w-lin-2}, respectively. We use the boundedness of the transformation \eqref{backstepping-trans}, and get 
\begin{equation} \label{backtoplant1}
	\begin{split}
		\no{u(t)}_{H^1(0,L)} &\leq \left(1 + \|k\|_{H^1(\mathcal T)}\right) \|w(\cdot,t)\|_{H^1(0,L)} \\
		&\le \left(1 + \|k\|_{H^1(\mathcal T)}\right) e^{-\eta^* t} \|w_0\|_{H^1(0,L)}.
	\end{split}
\end{equation}
Next, we use the invertibility of the  transformation \eqref{backstepping-trans} with a bounded inverse to write
\begin{equation} \label{backtoplant2}
	\|w_0\|_{H^1(0,L)} \leq  \|T_N^{-1}\|_{H^1(0,L) \to H^1(0,L)} \no{u_0}_{H^1(0,L)}.
\end{equation}
Combining \eqref{backtoplant1} and \eqref{backtoplant2}, we conclude that
\begin{equation} \label{backtoplant3}
	\no{u(t)}_{H^1(0,L)} \leq c_k e^{-\eta^* t} \no{u_0}_{H^1(0,L)}, \quad \text{a.e } t \geq 0,
\end{equation}
where $c_k = (1 + \|k\|_{H^1(\mathcal T)}) \|T_N^{-1}\|_{H^m(0,L) \to H^m(0,L)}$ is a nonnegative constant independent of the initial datum. This proves Theorem \ref{thm-stab-linear}.

\subsubsection{Nonlinear plant (Proof of Theorem \ref{thm-stab-nonlinear})}
If we perform the calculations to transform nonlinear plant \eqref{pde_nonlin} to an associated target model by using the transformation $u = T_N w$ with the kernel that satisfies \eqref{kernel_pde}, then there will be an extra term in the main equation of the target model, which occurs due to the nonlinear term, $(\kappa + i\beta) |u|^p u$ in the original plant. Denote the extra term in the target model by $f(w)$. Using the same procedure in \eqref{kert}-\eqref{ker3} will lead to additional terms in \eqref{ker_add}, i.e., $-(\kappa + i \beta)|u|^p u$ at the left side and $f(w) + \int_0^x k(x,y) (P_N f(w)))(y,t) dy$ at the right side. In order to guarantee that \eqref{ker_add} holds, we must have
\begin{equation*}
	(\kappa + i \beta)|u|^p u = f(w) +  \int_0^x k(x,y) (P_N f(w)))(y,t) dy,
\end{equation*}
which implies
\begin{equation*}
	f(w) = (\kappa + i \beta) T_N^{-1}[ |T_N w|^p (T_N w)].
\end{equation*}
So the target model associated with the nonlinear plant \eqref{pde_nonlin} is of the form
\begin{eqnarray} \label{pde_tarnonlin}
	\begin{cases}
		w_t - (\nu + i\alpha) w_{xx} - \gamma w + \mu P_N w + f(w) = 0, &(x,t) \in(0,L) \times (0,T), \\
		w(0,t) = w_x(L,t) = 0, & t \in (0,T),\\
		w(x,0) = w_0(x), &x \in (0,L),
	\end{cases}
\end{eqnarray}
where $f(w) = (\kappa + i \beta)(I - \Upsilon_N) \left[|T_Nw|^p T_N w\right].$ Below, we derive exponential decay estimates for solutions of \eqref{pde_tarnonlin} in $H^1$ level.

We take the $L^2$ inner product of the main equation in \eqref{pde_tarnonlin} by $2w$,
\begin{equation} \label{l2-inner-nonlin}
\begin{aligned}
	2 \int_0^L w_t \overline w dx & - 2(\nu + i\alpha) \int_0^L w_{xx} \overline w dx - 2 \gamma \int_0^L |w|^2 dx + 2\mu \int_0^L \overline w P_N w dx \\ 
	 &= -2(\kappa + i \beta) \int_0^L T_N^{-1} \left[|T_N w|^p( T_N w) \right]\overline w dx.
    \end{aligned}
\end{equation}
Using boundedness of $T_N^{-1}$ on $L^2(0,L)$ and Cauchy-Schwarz inequality, we can bound the right side as
\begin{equation} \label{rhs1}
	-2(\kappa + i \beta) \int_0^L T_N^{-1} \left[|T_N w|^p(T_N w)\right]\overline w dx \leq c_0 \|(T_N w)^{p+1}(t)\|_{L^2(0,L)} \|w(t)\|_{L^2(0,L)},
\end{equation}
where $c_0 = 2 |\kappa + i\beta| \|T_N^{-1}\|_{L^2(0,L) \to L^2(0,L)}$. Applying Gagliardo-Nirenberg inequality and boundedness of the  transformation \eqref{backstepping-trans}, we get
\begin{equation*}
	\begin{split}
		c_0\|(T_N w)^{p+1}(t)\|_{L^2(0,L)} &= c_0\|(T_N w)(t)\|_{L^{2p+2}(0,L)}^{p+1} \\
		&\leq c_0 c_{g,1} \left(\|\partial_x (T_N w)(t)\|_{L^2(0,L)}^{\frac{p}{2p+2}} \|(T_Nw)(t)\|_{L^2(0,L)}^{\frac{p+2}{2p+2}} + \|(T_N w)(t)\|_{L^2(0,L)}\right)^{p+1} \\
		&\leq c_1 \|\partial_x w(t)\|_{L^2(0,L)}^{\frac{p}{2}} \|w(t)\|_{L^2(0,L)}^{\frac{p+2}{2}} + c_2\|w(t)\|_{L^2(0,L)}^{p+1},
	\end{split}
\end{equation*}
where $c_{g,1}$ is the constant of the Gagliardo-Nirenberg inequality,
\begin{equation} \label{norm-constants}
	c_1 = c_0 c_{g,1}\|T_N\|_{H^1(0,L) \to H^1(0,L)}^{\frac{p}{2}} \|T_N\|_{L^2(0,L) \to L^2(0,L)}^{\frac{p+4}{2}}, \quad c_2 = c_0 c_{g,1} \|T_N\|_{L^2(0,L) \to L^2(0,L)}^{p+1}.
\end{equation}
Consequently, the right side of \eqref{rhs1} is bounded from above by
\begin{equation} \label{c2}
	c_1\|\partial_x w(t)\|_{L^2(0,L)}^{\frac{p}{2}} \|w(t)\|_{L^2(0,L)}^{\frac{p + 4}{2}} + c_2\|w(t)\|_{L^2(0,L)}^{p+2}.
\end{equation}
Now, we use Young's inequality with $\varepsilon_1 > 0$ with the pair $(\frac{4}{p},\frac{4}{4-p})$ for the first term and get
\begin{equation}\label{nonlin-1}
	\begin{split}
		&c_1\|\partial_x w(t)\|_{L^2(0,L)}^{\frac{p}{2}} \|w(t)\|_{L^2(0,L)}^{\frac{p + 4}{2}} + c_2\|w(t)\|_{L^2(0,L)}^{p+2}\\ \leq &\varepsilon \left(\|\partial_x w(t)\|_{L^2(0,L)}^{\frac{p}{2}}\right)^{\frac{4}{p}} + c_{1,\varepsilon}\left(\|w(t)\|_{L^2(0,L)}^{\frac{p + 4}{2}}\right)^{\frac{4}{4-p}} \\
		&+ c_2\|w(t)\|_{L^2(0,L)}^{p+2} = \varepsilon \|\partial_x w(t)\|_{L^2(0,L)}^2 + c_{1,\varepsilon} \|w(t)\|_{L^2(0,L)}^{\frac{2p+8}{4-p}} + c_2\|w(t)\|_{L^2(0,L)}^{p+2}
	\end{split}
\end{equation}
with $p < 4$, where
\begin{equation} \label{c1star}
	c_{1,\varepsilon} = \frac{4 - p}{4} \left(\frac{c_1 p}{4 \varepsilon}\right)^{\frac{p}{4 - p}}.
\end{equation}
For the linear terms in \eqref{l2-inner-nonlin}, we proceed with steps similar to those used for \eqref{l2-inner}-\eqref{l2-inner-3}, then we combine the result with \eqref{nonlin-1} to get
\begin{multline}\label{nonlin-1.5}
	\frac{d}{dt} \|w(t)\|_{L^2(0,L)}^2 + \left(2\nu - \frac{\mu}{\epsilon_1 \lambda_1 (2N + 1)^2}\right) \|w_x(t)\|_{L^2(0,L)}^2 + \left(2\mu - \epsilon_1\mu - 2\gamma\right) \|w(t)\|_{L^2(0,L)}^2 \\ \leq \varepsilon \|\partial_x w(t)\|_{L^2(0,L)}^2 + c_{1,\varepsilon} \|w(t)\|_{L^2(0,L)}^{\frac{2p+8}{4-p}} + c_2\|w(t)\|_{L^2(0,L)}^{p+2}.
\end{multline}
Let $\varepsilon$ be such that
\begin{equation} \label{eps-cond1}
	2\nu - \varepsilon - \frac{\mu}{\epsilon_1 \lambda_1 (2N + 1)^2} > 0.
\end{equation}
Recall that depending on our purpose, i.e., achieving rapid stabilization or the minimal Fourier mode count, we choose $\epsilon_1 = \frac{1}{2N+1}$ and $\epsilon_1 = 1$, respectively. Applying Poincar\'e inequality for the second term at the left side of \eqref{nonlin-1.5}, we end up with the following inequality,
\begin{equation} \label{nonlin-2}
	\frac{d}{dt} \|w(t)\|_{L^2(0,L)}^2 + 2\left(\eta^* - \frac{\varepsilon \lambda_1}{2}\right) \|w(t)\|_{L^2(0,L)}^2 \leq c_{1,\varepsilon} \|w(t)\|_{L^2(0,L)}^{\frac{2p+8}{4-p}} + c_2\|w(t)\|_{L^2(0,L)}^{p+2}.
\end{equation}
Here $\eta^*$ is either $\eta_1$ or $\eta_2$, representing the exponential decay rate for the rapid stabilization problem or the exponential decay rate for the problem of finding minimal number of Fourier modes for stabilization, respectively. Let us set $y(t) =\|w(t)\|_{L^2(0,L)}^2$. Then \eqref{nonlin-2} implies local exponential stability of zero equilibrium of \eqref{pde_tarnonlin} at $L^2$ level by Lemma \ref{lem-nonlin-1} below.

\begin{lemma} \label{lem-nonlin-1}
	Let $A, B, C > 0$ and $y(t) > 0$ satisfies the differential inequality
	\begin{equation} \label{ineq-nonlin-l2}
		y^\prime + A y \leq By^{\frac{2p + 8}{p - 4}} + Cy^{p+2}.
	\end{equation}
	If $y(0) < \min \left\{\left(\frac{A}{B + C}\right)^{\frac{4-p}{3p + 4}},\left(\frac{A}{B+C}\right)^{\frac{1}{p+1}}\right\}$, then $y(t) \lesssim y(0) e^{-At}$, for a.e. $t > 0$.
\end{lemma}

\begin{proof}
	As a first case, suppose  $y(0) > 1$. Let $t_1 > 0$ be such that
	\begin{equation*}
		y(t) > 1, \quad \text{a.e. } t \in (0,t_1).
	\end{equation*}
	Since $p + 2 < \frac{2p + 8}{4- p}$ for $0 < p < 4$, it is true that
	\begin{equation*}
		y^{p+2}(t) < y^{\frac{2p + 8}{4 - p}}(t), \quad \text{a.e. } t \in (0,t_1).
	\end{equation*}
	Then from \eqref{ineq-nonlin-l2}, we can write
	\begin{equation*}
		y^\prime + A y \leq (B + C)y^{\frac{2p + 8}{4 - p}}, \quad \text{a.e. } t \in (0,t_1).
	\end{equation*}
	This is a Bernoulli type differential inequality which yields the following estimate
	\begin{equation} \label{ineq-nonlin-l2-2}
		y^{\frac{3p + 4}{4-p}}(t) \leq \left[\frac{B + C}{A} + \exp \left(A \frac{3p+4}{4-p} t\right) \times \left(y^{-\frac{3p+4}{4-p}}(0)- \frac{B+C}{A}\right)\right]^{-1}, \quad \text{a.e. } t \in (0,t_1).
	\end{equation}
	Under the assumption
	\begin{equation*}
		y^{-\frac{3p+4}{4-p}}(0)- \frac{B+C}{A}> 0 \Rightarrow y(0) \leq \left(\frac{A}{B + C}\right)^{\frac{4-p}{3p + 4}},
	\end{equation*}
	it follows from \eqref{ineq-nonlin-l2-2} that
	\begin{equation*}
		y(t) \lesssim y(0) e^{-At}, \quad \text{a.e. } t \in (0,t_1).
	\end{equation*}
	This decay estimate indicates that $t_1$ is finite. Moreover, there exists a $t_2$ such that
	\begin{equation} \label{ineq-nonlin-l2-3}
		y(t) \leq 1, \quad \text{a.e. } t \in (t_1,t_2).
	\end{equation}
	Notice that, in this case we have
	\begin{equation*}
		y^{\frac{2p + 8}{4-p}}(t) \leq y^{p+2}(t), \quad \text{a.e. } t \in (t_1,t_2).
	\end{equation*}
	Following from \eqref{ineq-nonlin-l2}, we can write
	\begin{equation*}
		y^\prime + A y \leq (B + C)y^{p+2}, \quad \text{a.e. } t \in (t_1,t_2).
	\end{equation*}
	This is again a Bernoulli type differential inequality which yields
	\begin{equation} \label{ineq-nonlin-l2-4}
		y^{p+1}(t) \leq \left[\frac{B + C}{A} + e^{-A (p+1) t} \left(y^{-(p+1)}(0)- \frac{B+C}{A}\right)\right]^{-1}, \quad \text{a.e. } t \in (t_1,t_2).
	\end{equation}
	Assuming that
	\begin{equation*}
		y(0) \leq \left(\frac{A}{B+C}\right)^{\frac{1}{p+1}},
	\end{equation*}
	\eqref{ineq-nonlin-l2-4} implies
	\begin{equation} \label{ineq-nonlin-l2-5}
		y(t) \lesssim y(0) e^{-At}, \quad \text{a.e. } t \in (t_1,t_2).
	\end{equation}
	Thus, $y(t)$ is a decreasing function on the time interval $(t_1,t_2)$ and $t_2$ can be chosen arbitrarily large. Consequently \eqref{ineq-nonlin-l2-5} is true for a.e. $t \geq t_1$. Observe that calculations through \eqref{ineq-nonlin-l2-3}-\eqref{ineq-nonlin-l2-5} covers the second case $y(0) \leq 1$ by taking $t_1 = 0$.
\end{proof}

Next, we show that $w_x(x,t)$ uniformly decays exponentially to zero in-time. To this end, let us take $L^2$ inner product on the main equation \eqref{pde_tarnonlin} by $-2w_{xx}$,
\begin{multline} \label{h1-inner-nonlin}
	-2 \int_0^L w_t \overline w_{xx} dx  + 2(\nu + i\alpha) \int_0^L |w_{xx}|^2 dx + 2 \gamma \int_0^L w \overline{w}_{xx} dx + 2\mu \int_0^L \overline w_{xx} P_N w dx \\ 
	= -2(\kappa + i \beta) \int_0^L T_N^{-1} \left[|T_N w|^p( T_N w) \right]\overline w_{xx} dx.
\end{multline}
$T_N^{-1}$ is a bounded operator on $L^2(0,L)$, so we can write
\begin{equation} \label{h1-rhs1}
	-2(\kappa + i \beta) \int_0^L T_N^{-1} \left[|T_N w|^p(T_N w)\right]\overline w_{xx}(x,t) dx \leq 2|\kappa + i\beta| \|(T_N w)^{p+1}(t)\|_{L^2(0,L)} \|w_{xx}(t)\|_{L^2(0,L)},
\end{equation}
where $c_0 = 2 |\kappa + i\beta| \|T_N^{-1}\|_{L^2(0,L) \to L^2(0,L)}$. Using Gagliardo-Nirenberg inequality, we obtain the following bound
\begin{equation*}
	\begin{split}
		c_0\|(T_N w)^{p+1}(t)\|_{L^2(0,L)} &= c_0\|(T_N w)(t)\|_{L^{2p+2}(0,L)}^{p+1} \\
		&\lesssim c_0 c_{g,2} \left(\|\partial_{xx} (T_N w)(t)\|_{L^2(0,L)}^{\frac{p}{4p+4}} \|(T_Nw)(t)\|_{L^2(0,L)}^{\frac{3p+4}{4p+4}} + \|(T_N w)(t)\|_{L^2(0,L)}\right)^{p+1} \\
		&\leq c_3 \| w_{xx}(t)\|_{L^2(0,L)}^{\frac{p}{4}} \|w(t)\|_{L^2(0,L)}^{\frac{3p+4}{4}} + c_2 \|w(t)\|_{L^2(0,L)}^{p+1},
	\end{split}
\end{equation*}
where $c_{g,2}$ is a constant comes from Gagliardo-Nirenberg inequality and 
\begin{equation*}
	c_3 = c_0c_{g,2}\|T_N\|_{H^2(0,L) \to H^2(0,L)}^{\frac{p}{4}} \|T_N\|_{L^2(0,L) \to L^2(0,L)}^{\frac{3p+4}{4}}, \quad c_4 = c_0 c_{g,2} \|T_N\|_{L^2(0,L) \to L^2(0,L)}^{p+1}.
\end{equation*}
So the right side of \eqref{h1-rhs1} is bounded from above as
\begin{equation*}
	c_3 \|w_{xx}(t)\|_{L^2(0,L)}^{\frac{p+4}{4}} \|w(t)\|_{L^2(0,L)}^{\frac{3p+4}{4}} + c_4\|w(t)\|_{L^2(0,L)}^{p+1} \|w_{xx}(t)\|_{L^2(0,L)}.
\end{equation*}
We use Young's inequality with $\varepsilon > 0$ for the first term with the pair $\left(\frac{8}{p+4},\frac{8}{4-p}\right)$ and Cauchy's inequality with the same $\varepsilon > 0$, for the second term to get
\begin{multline*}
	c_3 \|w_{xx}(t)\|_{L^2(0,L)}^{\frac{p+4}{4}} \|w(t)\|_{L^2(0,L)}^{\frac{3p+4}{4}} + c_4\|w(t)\|_{L^2(0,L)}^{p+1} \|\partial_{xx} w(t)\|_{L^2(0,L)}\\  \leq \frac{\varepsilon}{2} \|\partial_{xx} w(t)\|_{L^2(0,L)}^{2} + c_{3,\varepsilon} \|w(t)\|_{L^2(0,L)}^{\frac{6p+8}{4-p}} + \frac{\varepsilon}{2}\|\partial_{xx} w(t)\|_{L^2(0,L)}^{2}  + c_{4,\varepsilon} \|w(t)\|_{L^2(0,L)}^{2p+2} \\
	\leq \varepsilon \|\partial_{xx} w(t)\|_{L^2(0,L)}^{2}  +c_{3,\varepsilon,\lambda_1} \|w_x(t)\|_{L^2(0,L)}^{\frac{6p+8}{4-p}} + c_{4,\varepsilon,\lambda_1} \|w_x(t)\|_{L^2(0,L)}^{2p+2} 
\end{multline*}
provided that $p < 4$ where
\begin{equation*}
	c_{3,\varepsilon} = \frac{4 - p}{8} \left(\frac{c_3 (p + 4)}{4\varepsilon}\right)^{\frac{p+4}{4-p}}, \quad 	c_{4,\varepsilon} = \frac{c_4^2}{2\varepsilon},
\end{equation*}
and
\begin{equation} \label{c3c4}
	c_{3,\varepsilon,\lambda_1} = c_{3,\varepsilon} (\lambda_1^{-1})^{\frac{6p+8}{4-p}}, \quad c_{4,\varepsilon,\lambda_1} = c_{4,\varepsilon} (\lambda_1^{-1})^{2p+2}.
\end{equation}
Regarding the linear terms in \eqref{h1-inner-nonlin}, we apply similar arguments to those given by \eqref{h1-inner}-\eqref{h1-inner-4} and obtain the following inequality
\begin{multline*}
	\frac{d}{dt} \|w_x(t)\|_{L^2(0,L)}^2 + 2\left(\nu - \frac{\mu\epsilon_2}{2}\right) \|w_{xx}(t)\|_{L^2(0,L)}^2 dx + 2\left(\mu - \gamma - \frac{\mu}{2\epsilon_2\lambda_1 (2N+1)^2}\right) \|w_x(t)\|_{L^2(0,L)} \\ \leq \varepsilon \|\partial_{xx} w(t)\|_{L^2(0,L)}^{2} + c_{3,\varepsilon,\lambda_1}  \|w_x(t)\|_{L^2(0,L)}^{\frac{6p+8}{4-p}} + c_{4,\varepsilon,\lambda_1} \|w_x(t)\|_{L^2(0,L)}^{2p+2},
\end{multline*}
Choose $\varepsilon$ such that
\begin{equation} \label{eps-cond2}
	2\nu - \varepsilon - \mu\epsilon_2 > 0,
\end{equation}
where $\epsilon_2 = \frac{1}{\lambda_1 (2N+1)}$ or $\epsilon_2 = \lambda_{N+1}^{-1}$, depending either we want to achieve rapid stabilization or we want to minimize number of Fourier modes, respectively. Then, we obtain the following inequality
\begin{equation} \label{nonlin-3}
	\frac{d}{dt} \|w_x(t)\|_{L^2(0,L)}^2 + 2\left(\eta^* - \frac{\varepsilon \lambda_1}{2}\right) \|w_x(t)\|_{L^2(0,L)}^2 \leq c_{3,\varepsilon,\lambda_1} \|w_x(t)\|_{L^2(0,L)}^{\frac{6p+8}{4-p}} + c_{4,\varepsilon,\lambda_1} \|w_x(t)\|_{L^2(0,L)}^{2p+2},
\end{equation}
where $\eta^*$ is either $\eta_1$ or $\eta_2$ given by \eqref{l2-dec-est-w-lin} or \eqref{l2-dec-est-w-lin-2}, respectively. Denote $z(t) = \|w_x(t)\|_{L^2(0,L)}^2$. If $z(0)$ is sufficiently small, then $z$ exponentially decays to zero by the following lemma.

\begin{lemma} \label{lem-nonlin-2}
	Let $A, B, C > 0$ and $z(t) > 0$ satisfy the differential inequality
	\begin{equation} \label{ineq-nonlin-h1}
		z^\prime + A z \leq Bz^{\frac{6p+8}{4-p}} + Cz^{2p+2}.
	\end{equation}
	If $z(0) < \min \left\{\left(\frac{A}{B + C}\right)^{\frac{4 - p}{7p + 4}},\left(\frac{A}{B+C}\right)^{\frac{1}{2p+1}}\right\}$, then $z(t) \lesssim z(0) e^{-At}$, for a.e. $t > 0$.
\end{lemma}

\begin{proof} 
	The proof is identical to that of Lemma \ref{lem-nonlin-1}, observing that $\displaystyle \frac{6p + 8}{4 - p} > 2p + 2$ for $0 < p < 4$.
\end{proof}

Combining Lemma \ref{lem-nonlin-1} and Lemma \ref{lem-nonlin-2} with the assumptions on $w_0$ and $w^\prime_0$, yields the local exponential stabilization of zero equilibrium of the target model \eqref{pde_tarnonlin} in $H^1$ level. Moreover, we have the following exponential decay estimate
\begin{equation*}
    \|w(t)\|_{H^1(0,L)}^2 \lesssim c_k e^{-\left(\eta^* - \frac{\varepsilon \lambda_1}{2} \right)} \|w_0\|_{H^1(0,L)}, \quad \text{a.e. }t \geq 0.
\end{equation*}

Finally, to go back to the original plant, we apply the same arguments \eqref{backtoplant1}-\eqref{backtoplant3} and derive the decay estimates
\begin{equation*}
	\no{u(t)}_{H^1(0,L)} \lesssim c_k e^{-\left(\eta^* - \frac{\varepsilon \lambda_1}{2}\right) t} \no{u_0}_{H^1(0,L)}, \quad \text{a.e. }t \geq 0,
\end{equation*}
provided that $\no{u_0}_{H^1(0,L)}$ is sufficiently small. This proves Theorem \ref{thm-stab-nonlinear}. Here $c_k = (1 + \|k\|_{H^m(\mathcal T)}) \|T_N^{-1}\|_{H^m(0,L) \to H^m(0,L)}$ is a nonnegative constant independent of the initial datum.

\begin{remark}
	The conditions on $\mu$ and $N$ for the linear model stated in Theorem \ref{thm-stab-linear} remain same for the nonlinear model. As long as these conditions are fulfilled, one can always find an $\varepsilon > 0$ such that \eqref{eps-cond1}, \eqref{eps-cond2} as well as $\eta^* - \frac{\varepsilon \lambda_1}{2} > 0$ hold true.
\end{remark}

\subsection{Extending global uniqueness} \label{uniqueness-closed-loop}
In this part, we show that the uniqueness of the closed-loop system over the ball $\overline{B(0,\rho_T)} \subset X_T$, where $\rho$ is given by \eqref{rst-defnew}, can be extended to the entire space $X_T$. We show this by studying the associated linear and nonlinear target models. Uniqueness of solutions to the linearized and nonlinear plants follows by bijectivity of our integral transformation.

\subsubsection{Linearized plant.} Let $y_1$ and $y_2$ be two solutions to the linearized plant \eqref{pde_lin} corresponding to the same initial condition $y_0$. Since $T_N$ is bijection, we set the unique elements $z_1$ and $z_2$ associated to $y_1$ and $y_2$, i.e., solutions to associated target models, such that $z_1 = T_N^{-1} y_1$ and $z_2 = T_N^{-1} y_2$. Similarly, we set $z_0 = T_N^{-1} y_0$, the initial condition to the target model. Define $z = z_1 - z_2$. Then, using the linearity, $w$ solves the following model
\begin{equation*}
	\begin{cases}
		z_t - (\nu + i\alpha) z_{xx} - \gamma z + \mu P_N z = 0, &(x,t) \in(0,L) \times (0,T), \\
		z(0,t) = z_x(L,t) = 0, & t \in (0,T),\\
		z(x,0) = 0, &x \in (0,L).
	\end{cases}
\end{equation*}
Using energy arguments similar to those used in Section \ref{linstab}, one can derive the following estimate
\begin{equation*}
    \frac{d}{dt} \|z(t)\|_{L^2(0,L)}^2 \leq C \|z(t)\|_{L^2(0,L)}^2, \quad \text{a.e. }t \geq 0.
\end{equation*}
Employing the Grönwall's inequality and noting that $z_0 = 0$, we obtain $z \equiv 0$. By the global existence result, we know that $z(\cdot,t) \in H_0^1(0,L)$ for each $t > 0$. Since the embedding $H_0^1(0,L) \to C([0,L])$ is continuous, we infer that $z = 0$ for $x \in [0,L]$. Hence,
\begin{equation*}
    y_2 - y_1 = T_N z_2 - T_N z_1 = T_N (z_2 - z_1) = T_N 0 = 0,
\end{equation*}
for $x \in [0,L]$.

\subsubsection{Nonlinear plant.} Using the same setting as in the linear case, $z = z_2 - z_1$ solves the following model
\begin{equation*}
	\begin{cases}
		z_t - (\nu + i\alpha) z_{xx} - \gamma z + \mu P_N z + f(z_1) - f(z_2)= 0, &(x,t) \in(0,L) \times (0,T), \\
		z(0,t) = z_x(L,t) = 0, & t \in (0,T),\\
		z(x,0) = 0, &x \in (0,L).
	\end{cases}
\end{equation*}
Using the arguments through \eqref{l2-inner}-\eqref{l2-dec1}, we obtain
\begin{equation} \label{nonlin-unique}
    \frac{d}{dt} \|z(t)\|_{L^2(0,L)}^2 + 2 \eta^* \|z(t)\|_L^2(0,L)^2 \leq \int_0^L \left| (f(z_1) - f(z_2)) \right| |z| dx,
\end{equation}
where $\eta > 0$ is either $\eta_1$ or $\eta_2$ derived in Section \ref{linstab}. Using Cauchy-Schwarz inequality and boundedness of $T_N^{-1}$ on $L^2(0,L)$, we can bound the right side of \eqref{nonlin-unique} as
\begin{align}
        \int_0^L \left| (f(z_1) - f(z_2)) \right| |z| dx &\leq C_{\kappa,\beta} 
        \|z(t)\|_{L^2(0,L)} \left\|T_N^{-1} \left[|T_N z_1|^p (T_N z_1) - |T_N z_2|^p (T_N z_2) \right](t)\right\|_{L^2(0,L)} \nonumber \\
        \label{nonlin-unique-2}
        &\leq C_{\kappa,\beta} \|z(t)\|_{L^2(0,L)} \left\|\left(|T_N z_1|^p (T_N z_1) - |T_N z_2|^p (T_N z_2)\right)(t)\right\|_{L^2(0,L)}
\end{align}
Using the inequality $\left||a|^p a - |b|^p b\right| \leq \left(|a|^{p} + |b|^{p}\right) |a - b|$, where $a,b \in \mathbb{C}$, we can write
\begin{equation} \label{nonlin-unique-3}
    \begin{split}
        \int_0^L& \left| (f(z_1) - f(z_2)) \right| |z| dx \leq \text{Expression in }\eqref{nonlin-unique-2} \\
        &\leq C_{\kappa,\beta} \|z(t)\|_{L^2(0,L)} \left(\|(T_N z_1)(t)\|_{L^\infty(0,L)}^{p} + \|(T_N z_2)(t)\|_{L^\infty(0,L)}^{p}\right) \|(T_N z_1 - T_N z_2)(t)\|_{L^2(0,L)}.
    \end{split}
\end{equation}
Using the estimate in \eqref{nonlin-unique-3} and, also linearity and boundedness of $T_N$ on $L^2(0,L)$, it follows from \eqref{nonlin-unique} that
\begin{equation} \label{nonlin-unique-4}
    \frac{d}{dt} \|z(t)\|_{L^2(0,L)}^2 + 2 \eta^* \|z(t)\|_{L^2(0,L)}^2 \leq C_{\kappa,\beta} \|z(t)\|_{L^2(0,L)}^2 \left(\|z_1(t)\|_{X_T}^{p} + \|z_2(t)\|_{X_T}^{p}\right) ,
\end{equation}
Employing the Grönwall's inequality and noting that $z_0 = 0$ implies $\|z(t)\|_{L^2(0,L)} = 0$ which yields $z \equiv 0$. Since $z(\cdot,t) \in H_0^1(0,L)$ and thanks to the continuous embedding $H_0^1(0,L) \to C([0,L])$, $z(\cdot,t)$ is continuous on $x \in [0,L]$. Hence $z = 0$ for $x \in [0,L]$. Now, due to linearity and bijectivity of our integral transformation, we conclude $y = 0$.

\section{Numerical algorithm and simulations}
\label{numerical}
In this section, we establish a fully discrete form of the closed-loop systems
\begin{equation}\label{numeric-cgle}
	\begin{aligned}
		&u_t - \left(\nu + i \alpha\right) u_{xx} - \gamma u + f(u) = 0, \quad (x, t) \in (0, L) \times (0, T),
		\\
		&u(0, t) = 0, \quad u_x(L, t) = b_u(t),
        \\
		&u(x,0) = u_0(x),
	\end{aligned} 
\end{equation}
where $f(u) = (\kappa + i\beta)|u|^p u$ and
\begin{equation}
    \label{numeric-ctrl}
	b_u(t) = \int_0^L k_x(L,y) P_N w(y,t) dy -\frac{\mu L}{2 (\nu + i \alpha)} P_N w(x,t) \big|_{x = L}, \quad w = T_N^{-1} u = (I - \Upsilon_N)u.
\end{equation}
We treat both the linear case, i.e. $\beta = \kappa = 0$, and the nonlinear case. Then, we construct a numerical algorithm for computing an approximate solution and present numerical simulations verifying our theoretical stabilization results stated in Theorem \ref{thm-stab-linear} and Theorem \ref{thm-stab-nonlinear}. Throughout this part, we will use the following notations and, discretize the spatial and temporal intervals as described below. Consider a uniform discretization of the interval $[0,L]$ with $N_x$ spatial grid points. Let $\delta_x = \frac{L}{N_x - 1}$ denote the uniform grid spacing so that each grid point $x_i$, $i \in \{1, \dotsc , N_{x}\}$, is given by $x_i = (i-1) \delta_x$. Let $\varphi = \varphi(x)$ be a function defined on $[0,L]$. For each $i \in \{1, \dotsc , N_{x}\}$, $\varphi_i$ denotes an approximation of $\varphi(x_i)$ and $\boldsymbol\varphi= [\varphi_1 \, \dotsc \, \varphi_{N_x}]^T \in \mathbb{C}^{N_x}$ denotes the grid function approximating $\varphi$ on $[0,L]$. Let $T_{\max}$ be the final time and $N_t$ denote the number of uniformly spaced time steps. The uniform time step size is then given by $\delta_t = \frac{T_{\max}}{N_t}$ and the time levels are defined by $t_n = n\delta_t$, $n \in \{0, \dotsc, N_t\}$. For each $i \in \{1, \dotsc , N_{x}\}$ and $n \in \{0, \dotsc, N_t\}$, $\varphi_i^n$ represents an approximation of the function $\varphi(x,t)$ at the grid point $(x_i,t_n)$, that is, $\varphi_i^n \approx \varphi(x_i,t_n)$. For a given linear operator $M$, we denote by $\mathbf{M}^h$ its discrete matrix representation obtained after spatial discretization. Throughout this section, $*:\mathbb C^{N_x}\times\mathbb C^{N_x}\to\mathbb C^{N_x}$ denotes the componentwise multiplication of vectors in $\mathbb C^{N_x}$.
\\

\noindent\textbf{Outline of numerical approach.} In Section \ref{numsec-1}, we obtain discrete matrix representations of the operators that appear in the feedback law \eqref{numeric-ctrl}. The integral terms involved in these operators are discretized by the composite trapezoidal rule, which provides a second-order accurate approximation. In Section \ref{numsec-2}, we use these matrix representations to construct the row operator $\mathbf{b}_{\text{row}}^h$ that approximates the boundary control input as
\begin{equation*}
    b_u(t^{n+1}) \approx \mathbf{b}_{u}^h \mathbf{u}^{n+1}, \quad n \in \{0, \dotsc, N_t-1\}.
\end{equation*}
Forming such a row-operator approximation is indeed possible since the control law is linear in the state $u$. Note that in this way, at the time level $n+1$, the equation $u_x(L,t^{n+1})=b_u(t^{n+1})$ involving the control input can be approximated by a linear combination of the nodal values of $\mathbf{u}^{n+1}$ at the grid points $x_1,\dotsc,x_{N_x}$. This representation allows us to implement the feedback law directly into the fully discrete linear system of equations, which will be derived in Section \ref{numsec-3}. In that section, we approximate the spatial derivative in the main equation by a second-order central difference and the Neumann condition at the right endpoint by a second-order one-sided difference, while we perform the time iteration by the Crank--Nicolson. The fully discrete scheme for the nonlinear closed-loop system is also derived in the same section, where the nonlinear term is treated by a Picard linearization. Finally, in Section \ref{numsec-4}, we present numerical simulations illustrating the stable behavior of the closed-loop system, together with a discussion verifying the convergence order of the proposed scheme.

\subsection{Matrix representations of the operators in the feedback law}
\label{numsec-1}
We recall that the backstepping kernel $k$, which satisfies the kernel model \eqref{kernel_pde}, is given by the series representation \eqref{ksolrep}. For the numerical implementation, we set $y_j = x_j$, $j \in \{1, \dotsc, N_x\}$ and consider the triangular grid $\mathcal T_h := \{(x_i,y_j):1\leq j\leq i\leq N_x\}
\subset\mathcal T$. {Then, we truncate the series \eqref{ksolrep} as
\begin{equation*}
	k^M_{i,j} := k^M(x_i,y_j) = -\frac{\mu y_j}{2 (\nu + i\alpha)} \sum_{m=0}^M \left(-\frac{\mu}{4 (\nu + i\alpha)}\right)^m \frac{(x_i^2 - y_j^2)^m}{m! (m+1)!}, \quad 1 \leq j \leq i \leq N_x,
\end{equation*}
where we choose the truncation index $M$ such that
\begin{equation*}
\underset{1 \leq j \leq i \leq N_x}{\max}\left|k^{M+1}_{i,j}-k^M_{i,j}\right| < \textit{eps}.
\end{equation*}
Here, \textit{eps} denotes the floating-point precision, around $10^{-16}$. Hence, the effect of the kernel truncation is negligible and does not influence the observed numerical behavior.} Let $\mathbf{K}^h$ be the discrete matrix representation of the integral operator $K$ given by $(K \varphi)(x) = \int_0^x k(x,y) \varphi(y) dy$. Employing the composite trapezoidal rule to approximate the integral, we express $\mathbf{K}^h$ as an $N_x \times N_x$ matrix as
\begin{equation*}
	(\mathbf{K}^h)_{i,j} =
	\begin{cases}
		0,  &\text{if }i = 1\text{ or }j > i, \\
        \delta_x k^M_{i,j}, &\text{if } 1<j<i, \\
		\frac{\delta_x}{2}k^M_{i,j}, &\text{otherwise}. \\
	\end{cases}
\end{equation*}
Next, let us form the discrete matrix representation, $\mathbf{P}_N^h$, of the projection operator $P_N$. To this end, let $\mathbf{e}_m$ be the grid function approximation to $e_m(x) = \sqrt{\frac{2}{L}}\sin\left(\frac{(2m-1)\pi x}{2L}\right)$. Then, for a given $\varphi \in L^2(0,L)$, the projection $P_N\varphi$ is approximated on the grid by
\begin{equation} \label{pn1}
P_N\varphi \approx \sum_{m=1}^N \hat{\varphi}_m \mathbf{e}_m = 
    \big[\mathbf e_1 \, \cdots \, \mathbf e_N\big]
    \begin{bmatrix}
        \hat{\varphi}_1 \\
        \vdots \\
        \hat{\varphi}_N
    \end{bmatrix}.
\end{equation}
Denote the $N_x \times N$ matrix $\big[\mathbf e_1 \, \cdots \, \mathbf e_N\big]$ that appears on the right-hand side of \eqref{pn1} by $\mathbf W$. The Fourier coefficients $\hat{\varphi}_m = (\varphi,e_m)_2$ is approximated by the composite trapezoidal rule as
\begin{equation} \label{pn2}
    \hat{\varphi}_m = (\varphi,e_m)_2
    \approx \mathbf{e}_m^T \mathbf{Q} \boldsymbol{\varphi},
    \quad m=1,\dotsc,N,
\end{equation}
where $\mathbf{Q} = \delta_x \operatorname{diag}\left(\frac{1}{2},1,\dotsc,1,\frac{1}{2}\right)$ is the $N_x-$dimensional diagonal matrix whose entries are the trapezoidal quadrature weights. Combining \eqref{pn1}-\eqref{pn2}, we obtain $\mathbf{P}_N^h$ as
\begin{equation*}
    \mathbf{P}_N^h = \mathbf{W}\mathbf{W}^T\mathbf{Q}.
\end{equation*}
Next, to obtain a discrete matrix representation $\mathbf{\Upsilon}_N^h$ of $\Upsilon_N$, we first recall the recursive definition
\begin{equation} \label{iter-num}
    \begin{split}
	\Upsilon_{m} \varphi \equiv& (I - \Upsilon_{m-1})[K_{m} \varphi] -\dfrac{\left( (I-\Upsilon_{m-1})[K_m \varphi],e_{m}\right)_2}{1+\left((I-\Upsilon_{m-1})[K e_m],e_{m}\right)_2}(I-\Upsilon_{m-1})[K e_m], \quad m = 1, \dotsc, N, \\
	\Upsilon_0 =& 0.
    \end{split}
\end{equation}
First, we set $\mathbf{B}_m^h := (\mathbf I^h - \mathbf\Upsilon_{m-1}^h)\mathbf K^h_m$, where $\mathbf I^h$ is the $N_x-$dimensional identity matrix. Note that the same operator appears in the numerator of the second term of \eqref{iter-num}, which can be viewed as a linear mapping $\varphi \mapsto \left( (I-\Upsilon_{m-1})[K_m \varphi],e_{m}\right)_2$. Using the composite trapezoidal rule, this mapping can be approximated by
\begin{equation*}
    \boldsymbol{\varphi} \to \mathbf{e}_m^T \mathbf{Q} \mathbf{B}_m^h \boldsymbol{\varphi}.
\end{equation*}
Therefore, we define the associated row vector approximation as
\begin{equation}
    \mathbf{r}^h_m := \mathbf{e}_m^T \mathbf{Q} \mathbf{B}_m^h.
\end{equation}
Now, we define the column vector $\mathbf c_m^h$ that approximates the term $(I-\Upsilon_{m-1})[K e_m]$ that appears in the second term of \eqref{iter-num} by
\begin{equation*}
    \mathbf c_m^h := (\mathbf I^h-\mathbf\Upsilon_{m-1}^h)\mathbf K^h \mathbf e_m.
\end{equation*}
Then, using numerical integration again, the denominator in \eqref{iter-num} is approximated by $ d_m^h := 1 +  \mathbf{e}_m^T\mathbf{Q}\mathbf{c}_m^h$.
Combining the approximations $\mathbf{B}_m^h$, $\mathbf{r}^h_m$, $\mathbf c_m^h$ and $d_m^h$ constructed above, we obtain the following recursion
\begin{equation} \label{iter-discrete}
    \mathbf{\Upsilon}_m^h = \mathbf{B}_m^h - \frac{1}{d_m^h} \mathbf c_m^h \mathbf{r}^h_m, \quad m = 1, \dotsc, N.
\end{equation}
initialized with $\mathbf{\Upsilon}_0^h = \mathbf{0}^h \in \mathbb{C}^{N_x \times N_x}$, which yields the matrix representation $\mathbf{\Upsilon}_N^h$ of $\Upsilon_N$. Algorithm~\ref{alg:ctrl-row-op} summarizes this procedure.
\begin{center}
	\begin{algorithm}[H]
\caption{Numerical algorithm for construction of the discrete matrix representation $\mathbf\Upsilon_N^h$.}
\begin{algorithmic}[1]
\REQUIRE $\mathbf K^h$, $\{\mathbf P_m^h\}_{m=1}^N$, $\{\mathbf e_m\}_{m=1}^N$

\STATE Set $\mathbf{Q} = \delta_x \operatorname{diag}\left(\frac{1}{2},1,\dotsc,1,\frac{1}{2}\right)$
\STATE Set $\mathbf\Upsilon_0^h=\mathbf 0$

\FOR{$m=1,\dotsc,N$}
    \STATE $\mathbf K_m^h \gets \mathbf K^h\mathbf P_m^h$
    \STATE $\mathbf B_m^h \gets (\mathbf I^h-\mathbf\Upsilon_{m-1}^h)\mathbf K_m^h$
    \STATE $\mathbf r_m^h \gets \mathbf e_m^T\mathbf Q\mathbf B_m^h$
    \STATE $\mathbf c_m^h \gets (\mathbf I^h-\mathbf\Upsilon_{m-1}^h)\mathbf K^h\mathbf e_m$
    \STATE $d_m^h \gets 1+\mathbf e_m^T\mathbf Q\mathbf c_m^h$
    \STATE $\mathbf\Upsilon_m^h \gets \mathbf B_m^h - \dfrac{1}{d_m^h}\mathbf c_m^h\mathbf r_m^h$
\ENDFOR
\RETURN $\mathbf\Upsilon_N^h$
\end{algorithmic}
\label{alg:ctrl-row-op}
\end{algorithm}
\end{center}

\subsection{An approximation for the control input}
\label{numsec-2}
We next derive a discrete approximation of the control input $b_u(t)$. Since $b_u(t)$ depends linearly on $u$, its discrete counterpart can be represented by a row vector, denoted by $\mathbf b_u^h$. Note that in the next subsection, the time integration will be performed by the Crank--Nicolson, which is an implicit scheme. Therefore, to preserve the implicit structure, we approximate the control input at the new time level $t_{n+1}$, i.e.,
\begin{equation*}
    b_u(t_{n+1}) \approx \mathbf{b}_u^h \mathbf{u}^{n+1},
\end{equation*}
where the approximate solution $\mathbf{u}^{n+1}$ is still unknown.

Let us first consider the integral part of the control input \eqref{numeric-ctrl}. Let
$\mathbf k_{x,L}^M :=
\begin{bmatrix}
    k_x^M(L,x_1) & \cdots & k_x^M(L,x_{N_x})
\end{bmatrix}$
be an approximation for $k_x(L,y)$. Using the composite trapezoidal rule, we obtain
\begin{equation}
    \label{numeric-ctrl-w1}
    \int_0^L k_x(L,y)P_Nw(y,t_{n+1}) dy \approx \mathbf k_{L,x}^M \mathbf Q \mathbf P_N^h \mathbf w^{n+1}.
\end{equation}
Next, the boundary term in \eqref{numeric-ctrl} can be approximated as
\begin{equation}
    \label{numeric-ctrl-w2}
    -\frac{\mu L}{2 (\nu + i \alpha)} P_N w(x,t_{n+1}) \big|_{x = L} \approx - \frac{\mu L}{2 (\nu + i \alpha)} \mathbf{P}_N^h \mathbf{w}^{n+1}\big|_{i = N_x}.
\end{equation}
Combining \eqref{numeric-ctrl-w1}-\eqref{numeric-ctrl-w2}, we set the row vector $\mathbf{b}_w^h$ as
\begin{equation*}
    \mathbf{b}_w^h := \mathbf k_{L,x}^M \mathbf Q \mathbf P_N^h - \frac{\mu L}{2 (\nu + i \alpha)} \mathbf{P}_{N,L}^h,
\end{equation*}
where $\mathbf{P}_{N,L}^h \in \mathbb{C}^{N_x}$ denotes the last row of $\mathbf{P}_N$. Finally, using the relation $w = T_N^{-1} u = (I - \Upsilon_N) u$ together with its discrete matrix representation, we obtain
\begin{equation} \label{row-approx-ctrl}
    \mathbf{b}_u^h = \mathbf{b}_w^h (\mathbf{I}^h - \mathbf{\Upsilon}_N^h) \in \mathbb{C}^{N_x},
\end{equation}
which is the desired row vector approximation of the control operator.
\begin{remark}
    Since $\mathbf b_u^h$ is a row vector, the approximation $\mathbf b_u^h\mathbf u^{n+1}$ is a linear combination of the nodal values $u_i^{n+1}$, $i=1,\dotsc,N_x$. That is, the discretization of the boundary condition $u_x(L,t_{n+1})=b_u(t_{n+1})$ yields a linear equation in the unknown nodal values $u_i^{n+1}$, $i=1,\dotsc,N_x$. This allows the boundary feedback law to be implemented directly into the linear system of equations that will be derived in the next subsection to compute the unknown approximate solution at the time level $n+1$.
\end{remark}

\subsection{The fully discrete closed-loop systems}
\label{numsec-3}
We first derive the fully discrete form for the linear closed-loop system, namely $\kappa=\beta=0$, by employing finite differences. The nonlinear case is then treated by incorporating the nonlinear term $f(u)=(\kappa+i\beta)|u|^p u$ into the same finite difference framework.
\\

\noindent \textit{Linear case.} We begin with the spatial discretization of the main equation. To this end, let $\mathbf{\Delta}_0 : \mathbb{C}^{N_x} \to \mathbb{C}^{N_x-2}$ denote the second-order central difference operator used to approximate the second-order spatial derivative at the interior grid points $x_2, \dotsc, x_{N_x - 1}$. We define the operator $\mathbf{A} : \mathbb{C}^{N_x} \to \mathbb{C}^{N_x-2}$ by
\begin{equation*}
    \mathbf{A} := (\nu + i\alpha)  \mathbf{\Delta}_0 + \gamma \mathbf{I}
\end{equation*}
where $\mathbf{I}: \mathbb{C}^{N_x} \to \mathbb{C}^{N_x-2}$ denotes the restriction of the identity matrix to the interior grid points. Applying the above spatial discretization procedure together with the Crank--Nicolson scheme in time, we obtain the fully discrete scheme for the main equation at the interior grid points
\begin{equation} \label{linear-discrete}
	\left(\mathbf{I} - \frac{\delta_t}{2} \mathbf{A}\right) \mathbf{u}^{n+1} = \mathbf{F}_l^n, \quad n = 0, 1, \dotsc, N_t - 1,
\end{equation}
where $\mathbf{F}_l^n := \left(\mathbf{I} + \frac{\delta_t}{2} \mathbf{A}\right) \mathbf{u}^{n}$. It remains to implement initial and boundary conditions. Let $\mathbf u^0$ be the grid approximation of $u_0$, i.e., $u_0 \approx \mathbf{u}^0$. From the left boundary condition we have $u_1^{n+1} = 0$. Regarding the right boundary condition, which is of Neumann type, we use the second-order one-sided finite difference operator $\mathbf{\Delta}_{-}: \mathbb{C}^{N_x} \to \mathbb{C}$ and write
\begin{equation*}
    \mathbf{\Delta}_{-} \mathbf{u}^{n+1} = \mathbf{b}_u^h \mathbf{u}^{n+1} \Longleftrightarrow \left(\mathbf{\Delta}_{-} - \mathbf{b}_u^h\right)\mathbf{u}^{n+1} = 0.
\end{equation*}
Combining the fully discrete main equation \eqref{linear-discrete}, together with the discretized initial-boundary conditions, the fully discrete linear closed-loop system reads:
\begin{equation} \label{disc_lin}
	\begin{cases}
		\text{Given } \mathbf{u}^0 \in \mathbb{C}^{N_x}, \text{ find } \mathbf{u}^{n+1} \in \mathbb{C}^{N_x} \text{ such that} \\
		\left(\mathbf{I} - \frac{\delta_t}{2} \mathbf{A}\right) \mathbf{u}^{n+1} = \mathbf{F}_l^n, \quad n = 0, \dotsc, N_t - 1, \\
		u_1^{n+1} = 0, \quad \left(\mathbf{\Delta}_{-} - \mathbf{b}_u^h\right) \mathbf{u}^{n+1} = 0.
	\end{cases}
\end{equation}

\begin{algorithm}[H]
	\begin{algorithmic}[1]
        \REQUIRE $\mathbf u^0$, $N_t$
        \STATE $\mathbf b_u^h \gets \texttt{ControlRowOperator}\left(\mathbf K^h, \{\mathbf P_m^h\}_{m=1}^N, \{\mathbf e_m\}_{m=1}^N\right)$ \quad (Algorithm \ref{alg:ctrl-row-op})
		\FOR{$n=0 \to N_t -1 $}
		      \STATE Solve \eqref{disc_lin} for $\mathbf{u}^{n+1}$;
		\ENDFOR
        \RETURN $\{\mathbf u^n\}_{n=0}^{N_t}$
	\end{algorithmic}
	\caption{Numerical algorithm for linear closed-loop system.}
	\label{alg:numlin}
\end{algorithm}

\noindent\textit{Nonlinear case.} Using the same discretization strategy described for the linear case, we obtain the following fully discrete scheme for the nonlinear model at the interior grid points as
\begin{equation} \label{disc_nonlin}
	\left(\mathbf{I} - \frac{\delta_t}{2} \mathbf{A}\right)\mathbf{u}^{n+1} + \frac{\delta_t}{2} \mathbf{F}_{nl}^{n+1} = \mathbf{F}_l^n - \frac{\delta_t}{2} \mathbf{F}_{nl}^n,
\end{equation}
where $\mathbf{F}_{nl}^n := (\kappa + i\beta)\left(\mathbf{u}^{n}  * \left|\mathbf{u}^{n}\right|^p\right)$. We treat the nonlinear unknown term via Picard iteration method. To this end, let $\mathbf{u}^{n,s}$ be an approximation to $\mathbf{u}^{n+1}$ and consider the iteration
\begin{equation} \label{iter_nonlin}
	\begin{cases}
		\mathbf{u}^{n,s+1} = \mathbf{u}^{n,s} + \mathbf{du}, \quad s = 0, 1, 2 \dotsc, \\
		\mathbf{u}^{n,0} := \mathbf{u}^n,
	\end{cases}
\end{equation}
together with imposing the boundary conditions from the left endpoint and the Neumann actuation from the right endpoint
\begin{equation}
    \label{iter_nonlin-bcs}
    u_1^{n,s+1}=0, \quad \left(\mathbf\Delta_- - \mathbf b_u^h\right)\mathbf u^{n,s+1}=0.
\end{equation}
until $\underset{1 \leq i \leq N_x}{\max} |\mathbf{du}|$ is small enough. Once this criterion is satisfied, we stop the iteration and set $\mathbf{u}^{n+1} = \mathbf{u}^{n,s} + \mathbf{du}$.

To compute $\mathbf{du}$ for each iteration, we replace the unknown approximation $\mathbf{u}^{n,s} + \mathbf{du}$, for the linear factor and set the nonlinear factor equal to its current value:
\begin{equation*}
    \begin{split}
	\mathbf{F}_{nl}^{n+1} &= (\kappa + i\beta) \left(\mathbf{u}^{n+1} * \left|\mathbf{u}^{n+1}\right|^p\right) \\
    &\approx (\kappa + i\beta) \left((\mathbf{u}^{n,s} + \mathbf{du}) * \left|\mathbf{u}^{n,s}\right|^p \right) \\
    &= \mathbf{F}_{nl}^{n,s} + (\kappa + i\beta) \left(\left|\mathbf{u}^{n,s}\right|^p * \mathbf{du}\right).
    \end{split}
\end{equation*}
Employing this approximation in \eqref{disc_nonlin}, we obtain the following linear difference scheme in $\mathbf{du}$ at the interior grid points:
\begin{equation} \label{dw_lin}
	\left(\mathbf{I} - \frac{\delta_t}{2} \mathbf{A} \right)\mathbf{du} + \frac{\delta_t}{2}(\kappa + i\beta) \left|\mathbf{u}^{n,s}\right|^p *\mathbf{du} \\ = \mathbf{F}_l^n - \frac{\delta_t}{2}\mathbf{F}_{nl}^n - \left(\mathbf{I} - \frac{\delta_t}{2} \mathbf{A}\right)\mathbf{u}^{n,s} - \frac{\delta_t}{2}\mathbf{F}_{nl}^{n,s}.
\end{equation}
Similarly, replacing $\mathbf u^{n+1}$ by $\mathbf u^{n,s}+\mathbf{du}$ in \eqref{iter_nonlin-bcs} gives
\begin{equation}
    \label{dw_lin-bc}
    \begin{split}
    du_1 &= - u_1^{n,s}, \\
    \left(\mathbf\Delta_- - \mathbf b_u^h\right) \mathbf{du} &= -\left(\mathbf\Delta_- - \mathbf b_u^h\right) \mathbf{u}^{n,s}.
    \end{split}
\end{equation}
\eqref{dw_lin}-\eqref{dw_lin-bc} form a linear system in $\mathbf{du}$. Solving this system gives $\mathbf{du}$, hence $\mathbf{u}^{n,s+1}$, at both interior and boundary grid points. The resulting time iteration combined with the Picard linearization procedure is summarized in Algorithm \ref{alg:numnonlin}.
\begin{algorithm}[H]
	\begin{algorithmic}[1]
		\REQUIRE $\mathbf u^0$, $N_t$, $\text{TOL}$
        \STATE $\mathbf b_u^h \gets \texttt{ControlRowOperator}\left(\mathbf K^h, \{\mathbf P_m^h\}_{m=1}^N, \{\mathbf e_m\}_{m=1}^N\right)$ \quad (Algorithm \ref{alg:ctrl-row-op})
        \FOR{$n=0 \to N_t - 1$}
            \STATE $\mathbf u^{n,0} \gets \mathbf u^n$
            \STATE $s \gets 0$
            \STATE $\text{err} \gets \infty$
            \WHILE{$\text{err} > \text{TOL}$}
                \STATE \text{Solve \eqref{dw_lin}-\eqref{dw_lin-bc} for }$\mathbf{du}$
		          \STATE $\mathbf{u}^{n,s+1} \gets \mathbf{u}^{n,s} + \mathbf{du}$
		          \STATE $s \gets s + 1$
                \STATE $\text{err} \gets \|\mathbf{du}\|_\infty$
		      \ENDWHILE
		\STATE $\mathbf{u}^{n+1} \gets \mathbf{u}^{n,s}$
		\ENDFOR
        \RETURN $\{\mathbf u^n\}_{n=0}^{N_t}$
	\end{algorithmic}
	\caption{Numerical algorithm for nonlinear closed-loop system.}
	\label{alg:numnonlin}
\end{algorithm}

\subsection{Numerical experiments}
\label{numsec-4}
In this subsection, we first present two numerical experiments illustrating the theoretical stabilization results under the hypotheses of Theorem \ref{thm-stab-linear} and Theorem \ref{thm-stab-nonlinear}. We also provide a numerical verification of the discretization order at the end of the section.

For the numerical experiments verifying the stabilization results below, we take $N_x=1001$ spatial grid points and $N_t=1001$ time steps. These computations were performed in MATLAB on a computer equipped with an AMD Ryzen 7 5800X 8-Core Processor at 3.80 GHz and 16 GB RAM, running Windows 10 Enterprise. The elapsed computational times were approximately $24.01$ seconds for the linear model and $86.89$ seconds for the nonlinear model.
\\

\paragraph{\bf{Numerical experiment 1.}} Consider the linear model
\begin{eqnarray} \label{pde_lin_num}
	\begin{cases}
		u_t - (1 + 3i) u_{xx} - 23u = 0, &(x,t) \in(0,1) \times (0,T), \\
		u(0,t) = 0, u_x(1,t) = b(t), & t \in (0,T),\\
		u(x,0) = \sin \left(\frac{\pi x}{2}\right) + \frac{1}{2} \sin \left(\frac{3 \pi x}{2}\right), &x \in (0,1).
	\end{cases}
\end{eqnarray}
Here $\lambda_2 = \frac{9\pi^2}{4}$, $\lambda_3 = \frac{25 \pi^2}{4}$ and $\lambda_2 < \frac{\gamma}{\nu} = 23 < \lambda_3$, so the instability level of the problem is $2$. Therefore, by Theorem \ref{thm-stab-linear}-(ii), we should be able to stabilize zero equilibrium by using a control law involving $N = 2$ Fourier modes of the state:
\begin{equation} \label{exp1}
	b_u(t) = \int_0^1 k_x(1,y) \Gamma_2 [P_2 u](y,t) dy - \frac{\mu}{2 (1 + 3i)} \Gamma_2 [P_2 u](x,t) \bigg|_{x = 1}.
\end{equation}
Using the problem parameters, we calculate
\begin{equation*}
	2(\gamma - \nu \lambda_1) \left(1 - \frac{1}{2N+1}\right)^{-1} \simeq 51.3, \quad  2\nu \lambda_{3} \simeq 123.4,
\end{equation*}
Thus, in view of Theorem \ref{thm-stab-linear}-(ii), choosing $\mu = 60$ fulfills the criterion
\begin{equation}
	2(\gamma - \nu \lambda_1) \left(1 - \frac{1}{2N+1}\right)^{-1}< \mu < 2\nu \lambda_{N+1}.
\end{equation}
Note also that for $j = 1$
\begin{equation*}
	1 +	\left(K e_{1},e_{1}\right)_2 \approx d_1^h = 1 +  \mathbf{e}_1^T\mathbf{Q}\mathbf{c}_1^h = -0.26 + 0.87 i \neq 0
\end{equation*}
and for $j = 2$
\begin{equation*}
	1 + \left((I-\Upsilon_{1})[K e_{2}],e_{2}\right)_2 \approx d_2^h = 1 +  \mathbf{e}_2^T\mathbf{Q}\mathbf{c}_2^h \approx 0.67 + 0.17 i \neq 0
\end{equation*}
so that, in view of Lemma \ref{invlem-2}, $T_2^{-1}$ can be constructed via the recursion \eqref{Phij} for the pair $(\mu,N) = (60,2)$. See Fig. \ref{fig:ctrl_lin_3d} for the graph of the solution $\left|u(x,t)\right|$ of the stabilized linear model \eqref{pde_lin_num} and Fig. \ref{fig:ctrl_lin} for time evolutions of $L^2(0,1)$ norm and $H^1(0,1)$ norm of the solution.
\begin{figure}[!h]
    \begin{subfigure}[b]{0.49\textwidth}
		\centering
		\includegraphics[scale=0.7]{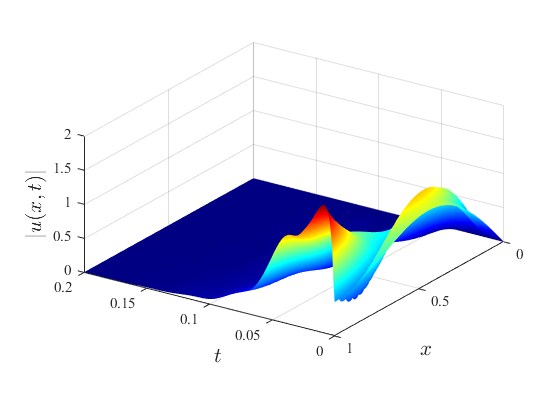}
        \vspace{-1 cm}
		\caption{Solution of the stabilized linear model~\eqref{pde_lin_num}.}
		\label{fig:ctrl_lin_3d}
	\end{subfigure}
	\\
	\begin{subfigure}[b]{0.49\textwidth}
		\centering
		\includegraphics[width=\textwidth]{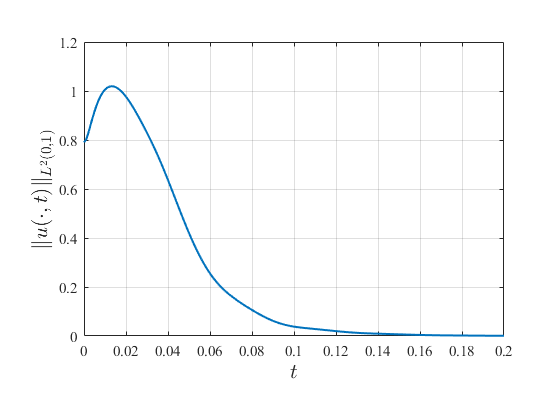}
		\caption{Time evolution of $L^2-$norm of solution of stabilized linear model~\eqref{pde_lin_num}.}
		\label{fig:ctrl_lin_l2}
	\end{subfigure}
	\hfill
	\begin{subfigure}[b]{0.49\textwidth}
		\centering
		\includegraphics[width=\textwidth]{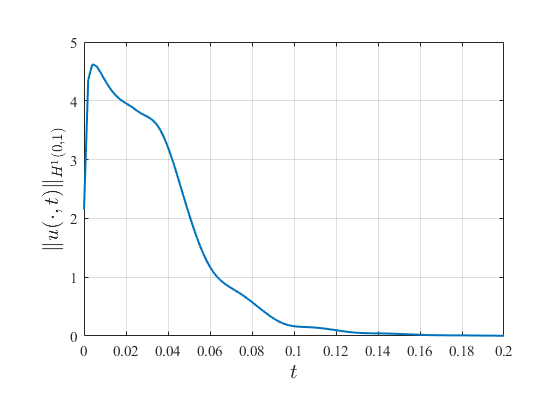}
        \caption{Time evolution of $H^1-$norm of solution of stabilized linear model~\eqref{pde_lin_num}.}
        \label{fig:ctrl_lin_h1}
	\end{subfigure}
	\caption{Closed-loop dynamics of the linear model \eqref{pde_lin_num} under the proposed feedback control \eqref{exp1}. The solution decays to the zero equilibrium in time. This is also reflected in the time evolution of the $L^2-$norm and $H^1-$norm, both of which decay to zero, which illustrate the stabilization of the zero equilibrium in these corresponding norms.}
	\label{fig:ctrl_lin}
\end{figure}

\paragraph{\bf{Numerical experiment 2.}} As a second example, we consider a nonlinear model
\begin{eqnarray} \label{pde_nonlin_num}
	\begin{cases}
		u_t - (1 + i) u_{xx} + (1 + 4i) |u|^2 u - 10u = 0, &(x,t) \in(0,1) \times (0,T), \\
		u(0,t) = 0, u_x(1,t) = b(t), & t \in (0,T),\\
		u(x,0) = \frac{3}{2}x(2-x) + 2i\sin(\frac{3 \pi x}{2}), &x \in (0,1).
	\end{cases}
\end{eqnarray}
In the absence of the control input, i.e. if we have a homogeneous Neumann type boundary condition on the right endpoint, the solution evolves to a nontrivial steady state (see Fig. \ref{fig:woctrl_nonlin}).
\begin{figure}[h!]
	\begin{subfigure}[b]{0.49\textwidth}
		\centering
		\includegraphics[width=\textwidth]{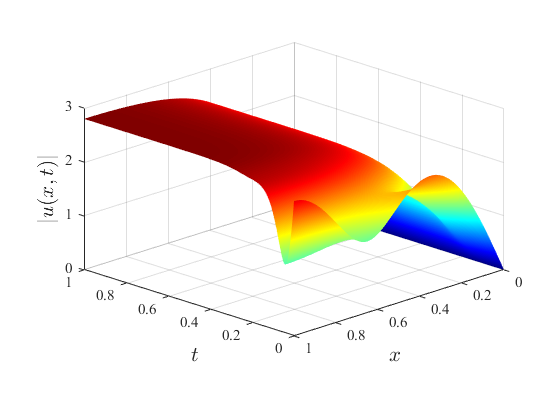}
		\caption{Solution of the uncontrolled nonlinear model~\eqref{pde_nonlin_num}.}
		\label{fig:woctrl_nonlin_3d}
	\end{subfigure}
	\hfill
	\begin{subfigure}[b]{0.49\textwidth}
		\centering
		\includegraphics[width=\textwidth]{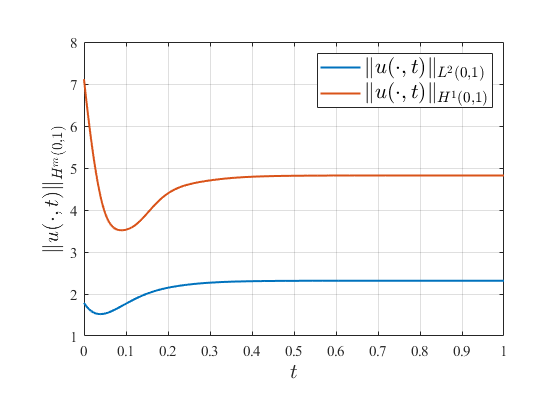}
		\caption{Time evolution of $L^2-$norm and $H^1-$norm of the solution of the uncontrolled nonlinear model \eqref{pde_nonlin_num}.}
		\label{fig:woctrl_nonlin_norm}
	\end{subfigure}
	\caption{Uncontrolled dynamics of the nonlinear model \eqref{pde_nonlin_num} with $b = 0$. In the absence of feedback control, the solution approaches a nontrivial steady state after a transient period. This behavior is also reflected in the $L^2-$norm and $H^1-$norm, which eventually become stationary and remain away from zero. This also illustrates that the zero equilibrium is not stable with respect to these norms in the absence of feedback control.}
	\label{fig:woctrl_nonlin}
\end{figure}

Let us consider the part (i) case of the Theorem \ref{thm-stab-nonlinear}: $\mu > \gamma - \lambda_1 \nu$ is given and find an associated value for $N$ so that the criterion \eqref{N_condd} holds. Let $\mu = 12$, which is greater than $\gamma - \lambda_1 \nu = 10 - \frac{\pi^2}{4}$. Then, 
\begin{equation} \label{N_conddd}
	N > \max \left\{\frac{\mu}{4 \nu \lambda_1} - \frac{1}{2},\frac{\mu}{2(\mu + \nu \lambda_1 - \gamma)} - \frac{1}{2} \right\} = \max \left\{\frac{12}{\pi^2} - \frac{1}{2}, \frac{24}{8 + \pi^2} - \frac{1}{2}\right\} \approx 0.843.
\end{equation}
Hence, $N=1$ satisfies \eqref{N_conddd} and we can consider a control law involving only the first Fourier mode of the state, given by
\begin{equation} \label{exp2}
	b_u(t) = \int_0^1 k_x(1,y) \Gamma_1 [P_1 u](y,t) dy - \frac{6}{1 + i} \Gamma_1 [P_1 u](x,t) \bigg|_{x = 1}.
\end{equation}
Note that
\begin{equation*}
	1 +	\left(K e_{1},e_{1}\right)_2 \approx d_1^h = 1 +  \mathbf{e}_1^T\mathbf{Q}\mathbf{c}_1^h = 0.42 + 0.37 i \neq 0.
\end{equation*}
Thus, thanks to the Lemma \ref{invlem-2}, the bounded inverse $T_1^{-1}$ can be constructed via the recursion \eqref{Phij}. See Fig. \ref{fig:ctrl_nonlin} for the simulations of the controlled nonlinear model \eqref{pde_nonlin_num}.
\begin{figure}[h!]
	\begin{subfigure}[b]{0.49\textwidth}
		\centering
		\includegraphics[scale=0.7]{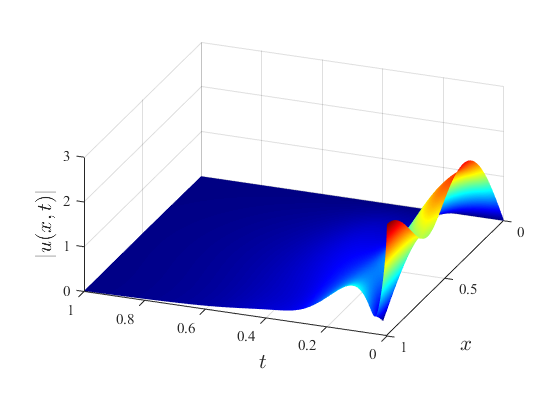}
        \vspace{-1 cm}
		\caption{Solution of the stabilized nonlinear model~\eqref{pde_nonlin_num}.}
		\label{fig:ctrl_nonlin_3d}
	\end{subfigure}
	\\
	\begin{subfigure}[b]{0.49\textwidth}
		\centering
		\includegraphics[width=\textwidth]{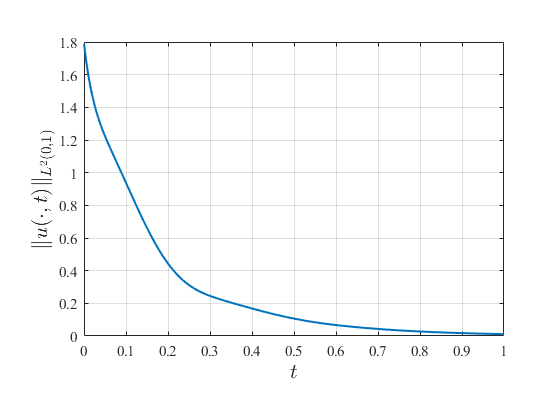}
		\caption{Time evolution of $L^2-$norm of solution of stabilized nonlinear model~\eqref{pde_nonlin_num}.}
		\label{fig:ctrl_nonlin_l2}
	\end{subfigure}
	\hfill
	\begin{subfigure}[b]{0.49\textwidth}
		\centering
		\includegraphics[width=\textwidth]{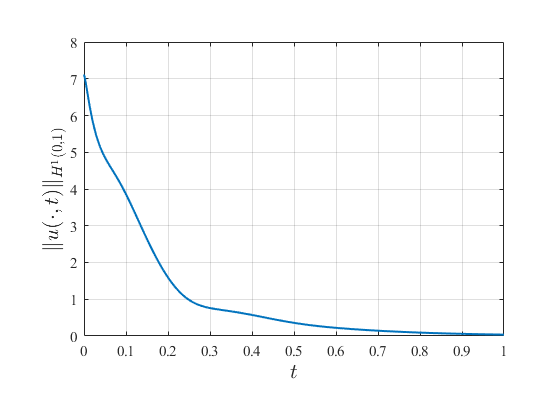}
		\caption{Time evolution of $H^1-$norm of solution of stabilized nonlinear model~\eqref{pde_nonlin_num}.}
		\label{fig:ctrl_nonlin_h1}
	\end{subfigure}
	\caption{Closed-loop dynamics of the nonlinear model \eqref{pde_nonlin_num} under the proposed feedback control \eqref{exp2}. The solution decays to the zero equilibrium in time. This is also reflected in the time evolution of the $L^2-$norm and $H^1-$norm, both of which decay to zero, which illustrate the stabilization of the zero equilibrium in these corresponding norms.}
	\label{fig:ctrl_nonlin}
\end{figure}
\\

{\paragraph{\bf Numerical experiment on spatial and temporal convergence orders.} First, we verify the convergence order with respect to the spatial discretization. We set $L = 1$ and compute approximate solutions on successively finer spatial grids with
\begin{equation*}
    N_x=2^q+1,\quad q=4,\dotsc,12,
\end{equation*}
The corresponding approximate solutions are denoted by $\mathbf u_q^n$. For these values of $N_x$, we fix the time step size as $\delta_t=2\times 10^{-4}$, which is lesser than the finest spatial mesh size $\delta_x =2^{-12}$. Since both the spatial discretization and the Crank--Nicolson scheme are second-order accurate, this choice ensures that the temporal discretization error remains smaller than the spatial discretization error and therefore does not affect the observed spatial convergence order.

The errors are computed by using two subsequent numerical solutions as
\begin{equation*}
    E_q := \max_{0\leq n\leq N_t} \left\|\mathbf u_q^n - \mathbf u_{q+1}^n \right\|_{\ell_h^2}, \quad q=4,\dotsc,11,
\end{equation*}
where $\mathbf u_{q+1}^n$ is restricted to the grid associated with $\mathbf u_q^n$. The observed convergence order is then formulated by
\begin{equation} \label{order}
    p_q = \log_2\left(\frac{E_q}{E_{q+1}}\right), \quad q=4,\dotsc,10.
\end{equation}

In the spatial discretization, the second-order derivative in the main equation, the Neumann-type boundary condition and the discretization of all integrals used in deriving the row vector approximation of the control input are approximated by second-order accurate formulas. Therefore, we expect that the observed spatial convergence order should be close to $2$. Indeed, this is reflected in the numerical results which are summarized in Table \ref{tab:spatial-order}.
\begin{table}[H]
\centering
\begin{tabular}{r@{\hspace{1.0cm}}c@{\hspace{1.0cm}}c@{\hspace{1.0cm}}c}
\hline \hline
$q$ & Spatial Step Size ($\delta_x$) & Error ($E_q$) & Observed Order ($p_q$) \\
\hline
$4$  & $2^{-4}$ & $1.215931\times 10^{-2}$ & $1.999151$ \\
$5$  & $2^{-5}$ & $3.041616\times 10^{-3}$ & $2.008537$ \\
$6$   & $2^{-6}$ & $7.559179\times 10^{-4}$ & $2.006277$ \\
$7$  & $2^{-7}$ & $1.881590\times 10^{-4}$ & $2.003564$ \\
$8$  & $2^{-8}$ & $4.692369\times 10^{-5}$ & $2.001872$ \\
$9$  & $2^{-9}$ & $1.171571\times 10^{-5}$ & $2.000960$ \\
$10$ & $2^{-10}$ & $2.926979\times 10^{-6}$ & $2.000530$ \\
\hline\hline
\end{tabular}
\caption{The error values $E_q$ and the observed spatial convergence orders $p_q$ corresponding to different values of $q$ are presented. The results reflect the expected second-order spatial convergence, since second-order schemes are used for the spatial discretization to derive the fully discrete closed-loop systems.}
\label{tab:spatial-order}
\end{table}

Next, we verify the convergence order with respect to the time discretization. To this end, we fix the final time as $T_{\max}=1$ and compute numerical solutions successively by increasing the number of time grid points as
\begin{equation*}
    N_t = 2^q,\quad q=4,\dotsc,12.
\end{equation*}
For these values of $N_t$ we fix $N_x = 2^{12} + 1$ which is sufficiently fine so that the spatial discretization error does not spoil the observed temporal convergence order. For each $q$, we denote the corresponding approximate solution by $\mathbf u_q^n$. The error $E_q$ is given by
\begin{equation*}
    E_q :=
    \max_{0\leq n \leq 2^q}
    \left\|
    \mathbf u_q^n - \mathbf u_{q+1}^{2n}
    \right\|_{\ell_h^2},
    \quad q=4,\dotsc,11,
\end{equation*}
so that two consecutive numerical solutions are compared at the same time levels on the coarser time grid. The observed temporal convergence order is calculated as in \eqref{order}.

Since the Crank--Nicolson scheme is second-order accurate in time, the observed temporal convergence order is expected to approach $2$ on finer time grids. This situation is reflected in the numerical results presented in Table~\ref{tab:temporal-order}.

\begin{table}[H]
\centering
\begin{tabular}{r@{\hspace{1.0cm}}c@{\hspace{1.0cm}}c@{\hspace{1.0cm}}c}
\hline\hline
$q$ & Temporal Step Size ($\delta_t$) & Error ($E_q$) & Observed Order ($p_q$) \\
\hline
$4$  & $2^{-4}$  & $3.243780\times 10^{-1}$ & $0.726978$ \\
$5$  & $2^{-5}$  & $1.959788\times 10^{-1}$ & $0.790863$ \\
$6$  & $2^{-6}$  & $1.132754\times 10^{-1}$ & $1.150914$ \\
$7$  & $2^{-7}$  & $5.101242\times 10^{-2}$ & $1.573830$ \\
$8$  & $2^{-8}$  & $1.713586\times 10^{-2}$ & $1.849952$ \\
$9$  & $2^{-9}$  & $4.753515\times 10^{-3}$ & $1.998329$ \\
$10$ & $2^{-10}$ & $1.189756\times 10^{-3}$ & $2.038368$ \\
\hline\hline
\end{tabular}
\caption{The error values $E_q$ and the observed temporal convergence orders $p_q$ corresponding to different values of $q$ are presented. The observed orders approach the expected value $2$ on finer time grids, reflecting the second-order accuracy of the Crank--Nicolson scheme in time.}
\label{tab:temporal-order}
\end{table}}

\paragraph{\textbf{Acknowledgement.}} The authors would like to thank the reviewers for their careful reading and constructive comments, which helped improve the quality and clarity of the manuscript.

\appendix

\section{Solution of the forced linear CGL equation on a finite interval}
\label{utm-sol-s}

We derive in detail the unified transform solution formula  \eqref{lcgl-sol-T} for the modified forced linear problem \eqref{lcgl-ibvp-w}. Recall that this formula provides the starting point for deriving the linear estimates used in the proof of the local well-posedness Theorem \ref{lwp-t}. Furthermore, by means of the transformation \eqref{gamma-trans}, it readily implies a corresponding solution formula for the original forced linear CGL problem \eqref{lcgl-ibvp}.

The finite interval Fourier transform of a function $\varphi \in L^2(0, L)$ is defined  by
\begin{equation}\label{ft}
\what \varphi(k) = \int_0^L e^{-ikx} \varphi(x) dx, \quad k \in \mathbb C.
\end{equation}
Note that, due to the bounded range of values of $x$, $\what \varphi(k)$ is an entire function as a consequence of a  standard Paley-Wiener-type theorem (e.g. see \cite{s1994}, Theorem 7.2.4).
Moreover, by extending $\varphi(x)$ by zero outside $(0, L)$ and employing the inversion formula for the Fourier transform on the whole line, we readily deduce the inverse of the finite interval Fourier transform \eqref{ft} as 
\begin{equation}\label{ift}
\varphi(x) = \frac{1}{2\pi} \int_{\mathbb R} e^{ikx} \, \what \varphi(k) dk, \quad x \in (0, L).
\end{equation}

Taking the finite interval Fourier transform \eqref{ft} of the forced linear equation in \eqref{lcgl-ibvp-w}, we obtain 
\begin{equation}\label{lcgl-ft-stage1}
\p_t	\what w(k, t) - \left(\nu + i \alpha \right) \left\{ e^{-ikL} w_x(L, t) - w_x(0, t) +ik \left[e^{-ikL} w(L, t) - w(0, t) + ik \, \what w(k, t) \right]\right\}  = \what f(k, t).
\end{equation}
In view of the two boundary conditions and the notation $u_x(0, t)=g_1(t)$, $u(L, t)=h_0(t)$ for the two unknown boundary values involved in \eqref{lcgl-ft-stage1}, we have
\begin{equation}\label{uhat-ode}
\p_t \what w(k, t) +  \left(\nu + i \alpha\right) k^2 \what w(k, t) = - \left(\nu + i \alpha\right) \Big\{ g_1(t)-e^{-ikL} h_1(t) + ik \big[g_0(t)-e^{-ikL} h_0(t)\big]\Big\} +  \what f(k, t).
\end{equation}
Integrating this equation in $t$ and using the initial condition as well as the notation
\begin{equation}\label{tilde-def}
\omega = \omega(k^2) := \left(\nu + i \alpha\right) k^2,
\quad
\widetilde \varphi(\omega, t) = \int_0^t e^{\omega t'} \varphi(t') dt',
\end{equation}
we obtain the following spectral identity known in the unified transform literature as the global relation:
\begin{equation}\label{lcgl-gr}
\begin{aligned}
e^{\omega t}\,  \what w(k, t) 
&=
\what u_0(k)
-
\left(\nu + i \alpha\right) \left\{ \big[\widetilde g_1(\omega, t) + ik \widetilde g_0(\omega, t)\big] - e^{-ikL}  \left[ \widetilde h_1(\omega, t) + ik  \widetilde h_0(\omega, t)\right]\right\}
\\
&\quad
+ \int_0^t e^{\omega  t'} \what f(k, t') dt', \quad k\in\mathbb C.
\end{aligned}
\end{equation}
Inverting the global relation for $k\in \mathbb R$ via  \eqref{ift}, we find
\begin{align}\label{lcgl-ir}
w(x, t) 
&=
\frac{1}{2\pi} 
\int_{\mathbb R} e^{ikx-\omega t} \, \what u_0(k) \, dk \,
+ \frac{1}{2\pi} 
\int_{\mathbb R} e^{ikx-\omega t} 
\int_0^t e^{\omega  t'} \what f(k, t') dt' dk
\\
&\quad
-
\frac{\nu+i \alpha}{2\pi} 
\int_{\mathbb R} e^{ikx-\omega t}  \, \big[\widetilde g_1(\omega, t) + ik \widetilde g_0(\omega, t)\big]  \, dk 
+
\frac{\nu+i \alpha}{2\pi} 
\int_{\mathbb R} e^{ik(x-L)-\omega t}   \left[ \widetilde h_1(\omega, t) + ik  \widetilde h_0(\omega, t)\right] dk.
\nn
\end{align}

The expression \eqref{lcgl-ir} is not an explicit solution formula, as it involves the two unknown boundary values $g_1, h_0$ via their corresponding temporal transforms \eqref{tilde-def}. 
The key idea of the unified transform is to combine the fact that $\omega$ is invariant under the symmetry $k \mapsto -k$ with appropriate complex contour deformations in order to eliminate these unknown terms. 

Specifically, note that $\left| e^{ikx} \right| = e^{-\text{Im}(k) x}$ thus, since $x>0$, the exponential  $e^{ikx}$ is bounded for $\text{Im}(k)\geq 0$ and decays to zero as $|k|\to \infty$ in the upper half-plane $\left\{\text{Im}(k) > 0\right\}$.
Similarly, $e^{-\omega(t-t')}$ with $0<t'<t$ decays to zero in the regions $\left\{\text{Im}(k) > 0\right\} \setminus \overline{D^+}$ and $\left\{\text{Im}(k) < 0\right\} \setminus \overline{D^-}$, where 
\begin{equation}\label{dpm-def}
D^\pm := \left\{k\in\mathbb C: \text{Im}(k) \gtrless 0 \text{ and } \text{Re}(\omega) < 0\right\}.
\end{equation} 
Hence, the exponential $e^{ikx - \omega (t-t')}$  decays to zero as $|k|\to \infty$ inside the region $\left\{\text{Im}(k) > 0\right\} \setminus \overline{D^+}$ and is bounded on the closure of that region. Moreover, since $x<L$, the same is true for $e^{ik(x-L) - \omega(t-t')}$   inside the region $\left\{\text{Im}(k) < 0\right\} \setminus \overline{D^-}$.

Combining the above exponential decay with analyticity of the relevant integrands and Cauchy's integral theorem along the lines of the proof of Lemma \ref{uh-l} below, we may express \eqref{lcgl-ir} in the form
\begin{align}\label{lcgl-ir-def}
w(x, t) 
&=
\frac{1}{2\pi} 
\int_{\mathbb R} e^{ikx-\omega t} \, \what u_0(k) \, dk 
+ \frac{1}{2\pi} 
\int_{\mathbb R} e^{ikx-\omega t} 
\int_0^t e^{\omega  t'} \what f(k, t') dt' dk
\\
&\quad
-
\frac{\nu+i\alpha}{2\pi} 
\int_{\p D^+} e^{ikx-\omega t}  \, \big[\widetilde g_1(\omega, t) + ik \widetilde g_0(\omega, t)\big]  \, dk 
-
\frac{\nu+i\alpha}{2\pi} 
\int_{\p D^-} e^{ik(x-L)-\omega t}   \left[ \widetilde h_1(\omega, t) + ik  \widetilde h_0(\omega, t)\right] dk,
\nn
\end{align}
where $\p D^\pm$ are the \textit{positively} oriented boundaries of $D^\pm$ depicted in Fig. \ref{dpm-f}.
Indeed, in view of the fact that 
$$
\text{Re}(\omega)
=
-\nu  \left[\text{Im}(k) + \frac{\alpha  - \sqrt{\alpha^2 + \nu^2}}{\nu} \, \text{Re}(k)\right] \left[\text{Im}(k) + \frac{\alpha  + \sqrt{\alpha^2 + \nu^2}}{\nu} \, \text{Re}(k)\right]
,
$$
the boundary of the region $D^+ \cup D^-$ corresponds to the pair of straight lines given in \eqref{hyp-eq}, as shown in Fig. \ref{dpm-f}.

As noted earlier, the transformation $k \mapsto -k$ is a symmetry of $\omega$ and hence leaves the temporal transforms $\widetilde g_j(\omega, t), \widetilde h_j(\omega, t)$ invariant. Thus, when applied to the global relation \eqref{lcgl-gr}, it gives rise to the additional identity
\begin{equation}\label{lcgl-gr-}
\begin{aligned}
e^{\omega t}\,  \what w(-k, t) 
&=
\what u_0(-k)
-
\left(\nu + i \alpha\right) \left\{ \big[\widetilde g_1(\omega, t) - ik \widetilde g_0(\omega, t)\big] - e^{ikL}  \left[ \widetilde h_1(\omega, t) - ik  \widetilde h_0(\omega, t)\right]\right\}
\\
&\quad
+ \int_0^t e^{\omega  t'} \what f(-k, t') dt', \quad k\in\mathbb C.
\end{aligned}
\end{equation}
The two identities \eqref{lcgl-gr} and \eqref{lcgl-gr-} can be solved for  the unknown transforms $\widetilde g_1(\omega, t)$, $\widetilde h_0(\omega, t)$ to yield
\begin{align}\label{bv-sol-s}
\widetilde g_1(\omega, t)
&=
\frac{1}{(\nu+i\alpha) \left(e^{ikL} + e^{-ikL}\right)} 
\bigg\{
-e^{\omega t} \left[ e^{ikL} \what w(k, t) + e^{-ikL} \what w(-k, t) \right]
+
 \left[ e^{ikL} \what u_0(k) + e^{-ikL} \what u_0(-k) \right]
\nn\\
&\quad
+
 \int_0^t e^{\omega  t'} \left[e^{ikL} \what f(k, t') + e^{-ikL} \what f(-k, t') \right] dt'  
-
(\nu+i\alpha) ik \left(e^{ikL} - e^{-ikL} \right) \widetilde g_0(\omega, t)
+
2 (\nu+i\alpha) \widetilde h_1(\omega, t)
\bigg\},
\nn\\
\widetilde h_0(\omega, t)
&=
\frac{1}{(\nu+i\alpha) ik \left(e^{ikL} + e^{-ikL}\right)} 
\bigg\{
e^{\omega t} \left[ \what w(k, t) - \what w(-k, t) \right]
-
\left[ \what u_0(k) - \what u_0(-k) \right]
\nn\\
&\quad
-
\int_0^t e^{\omega  t'} \left[\what f(k, t') - \what f(-k, t') \right] dt'  
+
2(\nu+i\alpha)ik \widetilde g_0(\omega, t)
+
(\nu+i\alpha) \left(e^{ikL}-e^{-ikL}\right) \widetilde h_1(\omega, t)
\bigg\}.
\end{align}

The expressions \eqref{bv-sol-s} still involve unknown quantities, namely the transforms $\what w(\pm k, t)$. In this connection, we have the following crucial result:
\begin{lemma}\label{uh-l}
For all $(x, t) \in (0, L)\times (0, \infty)$, 
\begin{align}
&\int_{\p D^+} \frac{e^{ikx}}{e^{ikL} + e^{-ikL}} \big[ e^{ikL}   \what w(k, t) + e^{-ikL} \what w(-k, t)\big] dk = 0,
\label{uh-l1}
\\
&
\int_{\p D^-} \frac{e^{ik(x-L)}}{e^{ikL} + e^{-ikL}} \big[ \what w(k, t) -  \what w(-k, t)\big] dk = 0.
\label{uh-l2}
\end{align}
\end{lemma}

\begin{proof}
We only prove \eqref{uh-l1} as \eqref{uh-l2} can be established in an entirely analogous way.
Consider the circular arc
$$
C_R(0) := \left\{R e^{i\theta}: \ \theta_1 \leq \theta \leq \theta_2 \right\}
$$
with the angles $\theta_1, \theta_2$ corresponding to the arguments of the points of intersection of the straight lines \eqref{hyp-eq} with $C_R(0)$ and hence given by 
$$
\theta_1 = \tan^{-1}(\lambda),
\quad
\theta_2 = \pi - \tan^{-1}\left(\frac 1\lambda\right) 
=
\frac \pi 2 + \tan^{-1}(\lambda),
$$
where for the last equality we have used that $\tan^{-1}(\lambda) + \tan^{-1}\left(\frac 1\lambda\right) = \frac \pi 2$ for $\lambda>0$. 
Note that $\theta_1 \in (0, \frac \pi 2)$  and $\theta_2 \in (\frac \pi 2, \pi)$ since $\lambda = \frac{-\alpha + 
\sqrt{\alpha^2+\nu^2}}{\nu}$ is positive and finite.
%

By analyticity in $D^+$, we can deform the contour of integration $\p D^+$ of \eqref{uh-l1} to the limit as $R\to\infty$ of $C_R(0)$ so that
$$
\int_{\p D^+}   
\frac{e^{ikx}}{e^{ikL}+e^{-ikL}} \left[e^{ikL}\,\what w(k, t) + e^{-ikL}\, \what w(-k, t)\right] dk
= 
-\lim_{R\to\infty} I_R(x, t)
$$
where
$$
I_R(x, t)
:=
\int_{C_R(0)}   
\frac{e^{ikx}}{e^{ikL} + e^{-ikL}} \left[e^{ikL}\,\what w(k, t) + e^{-ikL}\, \what w(-k, t)\right] dk.
$$
It thus suffices to show that $\displaystyle \lim_{R\to\infty} I_R(x, t) = 0$.

%
For $k \in C_R(0)$, since $0 < \sin\theta_1 \leq \sin\theta$, we have
\begin{align*}
&\left|e^{ikL}+e^{-ikL}\right|
\geq
\left|\left|e^{ikL}\right| - \left|e^{-ikL}\right|\right|
=
e^{L R \sin\theta} - e^{-L R \sin\theta}
\geq
e^{L R \sin\theta_1} - e^{-L R \sin\theta_1} > 0,
\\
&\left|e^{2ikL}+1\right|
\geq
\left|\left|e^{2ikL}\right| - 1\right|
=
1 - e^{-2L R \sin\theta}
\geq
1 - e^{-2L R \sin\theta_1} > 0.
\end{align*}
Therefore, using also the triangle inequality, 
\begin{align*}
\left|I_R(x, t)\right|
&\leq
\frac{1}{e^{L R \sin\theta_1} - e^{-L R \sin\theta_1}}
\int_{\theta_1}^{\theta_2}   
e^{-(x+L) R\sin\theta}
\left|\what w(Re^{i\theta}, t)\right| R d\theta
\nn\\
&\quad
+
\frac{1}{1 - e^{-2L R \sin\theta_1}}
\int_{\theta_1}^{\theta_2}   
e^{-x R\sin\theta}
\left|\what w(-Re^{i\theta}, t)\right| R d\theta.
\end{align*}

Furthermore, integrating by parts in the finite interval Fourier transforms $\what w(\pm k, t)$, we get the bounds
%
%
\begin{equation*}
\begin{aligned}
&\left| \what w(k, t) \right|
\leq
\frac{e^{\text{Im}(k) L}}{|k|}
\left(
\left| g_0(t) \right| + \left| h_0(t) \right| + \left\| w_x(t) \right\|_{L^1(0, L)}
\right),
\\
&\left| \what w(-k, t) \right|
\leq
\frac{1}{|k|}
\left(
\left| g_0(t) \right| + \left| h_0(t) \right| + \left\| w_x(t) \right\|_{L^1(0, L)}
\right),
\end{aligned}
\quad
\text{Im}(k)\geq 0,
\end{equation*}
and so
$$
\left|I_R(x, t)\right|
\leq
\left(
\left| g_0(t) \right| + \left| h_0(t) \right| + \left\| w_x(t) \right\|_{L^1(0, L)}
\right)
\left(
\frac{1}{e^{L R \sin\theta_1} - e^{-L R \sin\theta_1}}
+
\frac{1}{1 - e^{-2L R \sin\theta_1}}
\right)
\int_{\theta_1}^{\theta_2}   
e^{-x R\sin\theta} d\theta.
$$

Finally, observing that
$$
\int_{\theta_1}^{\theta_2}   
e^{-x R\sin\theta} d\theta
=
\int_{\theta_1}^{\frac \pi 2}   
e^{-x R\sin\theta} d\theta
+
\int_{\pi-\theta_2}^{\frac \pi 2} 
e^{-x R\sin\theta} d\theta
$$
and using the convexity inequality $\sin\theta \geq \frac{2}{\pi} \, \theta$, $0\leq \theta \leq \frac \pi 2$, which applies since $\theta_1 \in (0, \frac \pi 2)$ and $\theta_2 \in (\frac \pi 2, \pi)$, we find
$$
\int_{\theta_1}^{\theta_2}   
e^{-x R\sin\theta} d\theta
\leq
\int_{\theta_1}^{\frac \pi 2}   
e^{-\frac{2xR}{\pi} \, \theta} d\theta
+
\int_{\pi-\theta_2}^{\frac \pi 2} 
e^{-\frac{2xR}{\pi} \, \theta} d\theta
=
\frac{\pi}{2xR} \left[e^{-\frac{2xR}{\pi} \, \theta_1}-e^{-xR} + e^{-\frac{2xR}{\pi} (\pi-\theta_2)}-e^{-xR}\right].
$$
Since the right side tends to zero as $R\to\infty$, we deduce $\displaystyle \lim_{R\to\infty} I_R(x, t) = 0$, completing the proof of Lemma~\ref{uh-l}.
\end{proof}

Combining \eqref{lcgl-ir-def} with the two expressions in \eqref{bv-sol-s} and Lemma \ref{uh-l}, we obtain the following explicit solution formula for the modified forced linear  problem \eqref{lcgl-ibvp-w}:
\begin{align}\label{lcgl-sol-t}
w(x, t) 
&=
\frac{1}{2\pi} \int_{\mathbb R} e^{ikx-\omega t} \, \what u_0(k) \, dk 
+ \frac{1}{2\pi} \int_{\mathbb R} e^{ikx-\omega t} 
\int_0^t e^{\omega  t'} \what f(k, t') dt' dk
\nn\\
&\quad
- \frac{1}{2\pi} 
\int_{\p D^+}  
\frac{e^{ikx-\omega t}}{e^{ikL} + e^{-ikL}} 
\left\{
 \left[ e^{ikL} \what u_0(k) + e^{-ikL} \what u_0(-k) \right]
+
 \int_0^t e^{\omega t'} \left[e^{ikL} \what f(k, t') + e^{-ikL} \what f(-k, t') \right] dt'  
\right\} dk 
\nn\\
&\quad
- \frac{\nu+i\alpha}{\pi} 
\int_{\p D^+}  
\frac{e^{ikx-\omega t}}{e^{ikL} + e^{-ikL}} 
\left[
\widetilde h_1(\omega, t)
+
 ik e^{-ikL} \, \widetilde g_0(\omega, t)
\right] dk 
\\
&\quad
+ \frac{1}{2\pi} 
\int_{\p D^-}    
\frac{e^{ik(x-L)-\omega t}}{e^{ikL} + e^{-ikL}} 
\left\{
\big[ \what u_0(k) - \what u_0(-k) \big]
+
\int_0^t e^{\omega  t'} \left[\what f(k, t') - \what f(-k, t') \right] dt'  
\right\} dk
\nn\\
&\quad
- \frac{\nu+i\alpha}{\pi} 
\int_{\p D^-}   
\frac{e^{ik(x-L)-\omega t}}{e^{ikL} + e^{-ikL}} 
\left[
e^{ikL} \, \widetilde h_1(\omega, t)
+
ik \widetilde g_0(\omega, t) 
\right] dk.
 \nn
\end{align}
Finally, thanks to analyticity and exponential decay, arguing as in the proof of Lemma \ref{uh-l} we can replace the second argument in the transforms $\widetilde g_0(k^2, t)$ and $\widetilde h_1(k^2, t)$  by any fixed $T\geq t$ in order to express the solution formula~\eqref{lcgl-sol-t} in the form \eqref{lcgl-sol-T}, which is the one suitable for proving the well-posedness result of Theorem~\ref{lwp-t}.

\section{The kernel PDE model}
\label{app-kernel}
We differentiate \eqref{backstepping-trans} with respect to $t$, use the main equation in \eqref{pde_tarlin}
\begin{equation} \label{kert}
	\begin{split}
		u_t(x,t) &= w_t(x,t) + \int_0^x k(x,y) (P_N w)_t (y,t) dy \\
		&= w_t(x,t) + \int_0^x k(x,y) (P_N w_t) (y,t) dy \\
		&= w_t(x,t) + \int_0^x k(x,y) \left((\nu + i\alpha) P_N w_{yy} + \gamma P_N w - \mu P_N w\right)(y,t) dy.
	\end{split}
\end{equation}
$P_N$ and $\partial_y^2$ commutes. Indeed, $w(0,t) = w_x(L,t) = 0$ and $e_j(0) = e_j^\prime(L) = 0$, and
\begin{equation}
	\begin{split}
		P_N w_{yy}(y,t) &= \sum_{j=1}^N e_j(y) \int_0^L e_j(s) w_{ss}(s,t) ds \\
		&= \sum_{j=1}^N e_j(y) \left[e_j(x)w_x(x,t) - \frac{d}{dx} e_j(x) w(x,t)\Bigg |_{x=0}^{x = L} + \int_0^L \frac{d^2}{ds^2} e_j(s) w(s,t)ds \right] \\
		&= \sum_{j=1}^N  e_j(y) \int_0^L (-\lambda_j) e_j(s) w(s,t) ds \\
		&= \sum_{j=1}^N (-\lambda_j) e_j(y) \int_0^L e_j(s) w(s,t) ds \\
		&= \sum_{j=1}^N \frac{d^2}{dy^2} e_j(y) \int_0^L e_j(s) w(s,t) ds \\
		&= \frac{\partial^2}{\partial y^2}\left(\sum_{j=1}^N e_j(y) \int_0^L e_j(s) w(s,t) ds\right) \\
		&= \left(P_Nw(y,t)\right)_{yy}.
	\end{split}
\end{equation}
So following from \eqref{kert} and integrating by parts
\begin{equation} \label{ker1}
	\begin{split}
		u_t(x,t) =& w_t(x,t) + \int_0^x k(x,y) \left((\nu + i\alpha) \left(P_N w\right)_{yy} + (\gamma - \mu) P_N w\right)(y,t) dy \\
		=& w_t(x,t) + (\nu + i\alpha)\left(\left(k(x,y) (P_Nw(y,t))_y - k_y(x,y)P_Nw(y,t) \right)\right)\Bigg|_{y=0}^{y=x} \\
		&+ \int_0^x \left((\nu+ i\alpha) k_{yy} + (\gamma - \mu)k\right)(x,y) P_Nw(y,t) dy \\
		=& w_t(x,t) + (\nu + i \alpha) \left(k(x,x)(P_Nw(x,t))_x - k(x,0) (P_Nw(y,t))_y \Big|_{y=0} - k_y(x,x) P_Nw(x,t)\right) \\
		&+ \int_0^x \left((\nu + i\alpha) k_{yy} + (\gamma - \mu)k\right)(x,y) P_Nw(y,t) dy.
	\end{split}
\end{equation}
Next, we differentiate \eqref{backstepping-trans} with respect to $x$ two times to get
\begin{equation} \label{ker2}
	\begin{split}
		-(\nu + i\alpha) u_{xx}(x,t) =&  -(\nu + i\alpha) w_{xx}(x,t) - (\nu + i\alpha) \int_0^x k_{xx}(x,y)P_Nw(y,t) dy \\
		&- (\nu + i \alpha) k_x(x,y)P_Nw(x,t) -(\nu + i \alpha)\frac{d}{dx}k(x,x)P_Nw(x,t) \\
		&- (\nu + i\alpha) k(x,x) (P_Nw(x,t))_x.
	\end{split}
\end{equation}
Adding \eqref{ker1}-\eqref{ker2} together with
\begin{equation} \label{ker3}
	-\gamma u(x,t) = -\gamma w(x,t) -\gamma \int_0^x k(x,y) P_N w(y,t) dy
\end{equation}
side by side, we get
\begin{equation} \label{ker_add}
	\begin{split}
		0 =& -\mu P_Nw(x,t) - \int_0^x \left((\nu + i \alpha)(k_{xx} - k_{yy}) + \mu k\right)(x,y)P_Nw(y,t) dy \\
		&- (\nu + i \alpha)\left(k_x(x,x) + k_y(x,x) + \frac{d}{dx}k(x,x)\right) P_Nw(x,t) - \nu k(x,0) (P_Nw(y,t))_y\Big|_{y=0}.
	\end{split}
\end{equation}
Suppose $k(x,0)  = 0$. Since $\frac{d}{dx} k(x,x) = k_x(x,x) + k_y(x,x)$ and by integrating over $(0,x)$, we get 
\begin{equation}
	\begin{split}
		\left(2(\nu + i \alpha)\frac{d}{dx}k(x,x) + \mu\right) P_Nw(x,t) = 0 &\Rightarrow \frac{d}{dx}k(x,x) = - \frac{\mu}{2(\nu + i \alpha)} \\
		&\Rightarrow k(x,x) = -\frac{\mu x}{2(\nu + i \alpha)}.
	\end{split}
\end{equation}
Consequently
\begin{eqnarray}
	\begin{cases}
		(\nu + i \alpha) (k_{xx} - k_{yy}) + \mu k = 0, \\
		k(x,0) = 0, \quad k(x,x) = -\dfrac{\mu x}{2(\nu + i \alpha)},
	\end{cases}
    \quad (x,y) \in \mathcal T,
\end{eqnarray}
where $\mathcal T := \left\{(x,y) \in \mathbb{R}^2 : 0 \leq x \leq L, 0 \leq y \leq x\right\}$.

\bibliographystyle{myamsalpha}
\bibliography{references}

\end{document}